\documentclass[10pt]{amsart}

\usepackage{amsmath, amsthm, amssymb}
\usepackage{tikz}
\usepackage{mathabx}
\usepackage{pinlabel}
\usepackage[T1]{fontenc}            % font encoding 
\usepackage{palatino}               % text font
\usepackage{enumerate}
\usepackage{comment}

\usepackage[abs]{overpic}		  % overpic automatically loads graphicx

\usepackage{xcolor}
\definecolor{indigo}{rgb}{0.29, 0.0, 0.51}  % custom colors
\usepackage[colorlinks, urlcolor=indigo, linkcolor=indigo, citecolor=indigo]{hyperref}

\usepackage[hcentering, top = 1.5in, total={6in, 8.3in}]{geometry}  % 1.35 vertical margin; 1.25 horizontal margin
\usepackage[english]{babel}
\theoremstyle{plain}
\newtheorem{theorem}{Theorem}

\newtheorem{lemma}[theorem]{Lemma}

\theoremstyle{definition}

\theoremstyle{remark}
\newtheorem{remark}[theorem]{Remark}

\numberwithin{theorem}{section}

\makeatletter
\newcommand*\bigcdot{\mathpalette\bigcdot@{0.6}}
\newcommand*\bigcdot@[2]{\mathbin{\vcenter{\hbox{\scalebox{#2}{$\m@th#1\bullet$}}}}}
\makeatother

\DeclareMathOperator\tb{tb}  % Thurston-Bennequin
\DeclareMathOperator\tbr{\tb_{\mathbb{Q}}}
\DeclareMathOperator\rot{rot}  % rotation
\DeclareMathOperator{\rotr}{\rot_{\mathbb{Q}}}
\begin{document}

% title
\title{Non-loose Legendrian Hopf links in lens spaces}

% author information
\author{Rima Chatterjee}

%\author{John B. Etnyre}

%\address{School of Mathematics \\ Georgia Institute of Technology \\  Atlanta, GA}
%\email{etnyre@math.gatech.edu}

\address{Department of Mathematics \\ The Ohio State University,  Columbus, OH-43210}
\email{rchattmath@gmail.com, chatterjee.198@osu.edu}

%\subjclass[2020]{57R17}

% abstract
\begin{abstract}
 We give a complete classification of non-loose Legendrian Hopf links in $L(p,q)$ generalizing a result of the author with Geiges and Onaran \cite{chatterjee2025_Hopf}. The classification is for non-loose Hopf links for both zero and non-zero Giroux torsion in their complement. We also give an explicit algorithm for the contact surgery diagrams for all these Legendrian representatives with no Giroux torsion in their complement.
\end{abstract}

\maketitle

\section{Introduction}
\label{sec:Intro}
In recent years, there has been extensive studies on Legendrian and transverse knots in overtwisted contact manifolds with tight complements, known as non-loose knots (also known as exceptional ). Most of these results are about partial and complete classification of certain knot types \cite{EF, Geiges_exceptionaltorus, Matkovic, etnyre2022nonloose}. The author along with Etnyre, Min, and Mukherjee also proved the existence of these knots in general contact $3$ manifolds and studied their behavior under cabling \cite{Chatterjee-Etnyre-Min_2025_existence}.  These knots are extremely important in understanding certain questions in contact topology, for example they can be used to construct and classify tight contact structures on mani\-folds obtained by surgeries. Although there has been some progress in understanding non-loose knots, results concerning non-loose links remain scarce due to the lack of techniques and tools.  The only link type that has been studied is the Hopf link. Geiges--Onaran gave the first classification of non-loose Hopf links in $S^3$ in all contact structures \cite{Geiges_Onaran} and later in \cite{chatterjee2025_Hopf}, the author together with Geiges and Onaran proved a classification result for Hopf links in $L(p,1)$ with any contact structure. Note that, this is also the first Legendrian classification of a link type in a mani\-fold other than $S^3.$ Here by Hopf links we mean the core of the two Heegaard tori of $L(p,1)$. In \cite{chatterjee2025_Hopf}, we used the classification of tight contact structures on $T^2\times I$ \cite{Giroux, Honda_tight} to give an upper bound on the number of non-loose representatives and then explicitly found the contact surgery diagrams for those representatives.  It turns out that the same technique becomes lot complicated when $q>1$. In this paper, we used convex surface theory to extend the classification result in every $L(p,q)$. This method also allowed us to fully understand how the link components are related via stabilizations thus completing the classification. Though we did not use the contact surgery diagrams in our proof, we gave an algorithm on how to extract the contact surgery diagrams of the Legendrian representatives of the Hopf links. This algorithm shows a beautiful connection of the geometry of the Farey graph and contact surgeries.

While in \cite{chatterjee2025_Hopf} we only proved a classification for Hopf links having zero Giroux torsion in their complement (these are also known as strongly exceptional), in this paper we expand the classification to include Hopf links with Giroux torsion as well. All the classification mentioned in this paper are up to coarse equivalence.

\begin{figure}[!htbp]
     \labellist
     \small\hair 2pt
    \pinlabel \textcolor{blue}{$K_1$} at 40 100
    \pinlabel $a_n$ at 120 130
    \pinlabel $a_{n-1}$ at 190 130
    \pinlabel $\dots$ at 275 60
    \pinlabel $a_1$ at 375 130
    \pinlabel $a_0$ at 425 130
    \pinlabel $\textcolor{red}{K_2}$  at 500 100
    \endlabellist
   
    \centering
    \includegraphics[width=0.7\linewidth]{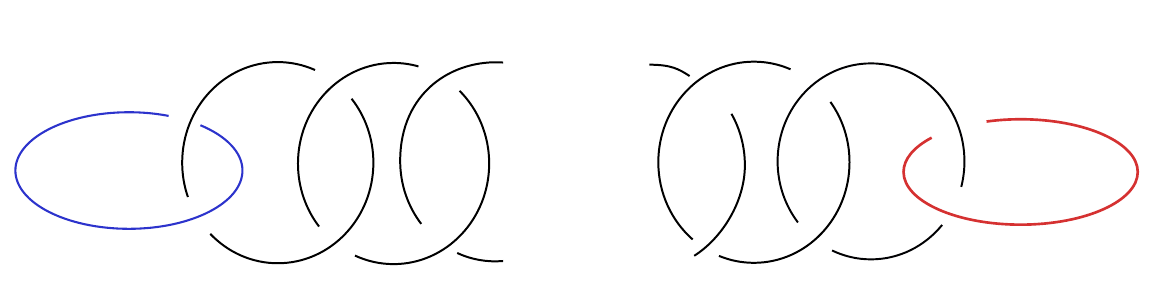}
    \caption{Hopf link in $L(p,q)$ where $-p/q=[a_0,a_1\cdots, a_n]$}
    \label{fig:Hopf}
\end{figure}
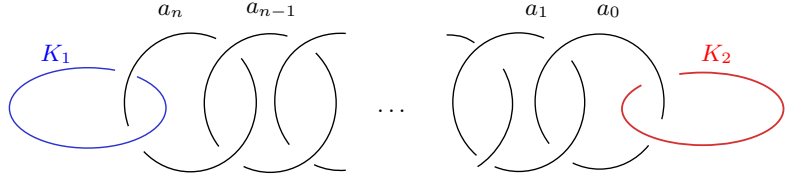

By Hopf link we mean the core of the two Heegaaard tori of $L(p,q)$. Figure \ref{fig:Hopf} shows the smooth Hopf link in $L(p,q)$. In this paper by Hopf link we also refer to the positive Hopf link (in the sense of \cite{chatterjee2025_Hopf}) unless otherwise stated. A classification of negative Hopf links will follow similarly after changing the orientation of one of the components. First we recall, the smooth classification of  rational unknots in lens spaces. Let $K_i$ denote the cores of the two Heegaard tori of $L(p,q).$

\[\text{rational unknots in }L(p,q)=\left\{ \begin{array}{cc}
     K_1 & p=2 \\
      K_1,-K_1  &p\neq 2, \ q\equiv\pm 1\pmod p\\
      K_1,-K_1, K_2, -K_2 & \text{otherwise}
     \end{array}\right\}\]

     We begin by mentioning some basic notions of classification of Legendrian knots in a contact mani\-fold. Given a knot type $K$ in an overtwisted manifold $(M,\xi)$, we denote the coarse equivalence class of its Legendrian representatives by $\mathcal{L}(K)$. Consider the map $\Phi\colon\mathcal{L}(K)\rightarrow \mathbb{Z}^2$ that sends a null-homologous Legendrian knot $L$ to $(\rot(L), \tb(L))$ where $\tb$ and $\rot$ denote the Thurston--Bennequin invariant and rotation number of the Legendrian knot. The image of $\phi$ along with the number of elements that has been sent to a single point is known as the {\it mountain range} of $K$. If $K$ is rationally null-homologous then we replace $\tb$ and $\rot$ by $\tbr$ and $\rotr$, the rational Thurston--Bennequin invariant and the rational rotation number and consider a similar map $\Phi\colon\mathcal{L}(K)\rightarrow\mathbb{Q}^2$ that sends $L$ to $(\rotr(L), \tbr(L)).$

     \begin{figure}[!htbp]
\centering
\labellist
\pinlabel $(a,b)$ at 100 30
\pinlabel $(a,b)$ at 170 30
\pinlabel $(a,b)$ at 420 30 
\endlabellist
    \includegraphics[width=0.7\linewidth]{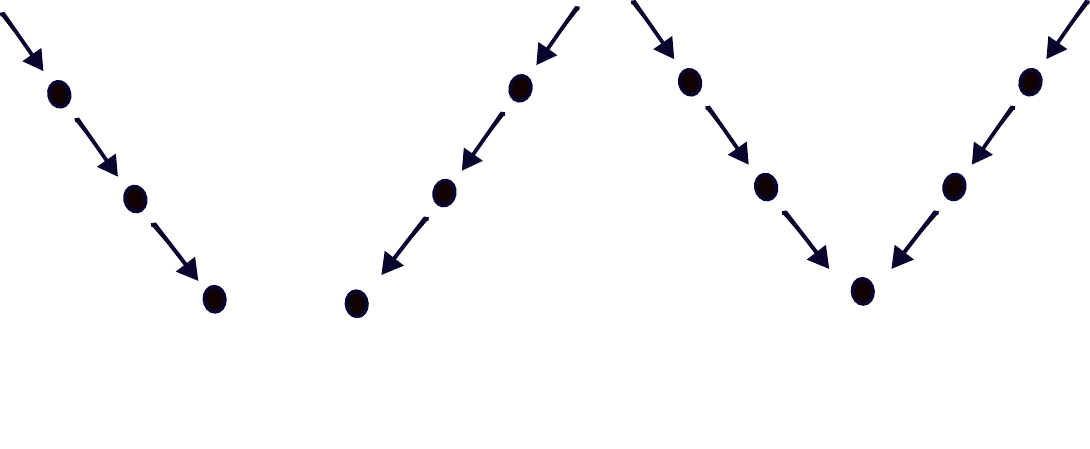}
  \caption{(a) A backward slash and a forward slash based at $(a,b)$. (b) On the right we see a $V$ based at $(a,b)$. }
    \label{fig:V}
\end{figure}
     
     We say that a mountain range for a knot contains a {\it non-loose} $V$ based at $(a,b)$ if the image of $\Phi$ contains non-loose Legendrian knots $L^0, L^i_\pm$ for $i\in\mathbb{N}$ such that \[\tb(L_\pm^i)=b+i\ \ \text{and}\ \ \rot(L_\pm^i)=a\pm i,\] 
     \[\tb(L^0)=b\ \ \text{and}\ \ \rot(L^0)=a\]
     and satisfy 
     \[S_\pm(L_\pm^{i+1})=L_\pm^i,\ \text{and}\ S_\pm(L^1_\pm)=L_0,\]
     \[S_\mp(L^i)\ \text{and}\ S_\pm(L^0)\ \text{are loose}.\]

     Here $S_\pm (L)$ denotes the positive (respectively negative) stabilization of $L$.

     One the other hand we call a $V$ a {\it loose} $V$ if the image of $\Phi$ contains loose Legendrian knots $L^0, L^i_\pm$ for $i\in\mathbb{N}$ which when included in a link gives a non-loose link, satisfies exact same relations are before, except that  \[S_\mp(L^i)\ \text{and}\ S_\pm(L^0)\ \text{will no longer be a part of the non-loose link}.\]

     See the right side of Figure~\ref{fig:V}. We say the mountain range of $L$ contains a {\it non-loose back slash based at $(a,b)$} if the image of $\Phi$ contains non-loose Legendrian knots $L^i$ for $i\in\mathbb{Z}_{\geq 0}$ such that 
     \[\tb(L^i)=b+i\ \text{and}\ \rot(L^i)=a-i\] and satisfy
     \[S_+(L^{i+1})=L^i,\] and \[S_-(L^i)\ \text{and}\ S_+(L^0)\ \text{are loose}.\]

     We say the it's a {\it loose back slash} if all the vertices are loose but could be included as a component of a non-loose link, relations are same as before and \[S_-(L^i)\ \text{and}\ S_+(L^0)\ \text{are not part of the non-loose link.}\]

     Similarly we say a mountain range for $K$ contains a {\it non-loose forward slash  based at $(a,b)$} if the image of $\Phi$ contains non-loose Legendrian knots $L^i$ for $i\in\mathbb{Z}_{\geq 0}$ such that 
     \[\tb(L^i)=b+i\ \text{and}\ \rot(L^i)=a+i\] and satisfy
     \[S_-(L^{i+1})=L^i,\] and \[S_+(L^i)\ \text{and}\ S_-(L^0)\ \text{are loose}.\]
Check the middle picture of Figure~\ref{fig:V}.
     We can define a {\it loose forward slash } exactly as we defined the loose back slash. 

%\textcolor{red}{When we mention a link complement we mean the complement of the standard neighborhood of the Legendrian representatives.}

Next we will define a {\it loose cone} of a Legendrian knot. A {\it cone C} of a Legendrian knot $L$ is defined as follows:
\[C(L)=\{S_+^i S_-^{j}(L)\ \text{where}\  i,j\geq 0\}\ \] that is the set which contains all possible stabilizations of the Legendrian $L$. 

We call $C$ a loose cone if every vertex in the cone is a loose knot but can be included in a non-loose link. We call a cone {\it double peaked} if there are two Legendrian knots $L_+$ and $L_-$ at the peak such that $\tb(L+)=\tb(L_-), \rot(L_+)=\rot(L_-)+2$ and $S_+(L_-)=S_-(L_+)$. See Figure~\ref{fig:loosecone}.

%In \cite{CGO}, we used Honda and Giroux's classification to give an upper bound and the find explicit contact surgery diagrams that realizes those bounds. In this paper, we use the geometry of the farey graph and techniques from convex surface theory. We also show how  the non-loose hopf links are related via stabilizations and finally give explicit contact surgery diagrams for all the non-loose representatives. In \cite{CGO}, we didn't classify Hopf links with non-zero Girouzx torsion. In this paper, we also classify non-loose Hopf links with non-zero Giroux torsion. This completes the classification.

%\begin{figure}
   % \centering
   % \includegraphics[width=0.5\linewidth]{non-loose_mountainL(p,1).pdf}
   % \caption{(a) A non-loose $V$ based at $(1+\frac{1}{p},\frac{-p+2k}{p})$ (There are total $p$ of them) and a non-loose dot based at $(1+\frac{1}{p},1)$. and another one at . Each of the blue and green dot pairs with a loose black dot from Figure \ref{fig:loosemountain_negative} to give a non-loose Hopf link. (b) The red dot represents a non-loose unknot with invariants $(\frac{1}{p}, 0)$ that pairs with itself to give a non-loose Hopf link with both components being non-loose.}
   % \label{fig:non-loosemountainL(p,1)}
%\end{figure}

\begin{figure}
    \centering
   
    \includegraphics[width=0.3\linewidth]{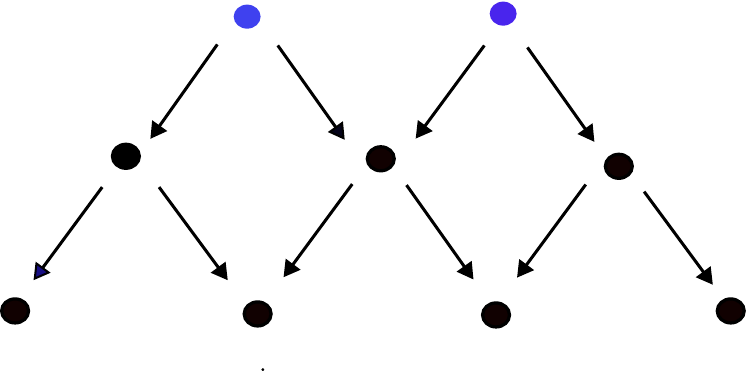}
    \caption{Part of a loose cone with two peaks.}
    \label{fig:loosecone}
\end{figure}

%\begin{figure}
    
    %\includegraphics[width=0.7\linewidth]{loosemountain1.pdf}
    %\caption{(a) For $L(p,1)$ on the left 2 forward slashes based at $(2+\frac{1}{p},\pm(1+\frac{-p+2m}{p} )$ respectively. Each of the red dots pairs with a green dot to give a non-loose Hopf link. (b) A loose $V$ based at $(1+\frac{1}{p}, \frac{-(p+2)+2k}{p})$. (c) A single dot based at $(1+\frac{1}{p},\frac{p+2}{p})$. Each of the blue dot pairs with itself to give a non-loose Hopf link. In (b) each of the red dot pairs with the blue dot to give a non-loose Hopf link. }
    %\label{fig:loosemoutain_L(p,1)_2}
%\end{figure}

We start by stating our result for $L(p,1)$. Note that, in $L(p,1)$ there are only two rational unknots $K_1$ and $-K_1$. Thus we will have two rational positive Hopf links $K_1\sqcup -K_1$ and $-K_1\sqcup K_1$. We will state our result for $K_1\sqcup -K_1$. The classification for the other one can be obtained by interchanging the components. %changing the rotation numbers of both the components simultaneously and switching the forward and the backward slashes.

\begin{theorem}
\label{thm:L(p,1)} Suppose $L_1\sqcup L_2$ denote the positive Hopf-link in $L(p,1)$. The non-loose realizations of $L_1\sqcup L_2$ (up to switching the components) are as follows: 
\begin{enumerate}

      \item (loose--non-loose pairing): In this case the Thurston-- Bennequin invariants for the pairs are given by $\tbr(L_1)=-(k_1+1)+\frac{1}{p}$ and $\tbr(L_2)=k_2+\frac{1}{p}.$ Fixing $L_1$, the non-loose realizations of $L_2$ forms one non-loose $V$ based at $(0,\frac{1}{p})$ and $p-1$ non-loose $V$s based  at $(-1+\frac{2m}{p}, 1+\frac{1}{p})$ for $m=1,2,\cdots, p-1$. The based vertex of the first $V$ (corresponding to $k_2=0$) pairs with each loose $L_1$ that forms a loose cone with two peaks at $(-1, -1+\frac{1}{p})$ and $(1, -1+\frac{1}{p})$  to give $|k_1-2|$ non-loose realizations. On the other hand, each of the vertex of the non-loose $V$s with $k_2\geq 1$ pairs with a loose cone of $L_1$ based at $(-1+\frac{2m}{p},-1+\frac{1}{p})$ to give non-loose realizations.  The first non-loose $V$ lives in Euler class $0$ and the other $p-1$ non-loose $V$s live in Euler class $-p+2m$ for $m=1,2,\cdots, p-1.$
    
      \item (non-loose and non-loose pairing) There exists a unique non-loose Hopf link with both components being non-loose. The classical invariants of the components are given by $(0,\frac{1}{p})$ and $(0,\frac{1}{p}).$ The contact structure has Euler class $0$. Fixing $L_2$, the positive stabilization of $L_1$ is the peak of the loose cone of $L_1$ from (1).

      \item (loose and loose pairing) The Thurston--Bennequin invariants in this case are given by $\tbr(L_i)=k_i+\frac{1}{p}$. The rotation numbers of the individual components are given in Table~\ref{tab:rotationL(p,1)}. The all possible non-loose realizations are: Fixing the loose unknot $L_2$ with
      \begin{enumerate}
          \item $k_2=0,$ the (loose) mountain range for $L_1$ is given by: $1$ loose back slash and $1$ loose forward slash based at $(-\frac{p+2}{p}, 1+\frac{1}{p})$ and $(\frac{p+2}{p}, 1+\frac{1}{p}$) respectively and $1$ loose $V$ based at $(0,2+\frac{1}{p})$. The forward and back slashes live in Euler class $2$ and the $V$ lives in Euler class $0.$ 
          
          \item $k_2=1$, the (loose) mountain range for $L_1$ is given by: $1$ loose back and $1$ loose forward slash based at $(-\frac{2}{p}, \frac{1}{p})$ and $(\frac{2}{p},\frac{1}{p})$ and $p+1$ loose $V$s based at $(-\frac{p+2-2m}{p}, 1+\frac{1}{p})$ for $m=1,2,\cdots, p+1$. The forward and backward slash live in Euler class $\pm 2$ and the $V$s live in Euler class $-(p+2-2m)$ for $m=1,2,\cdots p+1$.
          
          \item $k_2=2,$ the (loose) mountain range for $L_1$ is given by: $1$ loose back slash and $1$ loose forward slash based at $(-\frac{2}{p}, \frac{1}{p})$ and $(\frac{2}{p}, \frac{1}{p}$) respectively and $1$ loose $V$ based at $(0,\frac{1}{p})$. Moreover, there are $2p$ loose $V$s based at $(-\frac{p+1-2m}{p}\mp\frac{1}{p}, 1+\frac{1}{p})$ for $m=1,2,\cdots, p$. The forward and back slashes live in Euler class $\pm 2$ and the first $V$ lives in Euler class $0.$ The later $2p$ $V$s live in Euler class $-(p+1-2m)\mp 1.$
          
          \item $k_2>2$, the (loose) mountain range for $L_1$ is given by: $2$ loose back slashes based at $(-\frac{2}{p},\frac{1}{p}), (0,\frac{1}{p})$, $2$ loose forward slashes based at $(\frac{2}{p},\frac{1}{p}), (0,\frac{1}{p})$ and $2p$ loose $V$s based at $(-\frac{p+1-2m\mp 1}{p}, 1+\frac{1}{p})$ where $m=1, 2, \cdots p$. The 2 back slashes live in Euler class $-2$ and $0$, the two forward slashes live in Euler class $2$ and $0$ and the $2p$ $V$s live in Euler class $-(p+1-2m)\mp 1$ for $m=1,2,\cdots p.$
      \end{enumerate}

      Fixing loose $L_1$, we will see an exactly same mountain range for $L_2$ as well. %\textcolor{red}{One needs to be careful about the pairing though. How to write that?}

\end{enumerate}
	
\end{theorem}
\begin{remark}
    In the above statement, some of the candidates will have same Euler classes. But those have different $d_3$ invariants as shown in \cite{chatterjee2025_Hopf}. So they live in distinct overtwisted contact structures. Also, pairwise they have distinct rotation numbers.
\end{remark}
\begin{remark}
     The above theorem recovers the result of \cite{chatterjee2025_Hopf}. Moreover, it also tells us how the individual components of the non-loose representatives are related via stabilizations. Note that, all of the candidates from above have zero Giroux torsion in their complement. 
\end{remark}

%\begin{theorem}
 %   \label{thm:L(p,p-1)}
%\end{theorem}
   Next we give a complete classification of non-loose Hopf links in $L(2n+1,2)$. Here we need to consider both unknots $K_1$ and $K_2$, and $K_1\sqcup -K_2$ will give us a positive Hopf link. For a negative Hopf link, one needs to switch the orientation of $K_2$ and thus the rotation numbers $L_2$ ($L_i$ denote the Legendrian representatives of $K_i$). In the mountain range for $L_2$ one needs to switch the forward and the back slashes as well. In the following theorem $L_2$ is oriented opposite of $L_1$.

\begin{theorem}
\label{thm:L(2n+1,2)}
The non-loose realizations of $L_1\sqcup L_2$ in $L(2n+1,2)$ with $\tbr(L_1)=k_1+\frac{n+1}{2n+1}$ and $\tbr(L_2)=k_2+\frac{2}{2n+1}$ are as follows:
\begin{enumerate}
    \item (loose/non-loose pairing) For $k_1<0, k_2\geq 0$: Fixing loose $L_1$, the non-loose realizations of $L_2$ form one forward slash based at $(\frac{1}{2n+1}, \frac{2}{2n+1})$, one back slash based at $(-\frac{1}{2n+1}, \frac{2}{2n+1})$ and $n$ $V$s based at $(\frac{2(-n-1+2m)}{2n+1}, 1+\frac{2}{2n+1})$ for $m=1,\cdots, n$. The forward and back slashes live in Euler class $\pm (n+1)$ and the $V$s live in Euler class $(-n-1+2m)$ for $m=1,2,\cdots n$. Each of these non-loose candidates pairs with a loose $L_1$ cone based at $(r,-1+\frac{n+1}{2n+1})$ to give non-loose Hopf links. More detail about this pairing and $r$ will be given in Section~\ref{sec:examples}.
    
%    For $k_2=0$, the correspondence vertex ($L_2$) from forward and back slashes pairs with a loose cone ($L_1$) peaked at $(\pm\frac{n+1}{2n+1}, -1+\frac{n+1}{2n+1})$ to give non-loose realizations of Hopf links. For $k_1=1$, the corresponding vertices from the forward and back slashes pairs with the same loose cones as before % based at $(\pm\frac{n+2}{2n+1},-1+\frac{n+1}{2n+1})$ 
    %and each based vertex of the $V$'s pairs with loose cone peaked at $(\frac{2(-n-1+2m)}{2n+1}, -1+\frac{n+1}{2n+1})$ for $m=1,2,\cdots, n$. Finally for $k_2>1$, the corresponding vertex from the back and forward slashes pair the the same loose cones, the $n$ vertices from the right wings of the $n$ $V$s pair up with loose cones based at $(\frac{-n+2m}{2n+1}-\frac{1}{2n+1}, -1+\frac{n+1}{2n+1})$, and the $n$ vertices from the left wings of the $n$ $V$s pair up with the loose cones based at $(\frac{-n+2m}{2n+1}+\frac{1}{2n+1}, -1+\frac{n+1}{2n+1})$ for $m=1,2,\cdots, n.$

    %each vertex from the $n+1$ right wings with $k_2>1$ pairs with  a loose cone peaked at $(\frac{-n-2+2m}{2n+1}, -1+\frac{n+1}{2n+1})$ and each vertex with $k_2>1$ from the $n+1$ left wings pairs with a loose cone peaked at  $(\frac{-n+2m}{2n+1}, -1+\frac{n+1}{2n+1})$ to give non-loose Hopf links. These are all non-loose realizations with $K_1$ loose and $K_2$ non-loose. 

    \item (non-loose/loose pairing) For $k_1\geq 0, k_2<0$: Fixing loose $L_2$, the non-loose realizations of $L_1$ form $1$ forward slash based at $(\frac{n}{2n+1},\frac{n+1}{2n+1})$, one back slash based at $(-\frac{n}{2n+1},\frac{n+1}{2n+1})$, $n-1$ $V$s based at $(\frac{(-n+2m_1)}{{2n+1}}, \frac{n+1}{2n+1})$ where $m_1=1,\cdots n-1$ and $n$ $V$s based at $(\frac{(-n+1+2m_2)}{2n+1}, 1+\frac{n+1}{2n+1})$ for $m_2=0,1,\cdots n-1.$ The forward and back slashes live in Euler class $\pm n$, the first $n-1$ $V$s live in Euler class $-n+2m_1$ for $m_1=1,2,\cdots n-1$ and the later $n$\ $V$s live in Euler class $-n+1+2m_2$ for $m_2=0,1,\cdots n-1$.  Each of these non-loose candidates pairs with a loose $L_2$ cone based at $(r,-1+\frac{2}{2n+1})$ to give non-loose Hopf links. More detail on this pairing and $r$ will be given in Section~\ref{sec:examples}.
    
    %For $k_1=0$ each of the based vertex from the forward slash, the $n-1$ $V$ and the back slash pairs with a loose cone peaked at $(\frac{2(-n+2m_1)}{2n+1},-1+\frac{2}{2n+1})$ for $m_1=0,1,\cdots n$ respectively. For $k_1=1$, the corresponding vertex from the forward and back slash pair with loose cones peaked at $(\pm\frac{2n}{2n+1}, -1+\frac{2}{2n+1})$ respectively, the corresponding vertices from the right and left wings of  $n-1$ $V$s pair with the loose cone based at $(\pm\frac{2}{2n+1}+\frac{2(-n+1+2m_2)}{2n+1},-1+\frac{2}{2n+1})$ for $m_2=0,1,\cdots n-1$ respectively and the based vertices of the $n$ $V$s pair up with loose cone peaked at $(\frac{2(-n+1+2m_2)}{2n+1},-1+\frac{2}{2n+1})$ for $m_2=0,1,\cdots, n-1$.  Finally for $k_1>1$, the corresponding vertices from the the forward and back slashes pair up with the loose cones based at $(\pm\frac{2n}{2n+1},-1+\frac{2}{2n+1})$, the $n$ corresponding vertices from the left wing of the $n$ $V$s pair up with a loose cone peaked at $(\frac{2(-n+1+2m_1)}{2n+1},-1+\frac{2}{2n+1})$, same for the $n$ candidates of the right wings of the $V$s. And the $n-1$ vertices from the right and left wing of $n-1$ $V$s pair with loose cones peaked at $(\frac{\pm 2+2(-n+1+2m_1)}{2n+1}, -1+\frac{2}{2n+1}).$

    \item (loose/loose pairing) $k_1\geq 0, k_2\geq0$, both the components are loose. 
    
    \begin{itemize}\item Fixing $L_2$, the loose mountain range for $L_1$ follows: \begin{enumerate}
        \item For $k_2=0$, we have $1$ loose forward and 1 loose back slash based at $(\pm\frac{n+2}{2n+1},\frac{n+1}{2n+1})$ respectively and $2$ loose $V$ 's based at $(\mp\frac{n+1}{2n+1}, 1+\frac{n+1}{2n+1})$. The forward and back slashes live in Euler class $\pm (n+2)$ and the $V$s live in Euler class $\mp (n+1)$.
        \item  For $k_2=1$, the loose mountain range of $L_1$ is given by the same forward and back slash from before, $n+1$ loose $V$ s based at $(\frac{-n-2+2m_1}{2n+1},\frac{n+1}{2n+1})$ and $n+2$ loose $V$s based at $(\frac{-n-1+2m_2}{2n+1},1+\frac{n+1}{2n+1})$ where $m_1=0,\cdots n$ and $m_2=0,1,\cdots n+1$. The $n+1$ $V$s live in Euler class $(-n-2+2m_1)$ and the $n+2$ $V$s from above live in Euler class $(-n-1+2m_2)$.
        
        \item Finally for $k_2>1$, we have $2$ forward and $2$ back slashes based at $(\frac{n+1}{2n+1}\mp\frac{1}{2n+1},\frac{n+1}{2n+1})$ and $(-\frac{n+1}{2n+1}\mp\frac{1}{2n+1},\frac{n+1}{2n+1})$ respectively, $2n$ $V$'s based at $(\frac{-n-1+2m_1}{2n+1}\mp\frac{1}{2n+1}, \frac{n+1}{2n+1})$ and $2(n+1)$ $V$s based at $(\frac{-n+2m_2}{2n+1}\mp\frac{1}{2n+1},1+ \frac{n+1}{2n+1})$ where $m_1=1, 2,\cdots n$ and $m_2=0,1,\cdots n$. The two forward slashes live in Euler classes $n, n+2$, the two back slashes live in Euler class $-n, -n-2$, the $2n$ $V$s live in Euler class $-n-1+2m_1\mp 1$ and the $2(n+1)$ $V$s live in Euler class $-n+2m_2\mp 1.$
        
    \end{enumerate}  
    For fixed $L_2$ with $k_2\geq 0$, the loose mountain range of $L_1$ is given in Figure \ref{fig:looseL(2n+1,2)K2fixed}.

    \item Fixing $L_1$, the loose mountain range for $L_2$ follows:
    \begin{enumerate}
        \item For $k_1=0$, there is $1$ forward and $1$ back slash based at $(\pm\frac{3}{2n+1},\frac{2}{2n+1})$ respectively and $n+1$ $V$'s based at $(\frac{2(-n-2+2m)}{2n+1}, 1+\frac{2}{2n+1})$ where $m=1,2,\cdots n+1$. The forward and backward slashes live in Euler classes $\pm (n+2)$ and the $V$s live in $(-n-2+2m)$ for $m=1,2, \cdots n+1.$
        \item  For $k_1=1$, there is $1$ forward and $1$ backward slash from the previous case, $2$ V's based at $(\pm \frac{1}{2n+1}, \frac{2}{2n+1})$ and $3n$ $V$s based at $(\frac{-2+2m_1}{2n+1}+\frac{2(-n-1+2m_2)}{2n+1})$ where $m_1=0,1,2$ and $m_2=1,\cdots n.$ The two $V$s live in Euler class $\pm(1+n)$ and the $3n$ $V$s live in Euler class $(n+1)(-3+2m_1)+2m_2.$
        \item Finally for $k_1>1$, there are two forward slashes based at $(\frac{3}{2n+1},\frac{2}{2n+1})$ and at $(\frac{1}{2n+1}, \frac{2}{2n+1})$, $2$ back slashes based at $(-\frac{3}{2n+1},\frac{2}{2n+1})$ and at $(-\frac{1}{2n+1}, \frac{2}{2n+1})$ and $4n$ $V$s based at\\ $(\mp\frac{1}{2n+1}\mp\frac{1}{2n+1}+\frac{2(-n-1+2m)}{2n+1}, 1+\frac{2}{2n+1})$ for $m=1,2,\cdots n$. The Euler class of the forward slashes are $n+1$, for the back slashes are $-n-1$, for the $4n$ $V$s are $\pm n\mp (n+1)+(-n-1+2m)$.
    \end{enumerate}
    
     For fixed $L_1$ with $k_1\geq 0$ the loose mountain range of $K_2$ is given in Figure \ref{fig:looseL(2n+1,2)K1fixed}. The rotation numbers of the pairs are given in Table~\ref{tab:L(2n+1,2)}.
    \end{itemize}

\end{enumerate}                                                                        
\end{theorem}
\begin{remark}
    Note that, in case (3) when $L_1$ is fixed at $k_1>2$ among the $4n$ loose $V$s, $2n$ $V$s have same Euler class, but they can be distinguished by their $d_3$ invariant. The $d_3$ invariants of these overtwisted contact structures can be easily computed from the explicit contact surgery diagrams given in Section \ref{sec:contactsurgery}. But as they can be all distinguished by their pairwise distinct rotation numbers we refrained ourselves from the algebraic calculations.
\end{remark}

Next we give the general classification result for the positive Hopf link $L_1\sqcup L_2$ in $L(p,q)$ for $q>1.$ $L_1$ and $L_2$ are oriented opposite as before. We will denote the negative continued fraction 

\[
a_0 - \cfrac{1}{a_1 - \cfrac{1}{a_2 - \cfrac{1}{\ddots - \cfrac{1}{a_n}}}}
\]
for $a_i\leq -2$ by $[a_0,a_1,\cdots , a_n].$

\begin{theorem}
\label{thm:general} Suppose $(p,q)$ is a pair of relatively prime integers with $p>q> 1$. Let $-p/q=[a_0,a_1,\cdots a_n]$ and $n\geq 1$ as $q\neq 1$. In $L(p,q)$, the non-loose Legendrian representatives of $L_1\sqcup L_2$ are as follows:
	\begin{enumerate}
	    \item {(loose-nonloose pairing)} For fixed $L_1$, the non-loose realizations of $L_2$ form

  \[
  \left\{\begin{array}{lr}
       1 , & n=1\\
     |a_2+1||a_3+1|\cdots|a_n+1|  ,  &n\geq 2\\
     
        \end{array}\right\} 
  \]many non-loose forward slashes and the same number of non-loose back slashes based at $(r,\frac{q}{p})$, and
  \[ \left\{\begin{array}{lr}
       |a_1+2| , & n=1\\
     |a_1+2||a_2+1|\cdots|a_n+1| ,  &n\geq 2\\
     
        \end{array}\right\} \] many non-loose $V$'s based at $(r,\frac{q}{p})$. Additionally, there are \[|a_0+1||a_1+1|\cdots|a_n+1|\]
$V$'s based at $(r,1+\frac{q}{p}).$ Each of these pairs with a loose $L_1$ that belongs to a loose cone $C$ based at $(r, -1-\frac{p''}{p})$ to give us the loose and non-loose pairings where $ 1\leq p''\leq p \ \text{and}\\\ p''q\equiv p-1\pmod p$.
        
        \item {(nonloose-loose pairing) } For fixed $L_2$, the non-loose realizations of $L_1$ form

  \[
  \left\{\begin{array}{lr}
       1 , & n=1\\
     |a_0+1||a_1+1|\cdots|a_{n-2}+1|  ,  &n\geq 2\\
     
        \end{array}\right\} 
  \]many non-loose forward slashes and the same number of non-loose back slashes based at $(r,\frac{p'}{p})$, and
  \[ \left\{\begin{array}{lr}
       |a_0+2| , & n=1\\
     |a_0+1||a_1+1|\cdots|a_{n-1}+2| ,  &n\geq 2\\
     
        \end{array}\right\} \] many non-loose $V$'s based at $(r,\frac{p'}{p})$. Additionally, there are \[|a_0+1||a_1+1|\cdots|a_n+1|\]
$V$'s based at $(r,1+\frac{p'}{p}).$ Each of these pairs with each loose $L_1$ that belongs to a loose cone $C$ based at $(r, -1+\frac{q}{p})$ to give us the non-loose and loose pairings where $1\leq p'\leq p$ and $pq'\equiv 1\pmod p$.
        \item{(loose-loose pairings)} The non-loose Hopf links with both components loose and having $\tbr(L_1)=k_1+\frac{p'}{p}$ and $\tbr(L_2)=k_2+\frac{q}{p}$ form the following mountain range:
        
        When we fix the loose $L_1$
        %%%%%%%%%%%%%%%%%%%%%%%%%%%%%%%%%%%%%%%%%%%%%%%%%%%%%%%%%%%%%%%%%%%%%%%%%%%%%%%%%%%%%%
        \begin{enumerate}[(i)]
            \item For $k_1=0$, the loose Legendrian representatives of $L_2$ forms \[
  \left\{\begin{array}{lr}
       1 , & n=1,2\\
     |a_2+1|\cdots|a_{n-2}+1||a_{n-1}|  ,  &n\geq 3\\
     
        \end{array}\right\} 
  \]many loose forward slashes and the same number of loose back slashes based at $(r,\frac{q}{p})$, and
  \[ \left\{\begin{array}{lr}
       0 , & n=1\\
       |a_1+1| , & n=2\\
     |a_1+2||a_2+1|\cdots|a_{n-2}+1||a_{n-1}| ,  &n\geq 3\\
     
        \end{array}\right\} \] many loose $V$'s based at $(r,\frac{q}{p})$.
        
        Additionally, there are
          \[ \left\{\begin{array}{lr}
       2|a_0-1| , & n=1\\
    |a_0+1||a_1+1|\cdots|a_{n-1}| ,  &n\geq 2\\
       \end{array}\right\}\]
$V$'s based at $(r,1+\frac{q}{p}).$ 
%%%%%%%%%%%%%%%%%%%%%%%%%%%%%%%%%%%%%%%%%%%%%%%%%%%%%%%%%%%%%%%%%%%%%%%%%%%%%%%%%%%%%%%%%%%%%%%%%%%%
            \item  For $k_1=1$, the loose Legendrian representatives of $L_2$ forms \[
  \left\{\begin{array}{lr}
       1 , & n=1\\
     |a_2+1|\cdots|a_{n-1}+1||a_{n}-1|  ,  &n\geq 2\\
     
        \end{array}\right\} 
  \]many loose forward slashes and the same number of loose back slashes based at $(r,\frac{q}{p})$, and
  \[ \left\{\begin{array}{lr}
       |a_1| , & n=1\\
     |a_1+2||a_2+1|\cdots|a_{n-1}+1||a_{n}-1| ,  &n\geq 2\\
     
        \end{array}\right\} \] many loose $V$'s based at $(r,\frac{q}{p})$. Additionally, there are \[|a_0+1||a_1+1|\cdots|a_{n}-1|\]
$V$'s based at $(r,1+\frac{q}{p}).$ 

%%%%%%%%%%%%%%%%%%%%%%%%%%%%%%%%%%%%%%%%%%%%%%%%%%%%%%%%%%%%%%%%%%%%%%%%%%%%%%%%%%%%%%%%%%%%%%%%%%%%%%%%
            \item For $k_1\geq2$, the loose Legendrian representatives of $L_2$ forms \[
  \left\{\begin{array}{lr}
       2 , & n=1\\
     2|a_2+1|\cdots|a_{n-1}+1||a_{n}|  ,  &n\geq 2\\
     
        \end{array}\right\} 
  \]many loose forward slashes and the same number of loose back slashes based at $(r,\frac{q}{p})$, and
  \[ \left\{\begin{array}{lr}
       2|a_1+1| , & n=1\\
     2|a_1+2||a_2+1|\cdots|a_{n-1}+1||a_{n}| ,  &n\geq 2\\
     
        \end{array}\right\} \] many loose $V$'s based at $(r,\frac{q}{p})$. Additionally, there are  $2|a_0+1||a_1+1|\cdots|a_{n}|$         %\left\{\begin{array}{lr}
       %2|a_0+1||a_1+1| , & n=1\\
      % 2|a_0+1||a_1+1|\cdots|a_{n}|\\
      % \end{array}\right 
       %\]
        $V$'s based at $(r,1+\frac{q}{p}).$ 
        \end{enumerate}
        An algorithm for $r$ and the Euler classes will be given in Section~\ref{sec:background}.
       If we fix $L_2$, the mountain range is given in Figure~\ref{fig:Fix_K_2_general} for $n\geq 2$ and in Figure~\ref{fig:Fix_K_2_general_n=1} for $n=1$.
	\end{enumerate}
\end{theorem}

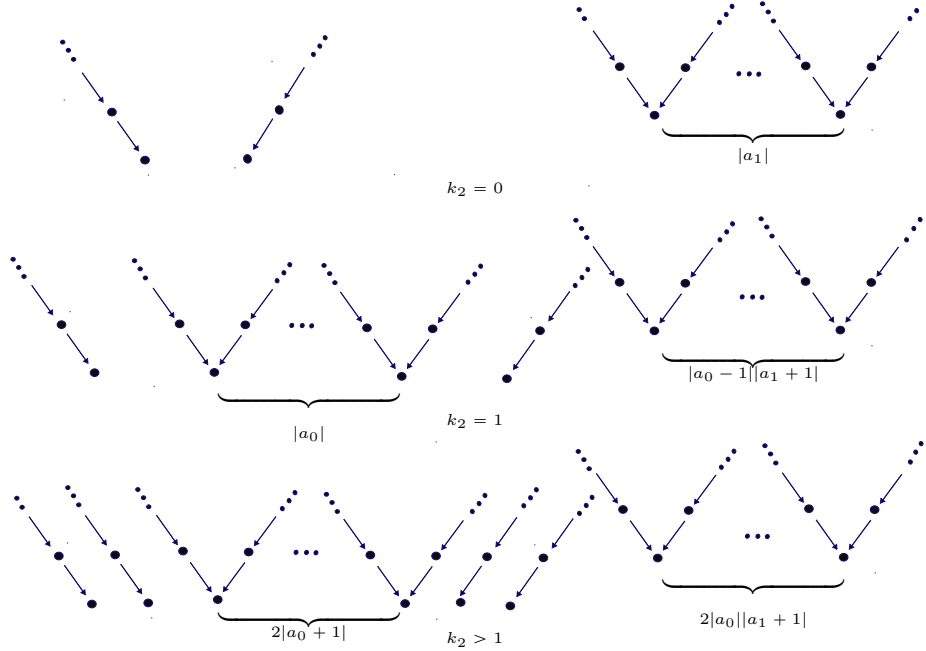
\begin{figure}
    \centering
    \labellist
   
    \pinlabel $\underbrace{\hspace{7 em}}$ at 450 330
    \pinlabel $\underbrace{\hspace{7 em}}$ at 1120 390
    \pinlabel $\underbrace{\hspace{7 em}}$ at 1120 730
    \pinlabel $\underbrace{\hspace{7 em}}$ at 450 00
  \pinlabel $\underbrace{\hspace{7 em}}$ at 1120 50
    \tiny
     \pinlabel $k_2=0$ at 700 650
     \pinlabel $k_2=1$ at 700 300
     \pinlabel $k_2>1$ at 700 -30
    \pinlabel $|a_0|$ at 450 280
    \pinlabel $|a_0-1||a_1+1|$ at 1120 370
    \pinlabel $|a_1|$ at 1120 700
    \pinlabel $2|a_0+1|$ at 450 -20
    \pinlabel $2|a_0||a_1+1|$ at 1120 0
    \endlabellist
    \includegraphics[scale=0.25]{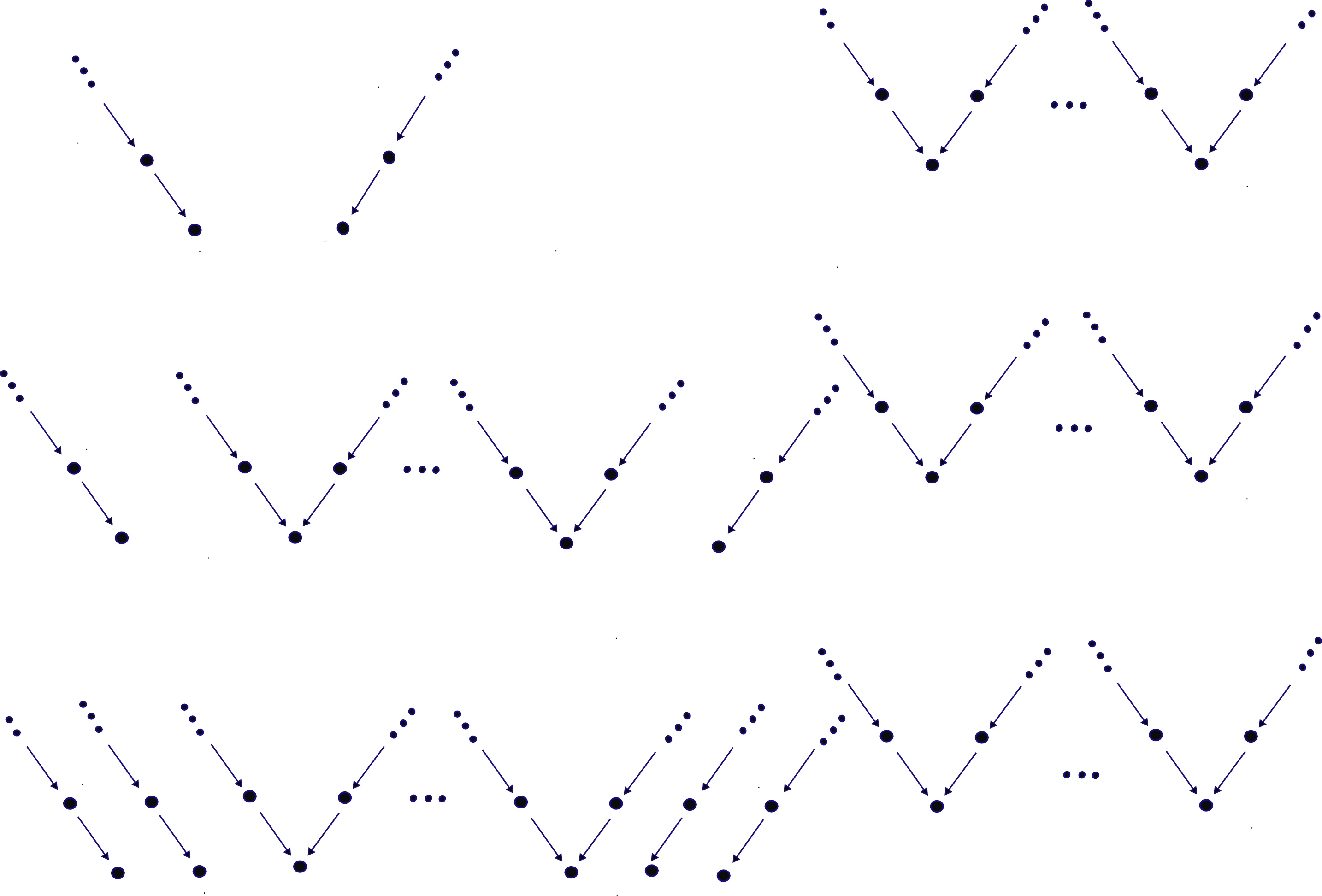}
    \caption{The loose mountain range of $L_1$ when we fix $L_2$ for $-p/q=[a_0,a_1].$ }
    \label{fig:Fix_K_2_general_n=1}
\end{figure}

Note that, all the above results are strongly exceptional as they have no Giroux torsion in their complement. The following theorem gives a complete classification of non-loose Hopf links with non-zero (convex) Giroux torsion.
\begin{theorem}
\label{thm:Giroux_torsion}
Suppose $L_1\sqcup L_2$ be the positive Hopf link in $L(p,q)$ where $p>q\geq 1$. The classification of exceptional $L_1\sqcup L_2$ with $\pi$ Giroux torsion in its complement is as follows:
\begin{enumerate}
    \item For $q=1$, for each choice of integers $(k_1, k_2)\neq 0$ and a natural number $m$ there is exactly a pair of non-loose Legendrian (positive) Hopf links $L_1\sqcup L_2$ distinguished by the rotation numbers, with $\tbr(L_i)=k_i+\frac{1}{p}$ and with $m\pi$ Giroux torsion in the complement. For $k_1=k_2=0$, there is exactly one non-loose Hopf link with $m\pi$ Giroux torsion in the complement. For $k_1<0, k_2<0$, the Hopf link is in tight $L(p,1).$ For $k_1<0$ and $k_2\geq 0$ $L_1$ is loose and $L_2$ is non-loose. For $k_1\geq 0$ and $k_2<0,$ $L_1$ is non-loose and $L_2$ is loose. Finally for $k_1\geq 0, k_2\geq 0 $ and $m>0$ both components are loose. 
    \item For $q\neq 1$, for each choice of integers and a natural number $m$ there is exactly a pair of non-loose Legendrian Hopf links $L_1\sqcup L_2$ distinguished by their rotation numbers and with $m\pi$ Giroux torsion in the complement. The $\tbr$ of the components are as follows:
    \begin{enumerate}
   \item  For $k_1<0, k_2<0,$ the Hopf links are in tight $L(p,q)$ with $\tbr(L_1)=k_1-p''/p$ and $\tbr(L_2)=k_2+q/p$.
    \item For $k_1\geq 0, k_2<0 ,$ $L_1$ is non-loose and $L_2$ is loose with $\tbr(L_1)=k_1+p'/p$ and $\tbr(L_2)=k_2+q/p.$
    \item  For $k_1< 0, k_2>0 ,$ $L_1$ is loose and $L_2$ is non-loose with $\tbr(L_1)=k_1-p''/p$ and $\tbr(L_2)=k_2+q/p.$ and
    \item  For $k_1\geq 0,k_2\geq 0$, both the components are loose with $\tbr(L_1)=k_1+p'/p$ and $\tbr(L_2)=k_2+q/p.$
\end{enumerate}
where $p', p''$ are same as in Theorem~\ref{thm:general}.
\end{enumerate}
   
\end{theorem}
\begin{remark}
    The above theorem recovers Theorem 1.2 (e) from \cite{Geiges_Onaran} for $p=1$. Furthermore, the above theorem also gives the classification of non-loose Hopf links with nonzero Giroux torsion in $L(p,1)$, filling the gap left in \cite{chatterjee2025_Hopf}.
\end{remark}
\subsection{Organization} In Section~\ref{sec:background} we give some preliminaries and definitions. In Section~\ref{sec:examples}, we prove Theorem~\ref{thm:L(p,1)}, and \ref{thm:L(2n+1,2)}. In Section~\ref{sec:general_result} we give the proof of our main theorems Theorem~\ref{thm:general} and ~\ref{thm:Giroux_torsion}. We finish by including an algorithm for the contact surgery diagrams of the non-loose realizations in Section~\ref{sec:contactsurgery}.
\subsection{Acknowledgement} The author would like to express her deepest gratitude to John Etnyre for his invaluable guidance and encouragement throughout the project, as well as for the many enlightening discussions. This project started when the author was visiting Georgia Tech in July 2024 and part of this work was completed during the author's visit in July 2025. This research is partially supported by the
Georgia Institute of Technology’s Elaine M. Hubbard Distinguished Faculty Award and NSF-AWM mentoring travel grant. The author would also like to thank Georgia Tech for their hospitality.

%%%%%%%%%%%%%%%%%%%%%%%%

\section{Background}
\label{sec:background}
We assume the reader is familiar with basic contact geometry, Legendrian knots and convex surface theory, as can be found in \cite{Etnyre_Honda_knots, Geiges, Honda_tight}. In this section, we will briefly recall some of the important results for the convenience of the reader and to establish the notations that we use throughtout the paper.  In Section~\ref{ssec:Farey} we recall the Farey graph and discuss its relation with curves on tori, next we discuss the classification of tight contact structures on solid tori, $T^2\times [0,1]$ and lens spaces. In Section~\ref{ssec:Legendrian}, we review basic facts about Legendrian knots, such as standard neighborhoods and how these are related by stabilizations. In Section~\ref{ssec:Hopf} and Section~\ref{ssec:computation_preliminary} we discuss rationally null homologous knots and computation of their classical invariants. 
\subsection{The Farey graph}\label{ssec:Farey} The Farey graph is essential in keeping track of the embedded essential curves on a torus. Recall that once we choose a basic for $H_1(T^2)$, the embedded essential curves on $T^2$ are in one to one correspondence with $\mathbb{Q}\cup\infty.$ 

The Farey graph is constructed in the following way. See Figure~\ref{fig:farey}. Consider the unit disk in the $xy$-plane. Label the point $(0,1)$ as $0=\frac{0}{1}$
and $(0,-1)$ as $\infty=\frac{1}{0}.$ Connect these two points by a straight line. Now if a point on the boundary of the disk has a positive $x$-coordinate and if it lies between two points labeled $\frac{a}{b}$ and $\frac{c}{d}$ then we label it as $\frac{a+c}{b+d}.$ We call this the ``Farey sum'' of $\frac{a}{b}$ and $\frac{c}{d}$ and denote as $\frac{a}{b}\bigoplus\frac{c}{d}$. Now we connect this point with both $\frac{a}{b}$ and $\frac{c}{d}$ by hyperbolic geodesics (note that, we consider a hyperbolic metric on the interior of the disk). We keep iterating this process until all positive rational numbers are labeled on the boundary of the disk. We do the same thing for all the negative national number by considering $\infty$ as $\frac{-1}{0}.$ We use $\frac{a}{b}\bigoplus k\frac{c}{d}$ to denoted the $k^{th}$ iterated Farey sum i.e. we add $\frac{c}{d}$ to $\frac{a}{b}$, $k$ times. Note that, two embedded curves on the torus with slopes $r$ and $s$ will form a basis if and only if there is exactly an edge between them in the Farey graph. We also introduce the dot product of two rational numbers here $\frac{a}{b}\bigcdot\frac{c}{d}=ad-bc$ as the minimum number of times the curves can intersect.

We have the following well-known lemma, See \cite{Etnyre_Lafountain_Tosun}
\begin{lemma}\label{lem:clockwise_anticlockwise}
Suppose $q/p<-1$. Given $q/p=[a_0,\cdots, a_n],$ let $(q/p)^c=[a_1,\cdots, a_n+1]$ and $(q/p)^a=[a_0,\cdots, a_{n-1}].$ There will be an edge in the Farey graph between each pair of numbers $q/p$, $(q/p)^c,$ and $(q/p)^a.$ Moreover, $(q/p)^c$ will be fathest clockwise point from $q/p$ that is larger than $q/p$ with an edge to $q/p$, while $(q/p)^a$ will be the farthest anti-clockwise point from $q/p$ that is less than $q/p$ with an edge to $q/p.$
    
\end{lemma}

In the above lemma if $a_n+1=-1$, then we consider $[a_0,\cdots, a_n+1]$ to be $[a_0,\cdots a_{n-1}+1].$ Also, if $q/p$ is a negative integer then $(q/p)^a=\infty.$

A path in the Farey graph is a sequence of elements $p_1,p_2,\cdots, p_k$ in $\mathbb{Q}\cup\infty$ moving clockwise such that each $p_i$ is connected to $p_{i+1}$ by an edge in the Farey graph, for $i<k.$ Let $P$ be the minimal path in the Farey graph that starts at $p_1$ and goes clockwise to $p_k. $ We say a path in the Farey graph is a {\it decorated path} if all of the edges are decorated by a $+$ or $-.$ We call a path in the Farey graph {\it a continued fraction block} if there is a change of basis such that the path goes from $0$ clockwise to $n$ for some positive $n$. We say two choices of signs on the continued fraction block are related by {\it shuffling} if the number of $+$ signs in the continued fraction blocks are the same.

Next we introduce a notation that we frequently use. Given two numbers $r,s$ in $\mathbb{Q}\cup\infty,$ we denote by $[r,s]$ all the numbers that are clockwise to $r$ and anti-clockwise to $s$ in the Farey graph.

\begin{center}
    \begin{figure}[htbp]
\centering
\def\svgwidth{0.99\columnwidth}
\labellist
\small\hair 2pt
  \pinlabel {${0}$} at 236 520
\pinlabel {$\infty$} at 236 20
\pinlabel{$1$} at 470 250
\pinlabel{$-1$} at 0 250
\pinlabel{$\frac{1}{2}$} at 420 420
\pinlabel{$2$} at 420 100
\pinlabel{$3$} at 350 30
\pinlabel {$\frac{3}{2}$} at 470 170
\pinlabel{$\frac{2}{3}$} at 470 330
\pinlabel{$\frac{1}{3}$} at 355 480
\pinlabel {$-\frac{1}{3}$} at 100 480
\pinlabel {$-\frac{2}{3}$} at 0 330
\pinlabel {$-\frac{1}{2}$} at 40 420
\pinlabel {$-\frac{3}{2}$} at 10 170
\pinlabel {$-2$} at 50 100
\pinlabel {$-3$} at 120 30
\endlabellist
  \includegraphics[scale=0.5]{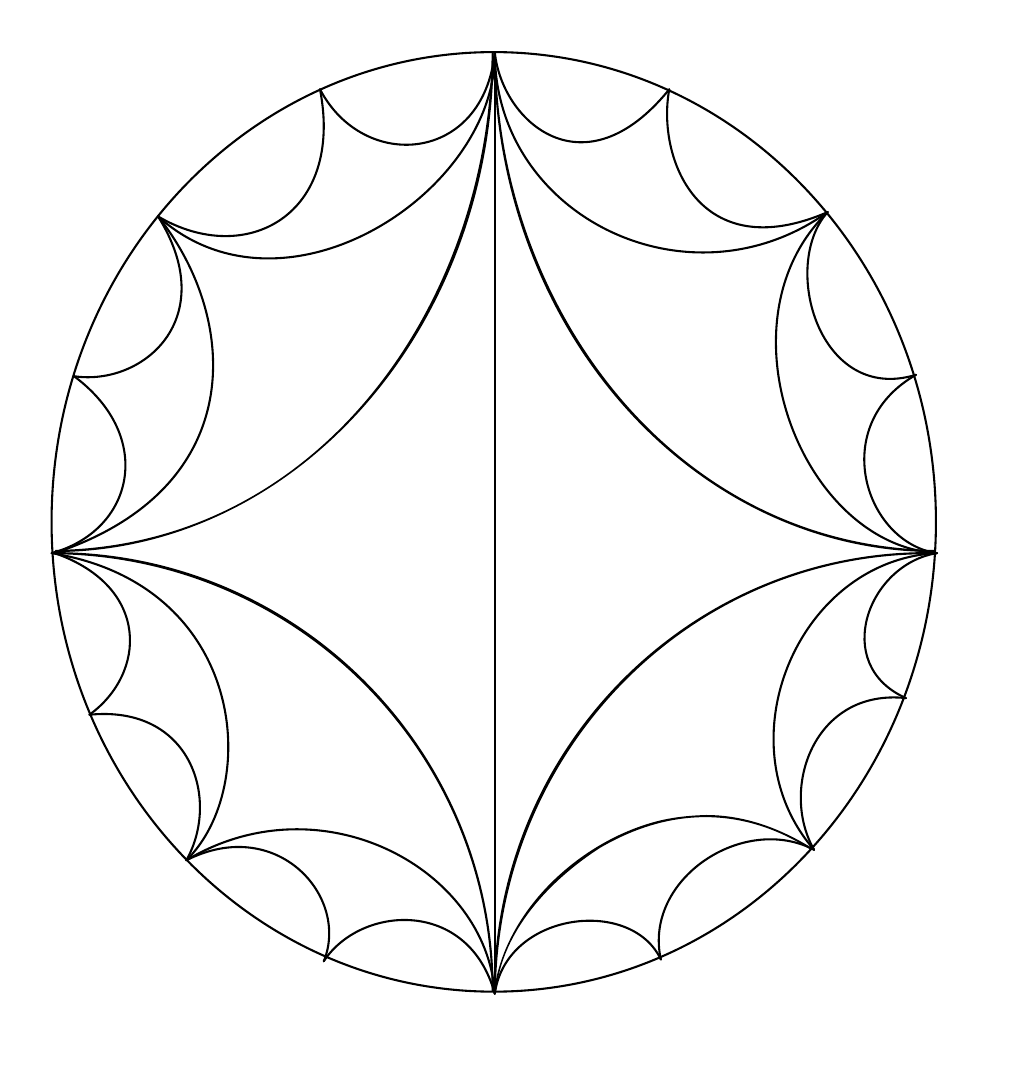}
  	\caption{The Farey graph. }
  	\label{fig:farey}
  \end{figure}
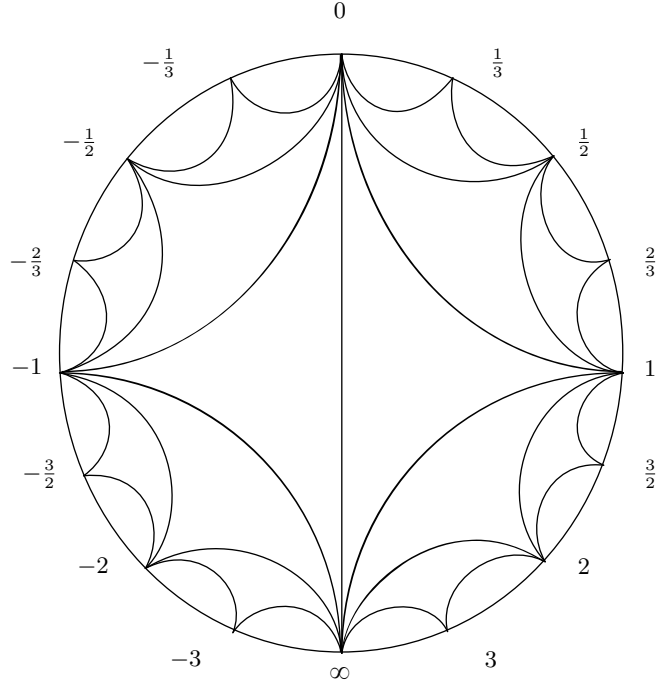 

\end{center}

\subsection{Tight contact structures on solid tori, $T^2\times[0,1]$ and $L(p,q)$} \label{ssec:tight_classification} Here we briefly recall the classiifcation of tight contact structures on $T^2\times [0,1], S^1\times D^2$ and lens spaces due to Giroux \cite{Giroux} and Honda \cite{Honda_tight}. We discuss the classiifcation results along the lines of Honda.
\subsubsection{Contact structures on $T^2\times[0,1]$} Consider a contact structure $\xi$ on $T^2\times [0,1]$ that has convex boundary with dividing curves of slope $s_i$ on $T^2\times\{i\}$ for $i=0,1$. We also assume that there are exactly two dividing curves on each of the boundary component. We say $\xi$ is {\it minimally twisting} if any convex torus in $T^2\times [0,1]$ parallel to the boundary has dividing slope in $[s_0,s_1]$. We denote the minimally twisting contact structures, up to isotopy, on $T^2\times[0,1]$ with the above boundary conditions as $\text{Tight}^{min}(T^2\times[0,1];s_0,s_1).$ Giroux \cite{Giroux} and Honda \cite{Honda_tight} classified tight contact structures in $\text{Tight}^{min}(T^2\times[0,1];s_0,s_1)$ establishing the following result.

\begin{theorem}\label{thm:tight_thickened_tori}
Each decorated minimal path in the Farey graph from $s_0$ clockwise to $s_1$ describes an element of $\text{Tight}^{min}(T^2\times[0,1];s_0,s_1).$ Two such decorated paths will describe the same contact structure if and only if the decorations differ by shuffling in the continued fraction blocks.
    
\end{theorem}

Notice that, if $s_0$ and $s_1$ has exactly one edge between them then there are exactly two tight contact structures in $\text{Tight}^{min}(T^2\times[0,1];s_0,s_1).$ These are called {\it basic slices} $B_\pm$ and the correspondence of the theorem can be understood via stacking basic slices according to the decoration in the path that describes the contact structure. The two different contact structures on a basic slice can be distinguished by their relative Euler classes, we call them {\it positive} and {\it negative } basic slices.

The relative Euler class of the contact structure in $\text{Tight}^{min}(T^2\times[0,1];s_0,s_1)$ can be computed as follows: let $s_0=r_0, s_1, \cdots, r_k=s_1$ be the vertices of the minimal path from $s_0$ to $s_1$ amd $\epsilon_i$ be the sign of the basic slice with boundary slopes $r_{i-1}$ and $r_i.$ Then the relative Euler class of the contact structure associated with this path  is Poincar\'{e} dual to the curve
\[\sum_{i=1}^{k}\epsilon_i(r_i\bigominus r_{i-1})\]
where $\frac{a}{b}\bigominus\frac{c}{d}=\frac{a-c}{b-d}.$

Next we discuss shortening of a non-minimal path. Suppose $P$ is a non-minimal path in the Farey graph. So there will be a vertex $v$ in $P$ such that there is an edge between its neighboring vertices $v'$ and $v''$. We can shorten this path by removing $v$ and the two edges and replacing it by the edge between $v'$ and $v''$ in the path. We call the new path $P'$. If $P$ were a decorated path, then we call the shortening to get $P'$ {\it inconsistent} if the edges removed had different signs and {\it consistent} if the signs are the same. When the shortening is consistent, then we can decorate the new edge in $P'$ by the sign of the removed edges, thus $P'$ is the new decorated path.

For any decorated path in the Farey graph, even non-minimal, one can construct a contact structure on $T^2\times[0,1]$ by stacking basic slices. The following result due to Honda tells us when these path will lead to tight contact structures.
\begin{theorem}\cite{Honda_tight}\label{thm:tight} Let $\xi$ be a contact structure on $T^2\times[0,1]$ described by a non-minimal decorated path $P$ in the Farey graph from $s_0$ to $s_1$. Then $\xi$ is tight if and only if one may consistently shorten the path to a shortest path from $s_0$ to $s_1$.
    
\end{theorem}

Next we discuss convex Giroux torsion. Consider $\xi=\ker(\sin 2\pi z dx+\cos 2\pi z dy)$ on $T^2\times \mathbb{R}$ where $(x,y)$ is the co-ordinate on $T^2$ and $z$ is the co-ordinate on $\mathbb{R}.$ Consider the region $T^2\times[0,k]$ for $k\in\frac{1}{2}\mathbb{N}$ and notice that the contact planes twist $k$ times as $z$ goes from $0$ to $k$. We can perturb $T^2\times\{0\}$ and $T^2\times\{k\}$ so that they become convex with two dividing curves of slope $0$. Let $\xi^k$ denotes the resulting contact structure on $T^2\times[0,1]$ after we identify $T^2\times [0,1]$ with $T^2\times[0,k]$.
For $k\in\frac{1}{2}\mathbb{N}$, we call $(T^2\times[0,1],\xi^k)$ a {\it convex $k$ Giroux torsion layer} and if it embeds into a contact manifold $(M,\xi)$, we say $(M,\xi)$ has {\it convex $k$ Giroux torsion}. We say $(M,\xi)$ has exactly $k$ Giroux torsion if one can embed $(T^2\times [0,1],\xi^k)$ into $(M,\xi)$ but cannot embed $(T^2\times[0,1], \xi^{k+\frac{1}{2}})$ in $(M,\xi).$ On the other hand $(M,\xi)$ has no convex Giroux torsion or zero Giroux torsion if $(T^2\times[0,1], \xi^k)$ does not embed in $(M,\xi)$ for any $k\in\frac{1}{2}\mathbb{N}.$

\begin{remark}There is a difference between a convex Giroux torsion layer and a Giroux torsion layer. When we talk about Giroux torsion in general we do not require the boundary tori to be convex, we just need them to be pre-Lagrangian. But as one can always find a convex tori with same boundary slope inside the layer, every Giroux torsion layer contains a convex Giroux torsion layer of same slope (converse need not be true). In this paper when we refer to Giroux torsion we will be talking about convex Giroux torsion layer.
 
\end{remark}
\subsubsection{Contact structures on solid tori} Now we discuss tight contact structures on solid tori. While we will usually use ``standard'' coordinates on a solid torus so that the meridional slope is $-\infty=\infty,$ but it will be convenient sometimes to use different coordinates. The notation is as follows: Consider $T^2\times[0,1]$ and choose a basis for $H_1(T^2)$ so that we may denote curves on $T^2$ by rational numbers $\cup\infty.$ Given $r\in\mathbb{Q}\cup\infty$ we can foliate $T^2\times\{0\}$ by curves of slope $r$. Let $S_r$ be the result of collapsing each leaf in the foliation of $T^2\times\{0\}$ to a point. One can check that $S_r$ is a solid torus with meridional slope $r$. We say $S_r$ is a {\it solid torus with lower meridian $r$.} One the other hand one could foliate $T^2\times\{1\}$ similarly by curves of slope $r$ and collapse the curves to obtain $S^r.$ This is called a {\it solid torus with upper meridian $r$}. Note that the standard solid torus is $S_\infty$ and unless otherwise stated this is the solid torus we are talking about.

\subsubsection{Contact structures on $L(p,q)$} The lens space $L(p,q)$ can be defined as $-p/q$ surgery on the unknot in $S^3$. Equivalently, we can think of $L(p,q)$ as $T^2\times[0,1]$ with the curves of slope $-p/q$ collaped on $T^2\times\{0\}$ and curves of slope $0$ collaped on $T^2\times\{1\}$. We can further describe $L(p,q)$ as a result of gluing $S_{-p/q}$, a solid torus with lower meridian $-p/q$ to another solid torus $S^0$, a solid torus of upper meridian $0.$ Giroux \cite{Giroux} and Honda \cite{Honda_tight} classified the tight contact structure on $L(p,q)$ as follows:
\begin{theorem}
Let $P$ be a minimal path in the Farey graph $-p/q$ clockwise to $0.$ The tight contact structures on $L(p,q)$ are in one-to-one correspondence with assignments of signs to all but the first and last edge in $P$ up tp shuffling in continued fraction blocks.
    
\end{theorem}

\subsection{Knots in contact maniolds}\label{ssec:Legendrian} A {\it standard neighborhood} of a Legendrian knot $L$ in $(M,\xi)$ is a solid torus $N(L)$ on which $\xi$ is tight and $\partial N(L)$ is convex with two dividing curves of slope $\tb(L).$ One may arrange via a small isotopy that the characteristic foliation consists of two lines of singularities called {\it Legendrian divides} and curves of slope $s\neq\tb(L).$ These are called {\it ruling curves.} Convesely, given a solid torus in a contact manifold $(M,\xi)$ on which $\xi$ is tight and having convex boundary and two dividing curves of slope $n$, then there exist a Legendrian knot $L$ with $\tb=n$ and $S$ being its standard neighborhood.

Given a Legendrian knot $L$, one can stabilize it in two ways, $S_\pm(L)$. Note that the standard neighborhood $N(S_\pm(L))$ of $S_\pm(L)$ is in $N(L)$ and $N(L)\setminus N(S_\pm(L))$is a basic slice where the sign of the basic slice depends on the sign of the stabilization. This basic slice has boundary slopes $\tb(L)$ and $\tb(L)-1$.
\subsection{Rationally null-homologous knots}\label{ssec:Rational}
For deatils on rationally null-homologous knots an interested reader is referred to \cite{Etnyre_rational}. Recall we say a knot $K$ is {\it rationally null-homologous} in $M$ if it is trivial in $H_1(M,\mathbb{Q}).$ In other words there exists a minimal integer $r$ such that $rK$ is trivial in $H_1(M,\mathbb{Z}).$ We call $r$ the order of $K$. One can build a {\it rational Seifert surface } $\Sigma$ for $K$ as explained in \cite{Etnyre_rational}. Note that, in general $\Sigma$ might not have connected boundary but it would not concern us in this paper. We assume $K$ is oriented and this will induce an orientation on $\Sigma.$

If $K'$ is another oriented knot in $M$ that is disjoint from $K$ then we define the {\it rational linking number} to be 
\[\text{lk}_{\mathbb{Q}}(K,K')=\frac{1}{r}\Sigma'\cdot K\]
where $\Sigma'\cdot K$ denotes the algebraic intersection of $\Sigma'$ and $K$. There is some ambiguity in this definition but that will not be an issue here.

Now let $K$ be a Legendrian knot in $(M,\xi)$. As mentioned in section~\ref{ssec:Legendrian} $K$ has a standard neighborhood with convex boundary and $2$ dividing curves determined by the contact framing. Let $K'$ be one of the Legendrian divides on $\partial N(K)$. We define the {\it rational Thurston-Bennequien invariant} of $K$ to be 
\[\tbr(K)=\text{lk}_{\mathbb{Q}}(K,K')\]

Next we define the rational rotation number following \cite{Etnyre_rational}. We consider the immersion $i\colon\Sigma\rightarrow M$ that is an embedding on the interior of $\Sigma$ and an $r$ to $1$ mapping of $\partial\Sigma$. We can now consider $i^*\xi$ as an oriented $\mathbb{R}^2$ bundle over $\Sigma$. Since $\Sigma$ is a surface with boundary we know that $i^*\xi$ can be trivialized as $i^*\xi=\mathbb{R}^2\times\Sigma$. Let $v$ be a non-zero vector field tangent to $\partial\Sigma$ inducing the orientation of $K$. Using the trivialization of $i^*\xi$ we can consider $v$ as  a map from $\partial\Sigma$ to $\mathbb{R}^2$. Now we can define the {\it rational rotation number} as follow:
\[\rotr(K)=\frac{1}{r}\text{winding}(v,\mathbb{R}^2)\]
Note that $\text{winding}(v,\mathbb{R}^2)$ is equivalent to the obstruction to extending $v$ to a non-zero vector field over $i^*\xi$ and thus can be interpreted as the relative Euler number.

\subsection{Hopf links in lens spaces} \label{ssec:Hopf} Here we will give the basics that we need in the next section. We will consider $L(p,q)$ as the union of two solid tori $V_1$ and $V_2$. We can think of $V_1$ as a solid torus with lower meridian $-p/q$ and $V_2$ a solid torus with upper meridian $0.$ If we fix two dividing curves of slope $s$ on $\partial V_1=\partial V_2,$ then a contact structure on $L(p,q)$ is determined by taking a tight contact structure on $V_1$ in $\text{Tight}(S_{-p/q},s)$ and another tight contact structure in $V_2$ in $\text{Tight}(S^0;s)$ and gluing them together. There is no guarentee that this gluing will give us a tight contact structure (in fact most of the time it will be overtwisted). 

We will think of the core of $V_i$ as the rational unknots $K_i$. If we are looking for Legendrian realizations of $K_i$ then we can take any slope $s$ with an edge to $-p/q$ (respectively $0$) in the Farey graph. We call the Legendrian representatives $L_i$. Now a tight contact structure on $V_i$ with convex boundary having two dividing curves of slope $s$ is a standard neighborhood of $L_1$. If we are looking for non-loose representative of $L_1\sqcup L_2$, then the link  complement must be tight. By link complement we mean the complement of the standard neighborhood of $L_1$ and $L_2$ which is diffeomorphic to $T^2\times I.$ So non-loose representaives of $L_1\sqcup  L_2$ are in one-to -one corrpondence with $\text{Tight}^{min}(T^2\times I; s_1,s_2)$ where $s_i$ are the dividing slopes of the standard neighborhood of $V_i$. Note that, $s_1$ can be either in $(-p/q, (-p/q)^c]$ or $[(-p/q)^a,-p/q)$ (recall our notation from section~\ref{ssec:Farey}). On the other hand, $s_2$ can be in either $(0, 0^c]$ or in $[(0)^a, 0)$.

We refer to all the slopes $s$ as {\it large slope} for $L_1$ %that is contained in the interval which contains the Seifert slope for $L_i$ and {\it small slope} otherwise. By 
if we start from the meridional slope of $V_1$ and traverse clockwise to slope $s$, we pass the Seifert slope for $L_1$. Otherwise, we call them {\it small slopes}. Thus for $L_1$ the slopes $s\in [(-p/q)^a,-p/q)$ are large slopes and $s\in(p/q,(-p/q)^c]$ are small slopes. On the other hand, for $L_2$ the slopes $s\in(0,0^c]$ are large slopes and $s\in[0^a,0)$ are small slopes (as $L_2$ is the core of $V_2$ and we are following the orientation of $T^2$ as the boundary of $V_1$, the convention is opposite in this case. In this case we need to start from the meridional slope and have to traverse counter clockwise to $s$).

Now when gluing $V_1$ with slope $s_1$, a tight $T^2\times [0,1]$ with slopes $s_1$ and $s_2$ and $V_2$ with slope $s_2$ we will have
four possibilities. If $s_1\in(-p/q, (-p/q)^c]$ and $s_2\in[0^a,0)$ the contact structure on $L(p,q)$ is clearly tight. In all three other cases that is when we consider a combination of small and large slopes or large slopes only, the contact structure on $L(p,q)$ will be overtwisted as in all the cases either the torus $V_1$ will contain a boundary parallel convex torus of slope $0$ or $V_2$ will contain a boundary parallel torus of slope $-p/q$ or both. Thus after gluing $V_1,V_2$ and the thickened torus the resulting contact structure will contain an overtwisted disk (a Legendrian divide of slope $0$ will bound a disk in $V_2$ and a Legendrian divide of slope $-p/q$ will bound an overtwisted disk in $V_1$). As we are considering non-loose Hopf links we will only consider those three cases.
  
%\textcolor{red}{Need to fix convention about small and large slope. Clockwise and anti-clockwise notion.}
\subsection{Computation of classical invariants in $L(p,q)$}\label{ssec:computation_preliminary} In this section we explain how to compute the classical invariants of the Hopf link components and the Euler class of the contact structure in $L(p,q)$.

\begin{lemma}
	Suppose $r=-\frac{p}{q}$ is the meridional slope of $V_1$ and $L_1$ being the core of $V_1$. Then $\tbr(L_1)=- \frac{1}{p}|s\cdot 0|$, if $s$ is a small slope and  $\tb(L_1)=\frac{1}{p}|s\cdot 0|$ if $s$ is large slope. Here $|s\cdot 0|$ denotes the number of intersection between $s$ and $0$. The same is true for $L_2$.
\end{lemma}
\begin{proof}
Notice that the meridional disk for $L_2$ will provide a Seifert surface for $L_1$ which has order $p$ in $L(p,q).$ So according to the definition the rational Thurston--Bennequin invariant is the rational linking number of $L_1$ with a contact push-off of $L_1$. In other words we count the number of intersection of a Legendrian divide on $\partial V_1$ with the Seifert slope and divide it by the order. The sign of $\tbr$ is determined by whether the slope in small or large. Note that, Thurston--Bennequin number measures the difference between the contact framing and the Seifert framing. As for small slope the contact framing is less that the Seifert framing, $\tbr$ is negative. On the other hand, as by definition the large slope is always greater than the Seifert framing and $\tbr$ must be positive.

%if a slope is small the contact structure starting from the meridional slope twist in a left handed way towards $s$. On the other hand, if the slope is large then the contact planes twist is a left handed way until it reaches the Seifert slope and then twist in a opposite direction 
\end{proof}
%\textcolor{red}{need to define $(-\frac{p}{q})^c$ and  $(-\frac{p}{q})^a$ }
\subsubsection{Computation of rotation number} Next we explain how to compute the rotation number of the components.  If $e$ is the Euler class of the contact structure on $T^2\times I$ and $D_i$ be the meridional disk of $L_i$, then the rational rotation number of the components are given by $\rotr(L_1)=\frac{1}{p}e(D_2)$ and $\rotr(L_2)=\frac{1}{p}e(D_1)$. For a similar explanation check \cite{Chatterjee-Etnyre-Min_2025_existence}. To calculate, the rotation number for the components $L_1\sqcup L_2$ having standard neighborhood with dividing slopes $s_{k_1}$ and $s_{k_2}$  we follow :
\begin{enumerate}
    %\item Write $-p/q=[a_0, a_1, \dots ,a_n]$. 
    \item We find the (decorated) shortest path $\{s_{k_1}=p_0, p_1, \dots, p_n=s_{k_2}\}$ between $s_{k_1}$ clockwise to $s_{k_2}$ on the Farey graph. If $s_{k_1}, s_{k_2}$ both belong to the negative region we continue. (In our case, $s_{k_1}<0$.) If $s_{k_2}> 0$, we apply a diffeomorphism of $T^2$ such that the new slope $<0$. In fact this can be done by applying the change of basis matrix \[\begin{pmatrix}
        1 &0\\
        -1&1\\
    \end{pmatrix} \]

 Thus we have a path in the negative region only.
 \item As mentioned before the relative Euler class of the contact structure associated with this path is Poincar\'e dual to the curve \[\sum_{i=0}^n\epsilon_i(p_i\bigominus p_{i-1})\] where $\epsilon=\pm 1$ depending on the decoration. 
 \item Now we evaluate this on the Seifert disk. In other words, for $L_1$ we evaluate on slope $0$ and for $L_2$ we use slope $-\frac{p}{q}.$ If we used the change of basis matrix from before we need to make the appropriate changes to the Seifert slope as well.
 \item Now $\rotr(L_1)=\frac{1}{p}e(D_2)$ and $\rotr(L_2)=\frac{1}{p} e(D_1)$.
    
\end{enumerate}

\begin{remark}
    One might notice a little discrepancy of the Euler class when evaluating on $D_1$ or $D_2$ in $L(p,q)$ for $q>1$. The reason behind this is unlike $L(p,1)$ case there does not exist any isotopy taking $D_1$ to $D_2$. Thus when we give the Euler class of $L(p,q)$ where the Hopf link lives, we are evaluating the Poincar\'{e} dual on $D_2$.
\end{remark}
\begin{remark}
    Note that, one could also use the method from \cite{Chatterjee_Kegel} to calculate the Euler classes of the contact structures. The algorithm from \cite{Ding_Geiges_Stipcisz} can be used to calculate the $d_3$ invariants from the explicit contact surgery diagrams of the candidates given in Section~\ref{sec:contactsurgery}.
\end{remark}

\subsubsection{Stabilization of Hopf link} For large slopes $s_{k_1}$ of $L_1$, $s_{k_1}$ must be in $[(-p/q)^a,-p/q)$ and has an edge to $-p/q$. So $s_{k_1}=(-p/q)^a\bigoplus k_1(-p/q)$. For large slopes of $L_2$, similarly we will have $s_{k_2}=1/k_2.$ Suppose $L_1$ be the Legendrian unknot corresponding to the tight contact structure on $V_1$ with dividing slope $s_{k_1}$ where $k_1>0$, $L_2$ be the Legendrian unknot corresponding to the tight contact structure on $V_2$ with dividing slope $s_{k_2}$ with $k_2>0$ and some tight contact structure on $T^2\times[0,1]$ in $\text{Tight}^{min}(T^2\times[0,1]; s_{k_1}, s_{k_2})$. Now inside each $V_i$ there are two solid tori $S_i^\pm$ (smoothly isotopic to $V_i$) with convex boundary and $2$ dividing curves of slope $s_{k_i-1}.$ Thus $V_i\setminus S_i^\pm$ is a Basic slice $B_i^\pm$ for $i=1,2$. Now $S_i^\pm$ is a standard neighborhood of a stabilization of $L_i$. Note that, as $V_1$ and $V_2$ have opposite orientation, if we fix the orientation of $V_1$ and work with it then while $B_1^\pm$corresponds to a positive and negative stabilization of $L_1$,  $B_2^\pm$ will correspond to a negative and positive stabilization of $L_2$. But as we are considering positive Hopf links $L_1$ and $L_1$ have opposite orientations. Thus, in our case $B_2^\pm$ will actually correspond to a positive and negative stabilization of $L_2$.

If $L_1$ and $L_2$ are both stabilized, then the complement is given by a contact structure on $B_1^\pm\cup T^2\times I\cup B_2^\pm$ where we attach $B_1$ on the front face of $T_0$ and attach $B_2$ on the back face $T_1$. Thus the path in the complement is extended by the two edges describing $B_1$ and $B_2$. This new path might not be minimal. If the path can be consistently shortened then we will see the link is still non-loose otherwise loose. Now note that, it is possible that $S_\pm(L_1)$ becomes loose as the complement of $S_\pm(L_1)$ i.e. $B_\pm\cup T^2\times [0,1]\cup V_1$ is overtwisted (there will be an inconsistent shortening) but the link remains loose. The same is true for $L_2$ as well. Next we show that any stabilization of $L_1$ and $L_2$ with dividing slopes $s_{k_1}, s_{k_2}$ where 
$k_1=k_2=0$ are loose.

\begin{lemma}
    \label{lem:stabilization_loose} Suppose $L_i^{s_{k_i}}$ is a non-loose rational unknot that is the core of $V_i$ with dividing slope $s_{k_i}$ where $s_{k_i}$ is a large slope. Then any stabilization of $L_i^{s_0}$ is loose for $i=1,2$.
\end{lemma}
\begin{proof}
    We prove it for $L_1$ and $L_2$ will follow similarly. Note that, $s_0$ large corresponds to the slope $(-p/q)^a.$ The complement of $L_1$ is given by $V_2$. When we stabilize $L_1^{s_0}$ we add a basic slice $B_\pm$ of slopes $\{(-p/q)^c, (p/q)^a\}$ to $V_1$. Notice that,  $B_\pm$ contains boundary parallel convex tori of  any slope between $(-p/q)^c$ clockwise to $(-p/q)^a$. In particular it contains a convex torus of slope $0$. Any Legendrian divide on this torus when included in $V_2$ will contribute to an overtwisted disk. Thus, $B_\pm\cup V_2$ will be overtwisted and $S_\pm(L_1)$ will be loose.
\end{proof}

\section{Classification result for $L(p,1)$ and $L(2n+1,2)$}

\label{sec:examples}
%\textcolor{purple}{A left wing in negative stabilization and a right wing is positive stabilization}

\begin{proof}[Proof of Theorem \ref{thm:L(p,1)}] Note that, as mentioned before to have an overtwisted lens space we need to consider three cases. (1) small slope $\cup$ large slope (2) large slope $\cup$ small slope and $(3)$ large slope $\cup$ large slope. For $L(p,1)$, we have $(-p)^c=-p+1$ and $(-p)^a=\infty$
\subsection{$\text{small slope} \cup \text{ large slope}$} \label{ssec:L(p,1)case1}In this case the dividing slopes on the standard neighborhood of $L_1$ and $L_2$ are given by $s_{k_1}=%(\frac{-p+1}{1})\oplus k_1(-p)=
-\frac{p(k_1+1)-1}{k_1+1}$ and $s_{k_2}=(\frac{1}{0})\oplus k_2(\frac{0}{1})=\frac{1}{k_2}$. %Denote the standard neighborhood of $L_i$ with dividing slope $s_{k_i}$ by $N(L_i)^{s_{k_i}}$ for $i=1,2$. 
From now on when we mention a link complement we mean the complement of the standard neighbourhoods of $L_1$ and $L_2$.
Note that, the link complement with these dividing curves is %$L(p,1)\setminus (N(L_1)^{s_{k_1}}\cup  (N(L_2)^{s_{k_2}})$ is 
 $T^2\times [0,1]$ with boundary slopes $s_{k_i}$. 
For $k_1=0$, a minimum path from $s_{k_1}$ to $s_{k_2}$ consists of a continued fraction block of length $k_1+1$ from $s_{k_1}$ to $\infty$ and thus corresponds to exactly $k_1+2$ tight contact structures. For $k_2=1$, a path from $s_{k_1}$ to $s_{k_2}$ consists of a continued fraction block of length $k_1$ from $s_{k_1}$ to $-p+1$, then a continued fraction block of length $p$ from $-p+1$ to $1$. This corresponds to $(k_1+1)(p+1)$ tight contact structures. Finally, for $k_2>1$
a minimum path from $s_{k_1}$ to $s_{k_2}$ consists of the following sub-paths, a continued fraction block of length $k_1$ from $s_{k_1}$ to $-p+1$, followed by a continued fraction block of length $p-1$ from $-p+1$ to $0$ and finally a jump from $-1$ to $\frac{1}{k_2}$. Thus decorations on this path correspond to $2(k_1+1)p$ tight contact structures.

Next we calculate the classical invariants of the components. Note that, by $L_i^{{k_i}}$ we mean the Legendrian unknot $L_i$ whose standard neighborhood has dividing slope $s_{k_i}.$

One could easily calculate the rational Thurston--Bennequin invariant of the components to be \[\tb_{\mathbb{Q}}(L_1^{{k_1}})=-(k_1+1)+\frac{1}{p}\] and  \[\tb_{\mathbb{Q}}(L_1^{{k_2}})=k_2+\frac{1}{p}.\]  To calculate the rotation numbers we follow the technique mentioned in Section \ref{sec:background}. For $k_1\geq 0$ and $k_2=0$, the rotation numbers of the components are $\rotr(L_1^{k_1})=(-k_1-1+2n)$ and $\rotr(L_2^0)=0$ where $0\leq n\leq k_1+1.$ The Euler class is 0. (Here $n$ counts the number of negative basic slices in the path from $s_{k_1}$ to $\infty$.) For $k_2=1$, the rotation numbers are $\rotr(L_1)^{k_1})=(-k_1+2n)-\frac{(p-2m)}{p}$, $\rotr(L_2^1)=-1+\frac{2m}{p}$, and the Euler class is given by $-p+2m$  where $0\leq m\leq p, \ 0\leq n\leq k_1.$ In the later case $n$ counts the number of negative basic slice in the path $s_{k_1}$ to $-p+1$ and $m$ counts the negative signs in the following path from $-p+1$ to $1$.

It is easy to see that $L_1^{{k_1}}$ is loose (the complement of $L^k_1$ in this $L(p,1)$ contains a convex torus of slope $0$, a Legendrian divide on this torus contributes to an overtwisted disk in $V_2$) and $L_2^{{k_2}}$ is non-loose for all $k_1, k_2\geq 0$. Next we will see how the stabilization of $L_2^{{k_2}}$ are related with one another. We will start with $L_2^0$. Any stabilization of $L_1^0$ will be loose as shown in Lemma~\ref{lem:stabilization_loose}. The complement of $L_2^{{0}}$ consists of a union of continued fraction blocks. If all the basic slices in each of the continued fraction block is positive, then  a positive stabilization of $L_2$ will add a half-Giroux torsion in the complement of the link and $L_1^{{k_1}}\sqcup S_+(L_2^0)$ is still non-loose. Same is true if there are only negative slices and we negatively stabilize $L_2^0$. In all the other cases, any stabilization of $L_2^0$ will loosen the link.

For $k_2=1$, we have $(k+1)(p+1)$ non-loose representative and the path in the complement 
 of $L_2^1$ is given by a solid torus with meridional slope $-p$ and dividing slope $s_{k_2}=1$. We can also think of this as solid torus with boundary slope $s_{k_1}$ and meridional slope $-p$ union a $T^2\times I$ with boundary slopes $\{s_{k_1}, -p+1\}$ and another $T^2\times I$ with slopes $\{-p+1,1\}$. Stabilizing $L_2^1$ corresponds to adding a basic slice $B_\pm$ with boundary slopes $\{1, \infty\}$. Now adding this basic slice will allow us to shorten the path $\{1,0,\cdots, -p+1\}$ to $\{-p+1,\infty\}$. If the path from ${-p+1} $ to $1$ consists of only positive (resp. negative) signs then we can consistently shortening the path if the basic slice is positive (resp. negative). If the path has a mixed sign, any stabilization will lead to an inconsistent shortening as we can shuffle the signs in a continued fraction block and arrange it so that the shortening is always inconsistent. If we denote the component $L_2^1$ with having the complementary path $\{-p+1,\cdots, 1\}$ decorated with all positive signs $(L_2^1)^{all, +}$ then $S_+((L_2^1)^{all, +})$ is non-loose. Similarly, $S_-((L_2^1)^{all, -}$) is non-loose. These two forms the base of the back and forward slashes. If the path has any mix of signs and we denote the number of negative slices in this path by $m$ then  $S_\pm((L_2^1)^m)$ for $m=1,\cdots, p-1$ will be loose and thus $(L_2^1)^m$ will be the bases of $p-1$ $V$s. For $k_2=2$, the complement of $L_1^{k_1}\sqcup L_2^2$ can be subdivided into a path from $-p+1$ to $0$ and then a jump from $0$ to $\frac{1}{2}$. Suppose $i$ denotes the number of negative slices in the first part of the path, $i=0,1,\cdots p-1$. We denote $L_2^{2}$ corresponding to this path as $(L_2^{2})^{i,\pm}$ where $\pm$ corresponds to the sign of the jump. Clearly, $S_\pm((L_2^{2})^{i,\pm})$ for $i=0,\cdots p-1$ are all non-loose as this leads to a consistent shortening. We see that  $S_+((L_2^{2})^{0,+})$ and  $S_-((L_2^{2})^{p-1,-})$ coincide with $(L_2^1)^{all,+}$ and $(L_2^1)^{all,-}$ and are part of the back and forward slashes. $S_\mp((L_2^2)^{m,\pm})$ are loose. A similar analysis works for $k_2>2$ too.

 %let us denote the sub-path  in the complement from $-1$ to $\frac{1}{k_2}$ by $P$ and we have the following choice for decorations: $++$, $+-$, $-+$ and $--$. After stabilization the path can be shortened to a continued fraction block. We denote $L_2^2$ with these choices as $(L_2^2)^{\pm\pm}$ and $(L_2^2)^{\pm\mp}$. For the same reason as above $S_\pm((L_2^2)^{\pm\pm})$ is non-loose and $S_\pm((L_2^2)^{\pm\mp})$ is loose. On the other hand, $S_\mp((L_2^2)^{\pm\mp})$ is non-loose and they are the same up to shuffling the signs in the complement. So, in total we will have $(p+1)(k_1+1)$ non-loose stabilizations which coincide with $L_2^1$. Now for $k_2>1$, a similar observation tells us that $S_\pm((L_2^{k_2})^{\pm\pm})$, $S_\mp((L_2^{k_2})^{\pm\mp})$ are all non-loose. The only difference is that now $S_-((L_2^{k_2})^{+-})$ and $S_+((L_2^{k_2})^{-+})$ are distinct. These stabilizations coincide with $L_2^{{k_2-1}}$ for $k_2>2$.

 Now putting all these together we have the following mountain range for $L_2$ fixing $L_1$, one non-loose $V$ based at $(0,\frac{1}{p})$ and $p-1$ non-loose $V$s based  at $(-1+\frac{2m}{p}, 1+\frac{1}{p})$ for $m=1,2,\cdots, p-1$.

  \begin{center}
     \begin{table}[htbp]
         \centering
         \begin{tabular}{||c||c||c||c||}
         \hline
              $k_1$ &$k_2$ &$T^2\times I$ & number \\
              \hline
              0 &0 & $I$-invariant &1\\
              \hline
              0 &1 &$(\infty,1)$ &2\\
              \hline
              0 &2 &$(\infty, \frac{1}{2})$ & 3\\
              \hline
              0 &>2 &$(\infty,\frac{1}{k_2})$ &4\\
              \hline
              1 &0 &$(-p-1, \infty)$ &2\\
              \hline
              1 &1 &$(-p-1,1)$ &$p+3$\\
              \hline
              1 &>1 &$(s_{k_1}, \frac{1}{k_1})$ & $2(p+2)$\\
              \hline
              2 &0 &$(s_{k_1}, \infty)$ &3\\
              \hline
              >2 &0 &$(s_{k_1},\infty)$ & 4\\
              \hline
              >1 &1 &$(s_{k_1}, 1)$ & $2(p+2)$\\
              \hline
              >1 &>1 &$(s_{k_1},\frac{1}{k_2})$ &$4(p+1)$\\
              \hline
         \end{tabular}
         \caption{Number of non-loose Hopf links in $L(p,1)$ for large $\cup$ large.}
         \label{tab:table_L(p,1)}
     \end{table}
 \end{center}
 %%%%%%%%%%%%%%%%%%%%%%%%%%%%%%%%%%%%%%%%%%%%%%%%%%%%%%%%%%%%%%%%%%%%%%%%%%%%%
\subsection{Large slope $\cup$ small slope}\label{ssec:L(p,1)case2} This case corresponds to dividing slopes $s_{k_1}=%(\frac{-1}{0})\oplus k_1(\frac{-p}{1})=
-\frac{k_1p+1}{k_1}$ and \\$s_{k_2}=-\frac{1}{k_2}$. Except for $k_1=k_2=0$, any path from $s_{k_1}$ to $s_{k_2}$ will consist of a path from $s_{k_1}$ to $-1$ and then followed by a continued fraction block of length $k_2-1$ from $-1$ to $-\frac{1}{k_2}$. We call the first part $P_1$ and the second part of the path $P_2$. Decorations on $P_2$ correspond to $k_2$ tight contact structures. For $k_1=0$ and $k_2\geq 1$, there is a continued fraction block of length $k_2$ between $\infty$ and $-\frac{1}{k_2}$. This gives us $k_2+1$ tight contact structures. On the other hand, for $k_1=1$ and $k_2\geq 1$  this path consists of a continued fraction block of length $p$ followed by $P_2$. This will correspond to $(p+1)k_2$ tight contact structures. For $k_1\geq 1$, there is a single jump from $s_{k_1}$ to $-p$, followed by a continued fraction block of length $p-1$ from $-p$ to $-1$ and then $P_2$. Decorations on this path correspond to $2pk_2$ contact structures.  These can 
 also be obtained by switching the role of $L_1$ and $L_2$ as this case the Hopf link is symmetric.

 The rational Thurston-Bennequin invariant of the components are given by \[\tb_{\mathbb{Q}}(L_1^{s_{k_1}})=k_1+\frac{1}{p}\] and  \[\tb_{\mathbb{Q}}(L_1^{s_{k_2}})=-k_2+\frac{1}{p}.\] The rational rotation number are same as the previous case after switching the components.
 \begin{remark}
     We need to be careful here as there is a shift in the Thurston Bennequin number. After swiching $L_1$ and $L_2$ one needs to replace $ k_1$ by $k_2$ but $k_2$ by $k_1+1$ to adjust the range. 
 \end{remark}

It is easy to see that in this case $L_1$ is non-loose and $L_2$ is loose for every $k_1\geq 0, k_2\geq 1$. Stabilizations of $L_1$ will be exactly same as stabilizations $L_2$ as before after switching their roles.  A detailed analysis of the stabilizations of $L_1$ is given in Theorem 1.1 \cite{Chatterjee-Etnyre-Min_2025_existence} as well.

 \subsection{Case 3 large slope $\cup$ large slope} The dividing slopes in this case are given by $s_{k_1}=-\frac{k_1p+1}{k_1}$ and $s_{k_2}=\frac{1}{k_2}$. %Any path from $s_{k_1}$ to $s_{k_2}$ will consist of a path from

 For $k_1=k_2=0$, we will see that the complementary $T^2\times I$ has both boundary slopes $\infty$. As $T^2\times I$ is minimally twisting, this must be an $I$-invariant contact structure and thus gives a unique tight contact structure. This is the unique Hopf link with both components non-loose. The Thurston--Bennequin invariants are given by $\tbr(L_i)=\frac{1}{p}$ and rotation number $\rotr(L_i)=0.$ The Euler class of this contact structure is zero as well. As before any stabilization of $L_1^0$ or $L_2^0$ is loose. $S_\pm(L_1^0)\sqcup L_2^0$ is non-loose and corresponds to the $2$ Hopf links (loose/non-loose pairing) from case~\ref{ssec:L(p,1)case1}. In fact, $S_\pm^k(L_1)\sqcup L_2^0$ is isotopic to the Hopf link $L_1^{k-1}\sqcup L_2^0$ from case~\ref{ssec:L(p,1)case1} for $k\geq 1$ where $S_\pm^k(L_1)$ denotes the $k-$fold stabilization. Similarly, $L_1^0\sqcup S_\pm^k(L_2)$ is isotopic to the non-loose Hopf link $L_1^{0}\sqcup L_2^{k}$ from case~\ref{ssec:L(p,1)case2} for $k\geq 1.$ Finally, $S_+(L_1^0)\sqcup S_+(L_2^0)$ is still non-loose but has a half Giroux torsion in the complement. The same is true when both the components are negatively stabilized. The mixed stabilization of the components (i.e. if $L_1$ is stabilized postively and $L_2$ is stabilized negatively) will loosen the link.

\begin{table}
    \centering
    \begin{tabular}{|c|c|c|c|c|}
    \hline
    \hline
         $k_1$ &$k_2$&$\rotr(L_1)$&$\rot(L_2)$ &\text{range of}\ $m$ \\
         \hline
         $0$&$1$ &$\mp\frac{2}{p}$ &$\mp\frac{p+2}{p}$ &\\
         \hline
         $0$ &$2$ &$\frac{2m-2}{p}$&$\frac{(2m-2)(p+1)}{p}$ &$0\leq m\leq 2$\\
         \hline
         $0$ &$>1$ &$\mp\frac{1}{p}\mp\frac{1}{p}$& $\mp\frac{p+1}{p}\mp((k_2-1)+\frac{1}{p})$&\\
         \hline
         $1$ &$1$ &$-\frac{p+2-2m}{p}$ &$-\frac{p+2-2m}{p}$&$0\leq m\leq p+2$\\
         \hline
         $1$ &$>1$ &$-\frac{p+1-2m}{p}\mp\frac{1}{p}$&$-\frac{p+1-2m}{p}\mp\frac{(1+p(k_2-1))}{p}$ &$0\leq m\leq p+1$\\
         \hline
         $>1$ &$>1$ &$\mp((k_1-1)+\frac{1}{p})-\frac{p-2m}{p}\mp\frac{1}{p}$&$\mp\frac{1}{p}-\frac{p-2m}{p}\mp((k_2-1)+\frac{1}{p})$&$0\leq m\leq p$\\
         \hline
         \hline
         
    \end{tabular}
    \caption{Rotation numbers for the components with $\tb(L_i)=k_i+\frac{1}{p}$}
    \label{tab:rotationL(p,1)}
\end{table}
 
 Now for $k_2=0$ and $k_1=1$, there is exactly one edge between $-p-1$ and $\infty$ and thus corresponds to $2$
 tight contact structure. For $k_2=0, k_1>1$, the path will consist of one jump from $s_{k_1}$ to $-p$ followed by another jump from $-p$ to $\infty$.
 For $k_1=2$ this is a continued fraction block of length $2$ so we will have $3$ tight contact structures in the complement. For $k_1>2$, this is not a continued fraction block, thus corresponds to $4$ tight contact structures. It is easy to check that both the components are loose. 

 The rational Thurston--Bennequin invariants of the components are given by \[\tbr(L_i)=k_i+\frac{1}{p}\]  and the rational rotation numbers are given in Table~\ref{tab:rotationL(p,1)}.

Next we analyze the loose mountain range for $L_1$ fixing $L_2^0$. The two loose $L_1^0$ are denoted as $(L_1^0)^\pm$ respectively. $S_\pm((L_1^0)^\pm)\sqcup (L_2^0)^{\pm}$ are non-loose but will have a half Giroux torsion in the complement. But $S_\pm((L_1^0)^\mp)\sqcup L_2^0$ are loose. For $k_1=1,$ we denote the corresponding $L_1^1$ as $(L_1^1)^{0}, (L_1^1)^{1} $ and $(L_1^1)^{2}$ where the superscript denotes the number of negative slices in the path from $s_{k_1}$ to $\infty$ which is a continued fraction block of length $2$. $S_+(((L_1^1)^{0}))$ and $S_-((L_1^1)^{2})$ coincide with $(L_1^0)^\pm$ respectively. On the other hand, $S_\pm(((L_1^1)^{+-}))\sqcup L_2^0$ will be loose. Finally, for $k_1=2,$ we have $2$ jumps which is not a continued fraction block. So, we denote the corresponding $L_1^2$ as $(L_1^2)^{ij}$,  where $i$ is the sign of the basic slice from $s_{k_1}$ to ${-p}$ and $j$ is the sign of basic slice from $-p$ to $\infty$.  $S_+((L_1^2)^{++})\sqcup L_2^0$ and  $S_-((L_1^2)^{--})\sqcup L_2^0$  are contactomorphic to $(L_1^1)^{0}\sqcup L_2^0$ and $(L_1^1)^{2}\sqcup L_2^0$ respectively. Moreover, $S_+((L_1^2)^{+-})\sqcup L_2^0$ and $S_-((L_1^2)^{-+})\sqcup L_2^0$ are both contactomorphic to $(L_1^1)^{+-}\sqcup L_2^0.$ So fixing $L_2^0$, putting all these together, we have 1 forward slash, 1 back slash based at $(\frac{p+2}{p}, 1+\frac{1}{p})$ and $(-\frac{p+2}{p}, 1+\frac{1}{p})$ respectively and one $V$ based at $(0, 2+\frac{1}{p}).$ 
 
 A similar analysis of the paths can be done fixing $k_2=1$ and $k_2>1$. The number of non-loose Hopf links are given in Table \ref{tab:table_L(p,1)}. Analyzing the stabilizations as before we see that for fixed $L_2$ with $k_2=1$ we have 1 forward slash, 1 backward slash based at $(\frac{2}{p}, \frac{1}{p})$ and $(-\frac{2}{p}, \frac{1}{p})$ respectively and $p+1$ Vs based at $(-\frac{p+2-2m}{p}, 1+\frac{1}{p})$ where $1\leq m\leq p+1$ for $L_1$. Here $m$ denotes the number of negative basic slice in the length $p+2$ continued fraction block from $-p-1$ to $1$. For fixed $L_2$ with $k_2>1,$ there are $2$ forward slashes based at $(\frac{2}{p},\frac{1}{p}), (0,\frac{1}{p}) $ and $2$ backward slashes based at $(-\frac{2}{p},\frac{1}{p}), (0,\frac{1}{p})$ respectively and $2p$ $V$s based at $(-\frac{p-2m}{p}\mp\frac{1}{p}, 1+\frac{1}{p})$ where $1\leq m\leq p$ for $L_1$. In this case $m$ denotes the number of negative basic slices in the path from $-p-1$ to $0$ and the $\mp$ sign corresponds to the $\pm$ basic slice from $0$ to $\frac{1}{k_2}$. The Euler classes can be computed as described in Section~\ref{sec:background}.

 One could do a similar analysis fixing $L_1$ with $k_1=0,1$ and $k_1>1$ to get the exact same loose mountain range for $L_2.$

 %Note that for $k_1=0$ and $k_2=0$, both the components are non-loose. Any stablizations of $L_1^0$ or $L_2^0$ are loose. This in the only lens space where we see a pair of non-loose components which are not coming from any non-loose stabilizations and also all of their stabilizations are loose. For any other pair of values of $k_i$, both the components are loose. \textcolor{red}{Do we need to talk about stablizations of the loose components? If yes, how is that possible?} \textcolor{olive}{ Yes. the loose loose components are related with each other via stabilizations. They do not have any Giroux torsion in the complement for $k_i\geq 1$. If we consider the candidate with $k_1=0, k_2=0$ and stabilize the components, the link will be no longer strongly exceptional any more. if will have a full Giroux torsion.}

\end{proof}
\begin{proof}{Proof of Theorem \ref{thm:L(2n+1,2)}} %We prove this theorem for $n\geq 2$. The proof for $L(3,2)$ is slightly different but can be done using similar techniques. 
Notice that, for $L(2n+1,2)$, $(-\frac{2n+1}{2})^a=-n-1$ and $(-\frac{2n+1}{2})^c=-n$. Like before, we will subdivide the cases in three subcases and work with them separately.
\subsection{Case 1: small $\cup$ large slope} In this case,  the dividing slopes of the standard neighborhood $L_1$ and $L_2$ are given by $s_{k_1}=-\frac{n(1+2k_1)+k_1}{1+2k_1}$ and $s_{k_2}=\frac{1}{k_2}.$ For $k_2=0$, any path from $s_{k_1}$ to $s_{k_2}=\infty$ consists of a continued fraction block of length $k_1$ from  $s_{k_1}$ to $-n$ and then a jump from $-n$ to  $\infty$ . This corresponds to $2(k_1+1)$ tight contact structures. For $k_2=1$, a path from $s_{k_1}$ to $s_{k_2}=1$ consists of a continued fraction block of length $k_1$ from $s_{k_1}$ to $-n$ followed by another continued fraction block of length $n+1$ from $-n$ to $1$ and thus corresponds to $(n+2)(k_1+1)$ tight contact structures. For $k_2>1$, we have a continued fraction block of length $k_1$ from $s_{k_1}$ to $-n$, followed by a continued fraction block of length $n$ from $-n$ to $0$ and finally an edge from $0$ to $\frac{1}{k_2}.$ Decorations on this path correspond to $2(n+1)(k_1+1)$ tight contact structures. It is easy to see that $L_1$ is loose and $L_2$ is non-loose.

The rational Thurston--Bennequin invariant of the components are given by \[\tb_\mathbb{Q}(L_1^{k_1})=-(k_1+1)+\frac{n+1}{2n+1}\] and  \[\tb_\mathbb{Q}(L_1^{k_2})=k_2+\frac{2}{2n+1}\] One can easily calculate the rotation number for the $2(k_1+1)$ candidates as $\rot_{\mathbb{Q}}(L_1^{k_1})=(-k_1+2m)\mp \frac{n+1}{2n+1}$  where $m=0,1.\cdots k_1$ and $\rot(L_2^0)=\mp\frac{1}{2n+1}$. The rotation numbers for the $(n+2)(k_1+1)$ candidates are $\rot_{\mathbb{Q}}(L_1^{k_1})=(-k_1+2m_1)-\frac{(-n-1+2m_2)}{2n+1}$ and $\rot_{\mathbb{Q}}(L_2^1)=\frac{2(-n-1+2m_2)}{2n+1}$ where $m_1=0,1, \cdots k_1$ and $m_2=0,1,\cdots n+1$. Finally, $\rot_{\mathbb{Q}}(L_1^{k_1})=(-k_1+2m_1)+\frac{(-n+2m_2\mp 1)}{2n+1}$ and $\rot_{\mathbb{Q}}(L_2^{k_2})=\frac{2(-n+2m_2)}{2n+1}\mp(k_2-\frac{2n-1}{2n+1})$ where $m_1=0,1, \cdots k_1$ and $m_2=0,1,\cdots n$ for the $2(n+1)(k_1+1)$ non-loose representatives. Here $m,m_1,m_2$ counts the number of negative signs in the corresponding continued fraction blocks as shown above.

As before $L_1^{k_1}$ is loose and $L_2^{k_2}$ is non-loose for all $k_1, k_2\geq 0$. Any stabilization of $L_2^0$ is loose as observed in other cases. For $L_2^1$, there are $(n+2)(k+1)$ non-loose representatives. Now note that, the complement of these candidates consist of a thickened torus with boundary slope $s_{k_1}$ and $s_{k_2}=1$. We could break this path into a path from $s_{k_1}$ to $-n$ and then from $-n$ to $1$. The last part in a continued fraction block of length $(n+2)$, let us call this $P$. Now stabilizing $L_2^1$ corresponds to adding a basic slice of slopes $\{1,\infty\}$ to this complement. Note that, we will only get a consistent shortening after adding a positive (resp. negative) basic slice if the path $P$ has only '+' (resp. '-') signs. Thus we will have $2(k_1+1)$ stabilizations that are non-loose. These coincides with $L_2^0.$
Now for $L_2^2$,  the complement contains a thickened torus of slopes $s_{k_1}$ and $\frac{1}{2}$. We can break this path into a path from $s_{k_1}$ to $-1$, then a jump from $-1$ to $0$ and finally another jump from $0$ to $\frac{1}{2}$.  Adding a basic slice of slope $\{\frac{1}{2},1\}$ to the complement will shorten the path. If the edge between $0$ and $\frac{1}{2}$ has a positive sign then a positive stabilization will be non-loose and  if it has negative sign then the negative stabilization will be non-loose. Also, notice that after adding the basic slice we will now have a continued fraction block of length $n+1$ and thus there will be $(n+2)(k_1+1)$ non-loose stabilizations that coincide with $L_2^1$. For $k_2>2$, the shortening gives us an edge between $0$ to $\frac{1}{k_2-1}$ which is not a part of continued fraction block and thus we will have $2(k_1+1)(n+1)$ non-loose stabilizations that coincide with $L_2^{k_2-1}.$ As a pair all of these non-loose Hopf links will be distinguished by their rotation numbers.

Putting all these together we will have: fixing the loose $L_1$, the non-loose mountain range for $L_2$ is given by one forward slash based at $(\frac{1}{2n+1}, \frac{2}{2n+1})$, one back slash based at $(-\frac{1}{2n+1}, \frac{2}{2n+1})$ and $n$ $V$s based at $(\frac{2(-n-1+2m)}{2n+1}, 1+\frac{2}{2n+1})$ for $m=1,\cdots, n$. The forward and back slash live in Euler class $\pm (n+1)$ and the $V$s live in Euler class $(-n-1+2m)$ for $m=1,2,\cdots n$ ( we use the $0$ sloped disk to evaluate the Euler class). A further careful observation of the pairwise rotation numbers give us the following pairing information:

 For non-loose $L_1^{k_1}\sqcup L_2^0$, the corresponding vertices ($L_2$) from the forward and back slashes pair with a loose cone ($L_1$) peaked at $(\pm\frac{n+1}{2n+1}, -1+\frac{n+1}{2n+1})$ to give non-loose realizations of Hopf links. For $L_1^{k_1}\sqcup L_2^1$, the corresponding vertices from the forward and back slashes pairs with the same loose cones as before % based at $(\pm\frac{n+2}{2n+1},-1+\frac{n+1}{2n+1})$ 
    and each based vertex of the $V$s pairs with loose cone peaked at $(\frac{2(-n-1+2m)}{2n+1}, -1+\frac{n+1}{2n+1})$ for $m=1,2,\cdots, n$. Finally for $L_1^{k_1}\sqcup L_2^{k_2}$, the corresponding vertex from the back and forward slashes pair with the same loose cones, the $n$ vertices from the right wings of the $n$ $V$s pair up with loose cones based at $(\frac{-n+2m}{2n+1}+\frac{1}{2n+1}, -1+\frac{n+1}{2n+1})$ for $m=0,\cdots, n-1$, and the $n$ vertices from the left wings of the $n$ $V$s pair up with the loose cones based at $(\frac{-n+2m}{2n+1}-\frac{1}{2n+1}, -1+\frac{n+1}{2n+1})$ for $m=1,2,\cdots, n.$

    %  For $k_2=0$, the correspondence vertex ($L_2$) from forward and back slashes pairs with a loose cone ($L_1$) peaked at $(\pm\frac{n+1}{2n+1}, -1+\frac{n+1}{2n+1})$ to give non-loose realizations of Hopf links. For $k_1=1$, the corresponding vertices from the forward and back slashes pairs with the same loose cones as before  based at $(\pm\frac{n+2}{2n+1},-1+\frac{n+1}{2n+1})$ 
    %and each based vertex of the $V$'s pairs with loose cone peaked at $(\frac{2(-n-1+2m)}{2n+1}, -1+\frac{n+1}{2n+1})$ for $m=1,2,\cdots, n$. Finally for $k_2>1$, the corresponding vertex from the back and forward slashes pair the the same loose cones, the $n$ vertices from the right wings of the $n$ $V$s pair up with loose cones based at $(\frac{-n+2m}{2n+1}-\frac{1}{2n+1}, -1+\frac{n+1}{2n+1})$, and the $n$ vertices from the left wings of the $n$ $V$s pair up with the loose cones based at $(\frac{-n+2m}{2n+1}+\frac{1}{2n+1}, -1+\frac{n+1}{2n+1})$ for $m=1,2,\cdots, n.$

\subsection{Case 2: large slope $\cup$ small slope} The dividing slopes on the standard neighborhoods of $L_1$ and $L_2$ are given by $s_{k_1}=-\frac{n(1+2k_1)+k_1+1}{1+2k_1}$ and $s_{k_2}=-\frac{1}{k_2}$. Note that, $k_1\geq 0$ and $k_2=0$ overlaps with the next case.
 %For $ k_1=0$, there is exactly one edge between $-n-1$ to $\infty$ and thus gives $2$ tight contact structures and thus $2$ non-loose representatives (corresponding to the positive and negative basic slice) whose rotation numbers are $\rotr(L_1^0)=\mp\frac{n+2}{2n+1} \text{and} \rotr(L_2^0)=\mp\frac{3}{2n+1}$.  For $k_1=1$, a path from $-\frac{3n+2}{3}$ to $\infty$ consists of a continued fraction block of length $3$ and thus this corresponds to $4$ tight contact structures. For $k_1>1$, a path from $s_{k_1}$ to $\infty$ consists of a jump from $s_{k_1}$ to $-\frac{2n+1}{2}$, then a continued fraction block of length $2$ from $-\frac{2n+1}{2}$to $\infty$. Thus we will have $6$ tight contact structures. It is easy to check that all of these candidates are componentwise loose. The rotation numbers of the $4$ candidates are $ \rotr(L_1^1)=\frac{(n+1)(-3+2m)}{2n+1}$ and  $\rotr(L_2^0)=\frac{-3+2m}{2n+1}$ where $m=0,1,2,3.$ The rotation numbers of the $6$ loose/loose pairs are given by $\rotr(L_1^{k_1})=\pm(k_1-\frac{n}{2n+1})+\frac{(-2+2m)(n+1)}{2n+1}$ and $\rotr(L_2^0)=\frac{\pm1+(-2+2m)}{2n+1}$ for $m=0,1,2.$ These candidates constitutes $1$ loose forward slash, $1$ loose backward slash and $2$ loose $V$'s. 

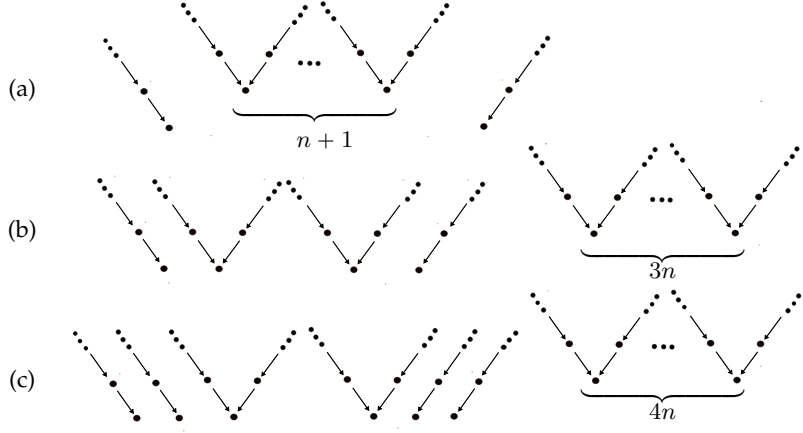
\begin{figure}
    \centering
   
   \labellist
\small\hair 2pt
\pinlabel{(a)} at -100 680
\pinlabel {$\underbrace{\hspace{7 em}}$} at 480 630
\pinlabel {$n+1$} at 500 580
%\pinlabel{$k_2=0$} at 630 600
%\pinlabel {$k_2=1$} at 630 680
%\pinlabel {$k_2>2$} at 630 750
%%%%%%%%%%%%%%%%%%%%%%%%%%%%%%%%%%%%%%%%
\pinlabel {(b)} at -100 400
\pinlabel{$\underbrace{\hspace{7 em}}$} at 1170 350
\pinlabel{$3n$} at 1170 320
%\pinlabel {$k_2=0$} at 
%%%%%%%%%%%%%%%%%%%%%%%%%%%%%%%%%%%%%%%%%
\pinlabel {(c)} at -100 100
\pinlabel{$\underbrace{\hspace{7 em}}$} at 1170 70
\pinlabel {$4n$} at 1170 40
\endlabellist

    \includegraphics[scale=0.2]{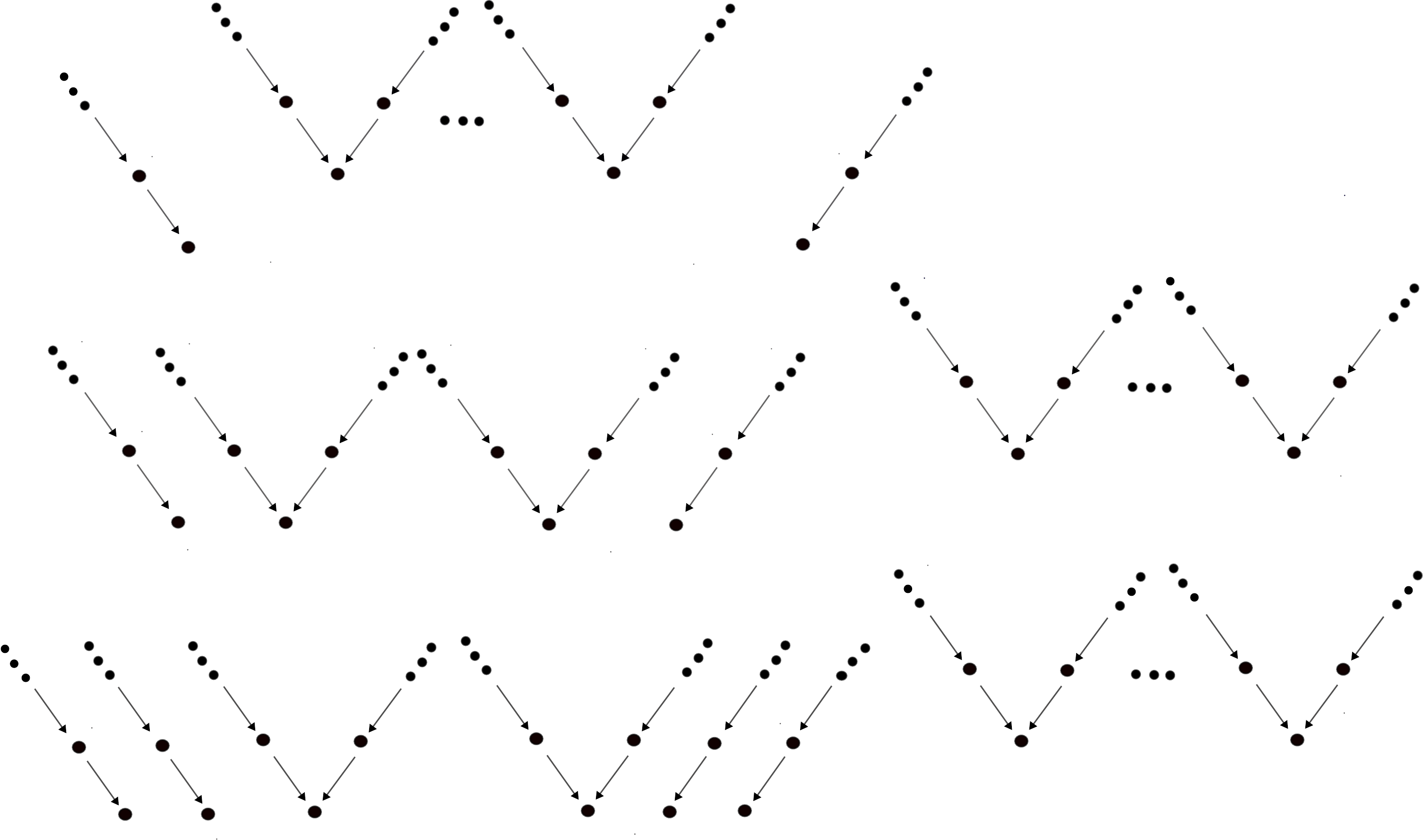}
    \caption{Fix $K_1$ in $L(2n+1,2)$. (a), (b) and (c) show loose mountain range for $K_2$ when $k_1=0, 1,>1$ respectively.}
    \label{fig:looseL(2n+1,2)K1fixed}
\end{figure}

 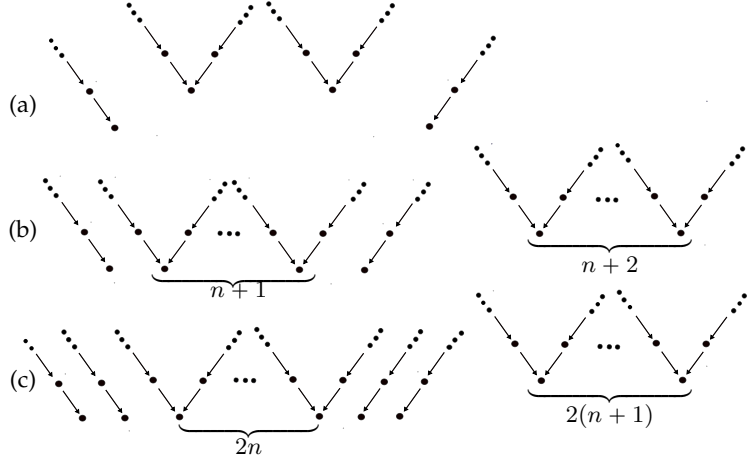
\begin{figure}[!htbp]

 \labellist
\small\hair 2pt
\pinlabel{(a)} at 0 650

%\pinlabel{$k_2=0$} at 630 600
%\pinlabel {$k_2=1$} at 630 680
%\pinlabel {$k_2>2$} at 630 750
%%%%%%%%%%%%%%%%%%%%%%%%%%%%%%%%%%%%%%%%
\pinlabel {(b)} at 0 400
\pinlabel {$\underbrace{\hspace{7 em}}$} at 420 300
\pinlabel{$\underbrace{\hspace{7 em}}$} at 1170 370
\pinlabel{$n+1$} at 430 280
\pinlabel{$n+2$} at 1170 330
%\pinlabel {$k_2=0$} at 
%%%%%%%%%%%%%%%%%%%%%%%%%%%%%%%%%%%%%%%%%
\pinlabel {(c)} at 0 100
\pinlabel{$\underbrace{\hspace{6 em}}$} at 450 0
\pinlabel $2n$ at 450 -30
\pinlabel{$\underbrace{\hspace{7em}}$} at 1170 70
\pinlabel {$2(n+1)$} at 1170 30
\endlabellist

 \centering
    \includegraphics[scale=0.2]{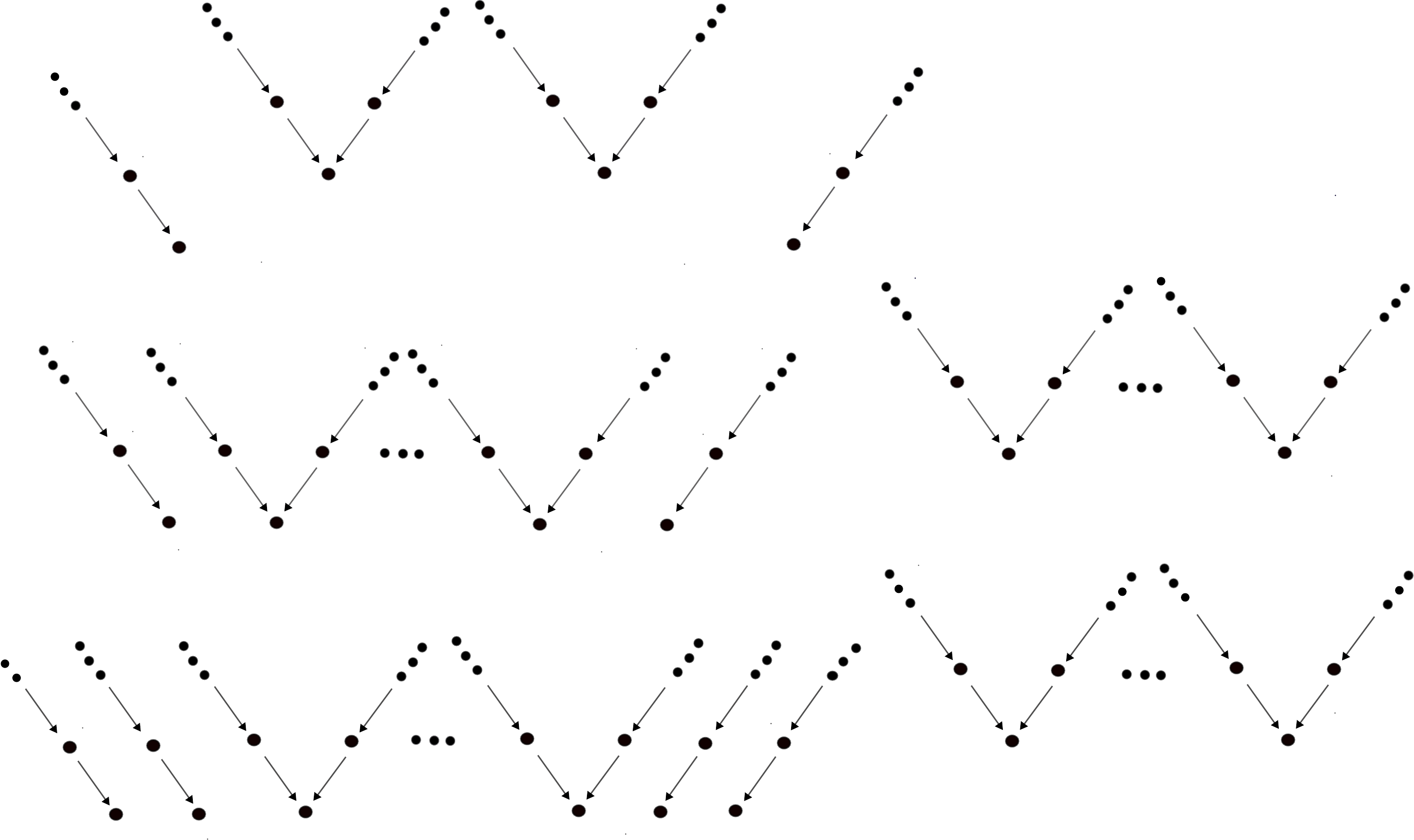}
    \caption{Fix $K_2$. (a), (b) and (c) shows loose mountain range for $K_1$ when $k_2=0, 1,>1$ respectively.}
    \label{fig:looseL(2n+1,2)K2fixed}
\end{figure}

So we assume $k_2\geq 1$. For $k_1=0$, there is a continued fraction block of length $n$ from $-n-1$ to $-1$ followed by a continued fraction block of length $k_2-1$ from $-1$ to $-\frac{1}{k_2}$, thus corresponds to $(n+1)k_2$ tight contact structures. For $k_1=1$, a path from $s_{k_1}$ consists of a continued fraction block of length $2$ from $s_{k_1}=-\frac{3n+2}{3}$ to $-n$, followed by a continued fraction block $P_1$ of length $n-1$ from $-n$ to $-1$ and then another continued fraction block $P_2$ of length $k_2-1$ from $-1$ to $-\frac{1}{k_2}.$ This corresponds to $3nk_2$ tight contact structures. Finally, assume $k_1>1$. In this case a path from $s_{k_1}$ to $s_{k_2}$ consists of a jump from $s_{k_1}$ to $-\frac{2n+1}{2}$, then another jump from $-\frac{2n+1}{2}$ to $-n$, followed by $P_1\cup P_2$. This gives us $4nk_2$ tight contact structures. Observe that, in all the cases $L_1$ is non-loose and $L_2$ is loose.

Next we calculate the classical invariants. The rational Thurston--Bennequin invariant of the components are given by 
\[\tbr(L_1^{k_1})=k_1+\frac{n+1}{2n+1}\] and \[\tbr(L_2^{k_2})=-k_2+\frac{2}{2n+1}.\]
 The rotation numbers are given as $\rotr(L_1^0)=\frac{(-n+2m_1)}{2n+1} \ \text{and} \rotr(L_2^{k_2})=\frac{2(-n+2m_1)}{2n+1}+(-k_2+1+2m_2)$ for $k_2\geq 1, m_1=0,1,\cdots n$ and $m_2=0,1,\cdots, k_2-1$. $m_1$ denotes the number of negative basic slice in the path from $-n-1$ to $-1$ and $m_2$ denotes the number of negative basic slices in the path from $-1$ to $-\frac{1}{k_2}$.  For $k_1=1, k_2\geq 1$, the rotation numbers are $\rotr(K_1^1)=\frac{(n+1)(-2+2m_1)}{2n+1}+\frac{-n+1+2m_2}{2n+1}$ and $\rotr(K_2^{k_2})=\frac{-2+2m_1}{2n+1}+\frac{2(-n+1+2m_2)}{2n+1}+(-k_2+1+2m_3)$ where $m_1=0,1,2;\  m_2=0,1,\dots n-1$ and $m_3=0,1,\cdots k_2-1$. $m_1$ denotes the number of negative basic slice from $s_{k_1}$ to $-n$, $m_2$ denotes the negative signs  in $P_1$ and $m_3$ denotes the negative signs in path $P_2$. Finally for $k_1>1$ the rotation numbers are $\rotr(L_1^{k_1})=\mp(k_1-\frac{n}{2n+1})\mp\frac{n+1}{2n+1}+\frac{-n+1+2m_2}{2n+1}$ and $\rotr(L_2^{k_2})=\frac{\mp1\mp1+2(-n+1+2m_2)}{2n+1}+(-k_1+1+2m_3)$ where $m_1=0,1,\cdots n-1, m_2=0,1,\cdots k_2-1$. $m_2$ and $m_3$ counts the number of negative basic slices in the paths $P_1$ and $P_2$ as before. The Euler classes can be computed using the techniques mentioned before. 

Any stabilization of $L_1^0$ is loose, but if we consider $L_1^0$ having all positive signs in the complement then $S_+(L_1^0)\sqcup L_2^{k_2}$ will still be non-loose but will have half-Giroux torsion in the complement. Same is true for $S_-(L_1^0)\sqcup L_2^{k_2}$  when we consider $L_1^0$ with only negative sign. All other cases, any stabilization will loosen the link. The other stabilizations of $L_1^{k_1}$ are exactly same as in Theorem 1.14 in \cite{Chatterjee-Etnyre-Min_2025_existence}. So we omit the details here.

Putting everything together we have the following: Fixing the loose $L_2$, the non-loose realizations of $L_1$ form $1$ forward slash based at $(\frac{n}{2n+1},\frac{n+1}{2n+1})$, one back slash based at $(-\frac{n}{2n+1},\frac{n+1}{2n+1})$, $n-1$ $V$s based at $(\frac{(-n+2m_1)}{{2n+1}}, \frac{n+1}{2n+1})$ where $m_1=1,\cdots n-1$ and $n$ $V$s based at $(\frac{(-n+1+2m_2)}{2n+1}, 1+\frac{n+1}{2n+1})$ for $m_2=0,1,\cdots n-1.$ 

For non-loose $L_1^0\sqcup L_2^{k_2}$ each of the based vertex from the forward slash, the $n-1$ $V$ and the back slash pairs with a loose cone peaked at $(\frac{2(-n+2m_1)}{2n+1},-1+\frac{2}{2n+1})$ for $m_1=0,1,\cdots n$ respectively. 
For the non-loose $L_1^1\sqcup L_2^{k_2}$the corresponding vertex from the forward and back slash pair with loose cones peaked at $(\pm\frac{2n}{2n+1}, -1+\frac{2}{2n+1})$ respectively, the corresponding vertices from the right wings of  $n-1$ $V$s pair with the loose cones based at $(\frac{2}{2n+1}+\frac{2(-n+1+2m_1)}{2n+1},-1+\frac{2}{2n+1})$ respectively for $m_1=0,\cdots n-2$, the corresponding vertices from the left wings of  $n-1$ $V$s pair with the loose cones based at $(-\frac{2}{2n+1}+\frac{2(-n+1+2m_1)}{2n+1},-1+\frac{2}{2n+1})$ for $m_1=1,\cdots n-1$,  the based vertices of the $n$ $V$s pair up with loose cone peaked at $(\frac{2(-n+1+2m_2)}{2n+1},-1+\frac{2}{2n+1})$ for $m_2=0,1,\cdots, n-1$.  Finally, for the non-loose $L_1^{k_1}\sqcup L_2^{k_2}$ for $k_1>1$,the corresponding vertices from the forward and back slashes pair up with the loose cones based at $(\pm\frac{2n}{2n+1},-1+\frac{2}{2n+1})$, the $n-1$ vertices from the right wing of the $n-1$ $V$s pair with $n-1$ loose $L_2$ cones peaked at $(\frac{ 2}{2n+1}+\frac{2(-n+1+2m_1)}{2n+1}, -1+\frac{2}{2n+1})$ for $m_1=0,\cdots, n-2$ and the $n-1$ vertices from the left wing of $n-1$ $V$s pair with $n-1$ loose cones peaked at $(-\frac{ 2}{2n+1}+\frac{2(-n+1+2m_1)}{2n+1}, -1+\frac{2}{2n+1})$ for $m_1=1,\cdots, n-1$. The $n$ corresponding vertices from the left wing of the $n$ $V$s pair up with a loose cone peaked at $(\frac{2(-n+1+2m_1)}{2n+1},-1+\frac{2}{2n+1})$ $m_1=0,\cdots n-1$, same for the $n$ candidates of the right wings of the $n$ $V$s. All of the Hopf links can be distinguished by their pairwise rotation numbers.

\subsection{Case 3: large slope $\cup$ large slope}The dividing slopes on the standard neighborhoods of $L_1$ and $L_2$ are given by $s_{k_1}=-\frac{n(1+2k_1)+k_1+1}{1+2k_1}$ and $s_{k_2}=\frac{1}{k_2}$. 

First we assume $k_2=0$. For $ k_1=0$, there is exactly one edge between $-n-1$ to $\infty$ and this gives $2$ tight contact structures and thus $2$ non-loose representatives (corresponding to the positive and negative basic slice) whose rotation numbers are $\rotr((L_1^0)^\pm)=\mp\frac{n+2}{2n+1} \text{and} \rotr((L_2^0)^\pm)=\mp\frac{3}{2n+1}$. We denote the two $L_1^0$ as $L_1^\pm$.  For $k_1=1$, a path from $-\frac{3n+2}{3}$ to $\infty$ consists of a continued fraction block of length $3$ and thus this corresponds to $4$ tight contact structures. These $4$ $L_1^1$ we denote as $(L_1^1)^m$ where $m$ denotes the negative basic slices in the above path.  For $k_1>1$, a path from $s_{k_1}$ to $\infty$ consists of a jump from $s_{k_1}$ to $-\frac{2n+1}{2}$, then a continued fraction block of length $2$ from $-\frac{2n+1}{2}$to $\infty$. Thus we will have $6$ tight contact structures. These $6$ candidates are $({L_1^{k_1}})^{\pm,m}$ where $\pm$ sign in the sign of the first jump and $m$ counts the number of negative signs in the later path. The rotation numbers of the $4$ candidates are $ \rotr((L_1^1)^m)=\frac{(n+1)(-3+2m)}{2n+1}$ and  $\rotr((L_2^0)^m)=\frac{-3+2m}{2n+1}$ where $m=0,1,2,3$. The rotation numbers of the $6$ loose/loose pairs are given by $\rotr((L_1^{k_1})^{\pm,m})=\mp(k_1-\frac{n}{2n+1})-\frac{(2-2m)(n+1)}{2n+1}$ and $\rotr((L_2^0)^{\pm,m})=\mp\frac{ 1}{2n+1}-\frac{(2-2m)}{2n+1}$ for $m=0,1,2$. It is easy to check that all of these candidates are componentwise loose. The rational Thurston Bennequin invariants are \[\tbr(L_1^{k_1})=k_1+\frac{n+1}{2n+1}\] and \[\tbr(L_2^0)=\frac{2}{2n+1}.\] 

Clearly, for $m=0$, $S_+((L_1^1)^0)\sqcup (L_2^0)^0$ is non-loose as adding a basic slice in the complement leads to a consistent shortening in the $T^2\times [0,1]$ and thus give us a  non-loose link. This coincides with the link $(L_1^0)^+\sqcup (L_2^0)^+$. Same is true when we negatively stabilize $(L_1^1)^3$. These two loose components $(L_1^1)^0$ and $(L_1^1)^3$ are parts of the back and forward slashes.  For $m=1,2$, any stabilization of $(L_1^1)^m$ will loosen the whole link. Thus these candidates form the base of the two $V$s. Now when we stabilize $(L_1^2)^{\pm, m }$ we see that a positive stabilization of $(L_1^2)^{+, m }$ will coincide with $(L_1^2)^{m-1}$ and a negative stabilization of $(L_1^1)^{-, m }$ coincides with $(L_1^1)^{m+1 }$.

So putting everything together, we see that fixing the loose $L_2$ with $k_2=0$ the loose $L_1$ constitutes $1$ loose forward slash based at $(\frac{n+2}{2n+1}, \frac{n+1}{2n+1})$, $1$ loose back slash based at $(-\frac{n+2}{2n+1}, \frac{n+1}{2n+1})$ and $2$ loose $V$s based at $(\mp\frac{n+1}{2n+1}, 1+\frac{n+1}{2n+1})$. 

If one fixes $L_2$ with $k_2=1$, for $k_1=0$ we will see a continued fraction block of length $n+2$ from $-n-1$ to $1$, for $k_2=1$, we have one continued fraction block of length $2$ from $-\frac{3n+2}{3}$ to $-n$ and another continued fraction block of length $n+1$ from $-n$ to $1$. Finally for $k_1>1$, we have one jump from $s_{k_1}$ to $-\frac{2n+1}{2}$, another jump from $-\frac{2n+1}{2}$ to $-n$ and then the same path as before. So we will have $n+3$, $3(n+2)$ and $4(n+2)$ non-loose Hopf links respectively. The rotation numbers of the components are given in Table~\ref{tab:L(2n+1,2)}. We now look at the stabilizations of $L_1$. Like before we denote the $(n+3)$ $L_1^0$ as $(L_1^0)^m$ where $m$ denote the number of negative signs in the length $n+2$ path.  For $k_1=1$, we denote the $3(n+2)$ candidates as $(L_1^1)^{m_1, m_2}$ where $m_1$ counts the number of negative signs in the path from $-\frac{3n+2}{3}$ to $-n$ and $m_2$ counts the number of negative signs in the next $n+1$ length path. Note that, $S_+((L_1^1)^{0,0}$ coincides with $(L_1^0)^0$ and $S_-((L_1^1)^{2,n+1}$ coincides with $(L_1^0)^3$. These form the base of the back and the forward slash respectively. The $S_+((L_1^1)^{0,m_2}$ and $S_-((L_1^1)^{2,m_2}$ will form the bases of the $n+1$ $V$s for $m_2=1,\cdots n+1$. Any stabilization of $(L_1^1)^{1, m_2}$ will loosen the whole link. Thus $(L_1^1)^{1, m_2}$ are the bases of the $n+2$ $V$s.

Putting everything together, we have the following loose mountain range for $L_1$ fixing $L_2$ with $k_2=1$, \  $1$ forward and $1$ back slash based at $(\pm\frac{n+2}{2n+1}, \frac{n+1}{2n+1})$ respectively , $n+1$ loose $V$ s based at $(\frac{-n-2+2m_1}{2n+1},\frac{n+1}{2n+1})$ and $n+2$ loose $V$s based at $(\frac{-n-1+2m_2}{2n+1},1+\frac{n+1}{2n+1})$ where $m_1=0,\cdots n$ and $m_2=0,1,\cdots n+1$.

One can fix $L_2$ with $k_2>1$ and can do a similar analysis to see how the mountain range changes for $L_1.$ Check (c) of Figure~\ref{fig:looseL(2n+1,2)K2fixed}. The mountain range for $L_2$ fixing $L_1$ is given in Figure~\ref{fig:looseL(2n+1,2)K1fixed}.

 For fixed $k_1$ and $k_2$, they pair with each other to give non-loose Hopf links with zero Giroux torsion in the complement.  Note that, the complement of  ${S_+}^k(L_1^0)\sqcup {S_+}^{k'}(L_2^0)$ contains a half Giroux torsion layer if $k+k'=1$ and it contains a full Giroux torsion layer if $k+k'>1$. If one component is stabilized positively and the other one negatively that loosens the link.
\begin{center}
    \begin{table}[!htbp]
    \tiny
        \centering
        \begin{tabular}{||c||c||c||c||c||c||}
        \hline
        
             $k_1$ & $k_2$  &numbers&$\rot_{\mathbb{Q}}(L_1)$ & $\rot_{\mathbb{Q}}(L_2)$ &range of $m$ \\
             \hline
             0&0  &2 &$\mp\frac{n+2}{2n+1}$&$\mp\frac{3}{2n+1}$ &\\
             \hline
             0&1&$(n+3)$ &$\frac{-n-2+2m}{2n+1}$ &$\frac{2(-n-2+2m)}{2n+1}$&$0\leq m\leq n+2$\\
             \hline
             0&>1&$2(n+2)$ &$\frac{-n-1+2m}{2n+1}\mp\frac{1}{2n+1}$ &$\mp (k_2-\frac{2n-1}{2n+1})+\frac{2(-n-1+2m_1)}{2n+1}$&$0\leq m\leq n+1$\\
             \hline
             1&0& 4 &$\frac{(2m-3)(1+n)}{2n+1}$ &$\frac{2m-3}{2n+1}$&$0\leq m\leq 3.$\\
             \hline
             1&1&$3(n+2)$ &$\frac{(-2+2m_1)(n+1)}{2n+1}+\frac{(-n-1+2m_2)}{2n+1}$ &$\frac{(-2+2m_1)}{2n+1}+\frac{2(-n-1+2m_2)}{2n+1}$&  \begin{tabular}{@{}c@{}}$0\leq m_1\leq 2$ \\$0\leq m_2\leq n+1$ \end{tabular}\\
             \hline
             1&>1&$6(n+1)$ &$\frac{(-2+2m_1)(n+1)}{2n+1}+\frac{(-n+2m_2)}{2n+1}\mp\frac{1}{2n+1}$ &$\frac{(-2+2m_1)}{2n+1}+\frac{2(-n+2m_2)}{2n+1}\mp (k_2-\frac{2n-1}{2n+1})$ &\begin{tabular}{@{}c@{}}$0\leq m_1\leq 2$ \\$0\leq m_2\leq n$ \end{tabular}\\
             \hline
             >1&0&6&$\mp k_1\pm\frac{n}{2n+1}-\frac{(n+1)(2-2m)}{2n+1}$ &$\mp\frac{1}{2n+1}-\frac{2-2m}{2n+1}$ &$0\leq m\leq 2$\\
             \hline
             >1&1&$4(n+2)$&$\mp(k_1-\frac{n}{2n+1})\mp\frac{n+1}{2n+1}+\frac{-n-1+2m}{2n+1}$ &$\mp\frac{1}{2n+1}\mp\frac{1}{2n+1}+\frac{2(-n-1+2m)}{2n+1}$ &$0\leq m\leq n+1$\\
             \hline
             >1&>1&$8(n+1)$&$\mp(k_1-\frac{n}{2n+1})\mp\frac{n+1}{2n+1}+\frac{-n+2m}{2n+1}\mp\frac{1}{2n+1}$ &$\mp\frac{1}{2n+1}\mp\frac{1}{2n+1}+\frac{2(-n+2m)}{2n+1}\mp(k_2-\frac{2n-1}{2n+1})$ &$0\leq m\leq n$\\
             \hline
             
        \end{tabular}
        \caption{Rotation numbers of non-loose Hopf links in $L(2n+1,2)$ for $n\geq 1$ and for large $\cup$ large slope.}
        \label{tab:L(2n+1,2)}
    \end{table}
\end{center}
    
\end{proof}

\section{The general Classification result}
\label{sec:general_result}

In this section we prove the general classification of rational Hopf links. As we are classifying the  positive Hopf links, we consider $L_1, L_2$ with opposite orientations. A classification for the negative Hopf link can be done by switching the orientation of $L_2$, thus changing the signs of the rotation numbers and switching the back and the forward slashes.
\begin{proof}[Proof of Theorem \ref{thm:general}] Consider the lens space $L(p.q)$ where $q\neq 1$ or $p-1$. Let $-\frac{p}{q}=[a_0,a_1,\cdots a_n]$. Note that, $(-\frac{p}{q})^a=[a_0,a_1,\cdots a_{n-1}]$ and we denote this as $-\frac{p'}{q'}$ where $1\leq p'\leq p$ and $p'q\equiv 1  \pmod p$.  On the other hand, we denote $(-\frac{p}{q})^c= [a_0,a_1,\cdots, a_n+1]$ as $-\frac{p''}{q''}$ where $1\leq p''\leq p$ and  $p''q\equiv p-1\pmod p$.

Like in the previous section, for non-loose representatives of $L_1\cup L_2$ we need to consider the following pairs of slopes. small slope $\cup$ large slope, large slope $\cup$ small slope and large slope $\cup$ large slope. 
\subsection{Case 1: small slope $\cup$ large slope}\label{ssec:small_large_general}  Note that, in this case, the possible dividing slopes for a standard neighborhood of a non-loose representative $L_1\sqcup L_2$ are $-\frac{p''}{q''}\oplus k_1(-\frac{p}{q})=-\frac{p''+k_1p}{q''+k_1q}$ for $k_1\geq 0$ and $\frac{1}{k_2}$ for $k_2\geq 0$. To understand the non-loose representatives we need to understand the signed paths describing the tight contact structures on $L(p,q)\setminus(V_1\cup V_2)=T^2\times I$ with boundary conditions $s_{k_1}$ and $s_{k_2}$. Now for $n\geq 2$ ($n=1$ case is slightly different which is left for the reader), we will divide this path into two sub-paths. One from $s_{k_1}$ clockwise to $-\frac{p''}{q''}$ which is a continued fraction block of length $k_1$ , then a path $P_1$ from $-\frac{p''}{q''}$ clockwise to $s_{k_2}=\frac{1}{k_2}$. Note that, depending on the value of $k_2$ the path from $-\frac{p''}{q''}$ to $\frac{1}{k_2}$ will be different. For $k_2=0$, decorations of the path $P_1$ will correspond to $A_1=|a_1|\cdots |a_n+1|$ tight contact structures and thus we have $A_1(k_1+1)$ non-loose representatives. For $k_2=1,$ decorations of this path correspond to $A_2=|a_0-1||a_1+1|\cdots |a_n+1|$ tight contact structures. Thus we have $A_2(k_1+1)$ non-loose representatives in this case. Finally, when we consider $k_2>1$, we can subdivide $P_1$ into two parts: one from $-\frac{p''}{q''}$ to $0$ and then a jump from $0$ to $\frac{1}{k_2}.$ Thus decorations of this path gives us $A_3=2|a_0||a_1+1|\cdots|a_n+1|$ tight contact structures and we will have a total $A_3(k_1+1)$ non-loose Hopf links.
 It is easy to see from the Farey graph that with these dividing slopes, we will have $L_1$ loose and $L_2$ non-loose.

Next we will show how these non-loose candidates are related by stabilizations. Note that, $L_1$ is already loose and we will have a loose mountain peaked at $(r,-1-\frac{p''}{p})$ where $r$ is the rotation number that can be calculated from the algorithm given in Section~\ref{sec:background}.  We will now consider the stabilizations of $L_2$. We will denote the component $L_2$ with its standard neighborhood having slope $s_{k_2}$ as $L_2^{k_2}$. %For $k_2=0$, the standard neighborhood $V_2$ of $L_2^{0}$ has dividing slope $\infty$ and by stabilizing we are adding a basic slice $B_\pm$ with boundary slope $\infty$ and $-1$. Clearly, $V_1\cup B_\pm$ is overtwisted as the path descibing a contact structure on $V_1\cup B_{\pm}$ crosses $-\frac{p}{q}$ twice which is the meridional slope of the complement and thus gives a overtwisted disk. Thus any stabilization of $L_2^{0}$ is loose. 

By Lemma~\ref{lem:stabilization_loose} any stabilization of $L_2^0$ is loose. Note that, the link can still be non-loose. %Denote the Legendrian knot by % $(-L_2^0)_{m,n}$ whose complement contains a continued fraction block of length $|a_1+1|$ from $[a_0, a_1+1]$ to $\infty$ with $m$ positive signs and $n$ negative signs. The positive stabilization of $(-L_2^0)_{|a_1+1|,0}$ is loose, but $L_1^{s_{k_1}}\sqcup S_+((-L_2^0)_{|a_1|+1,0})$ is still non-loose with a half Giroux torsion in its complement. 
Suppose $P'$ be the path that describes the complement of $L_1\sqcup L_2^0$. This is a union of continued fraction blocks. Stabilizing $L_2$ means adding a $\pm$ basic slice $B_\pm$ with boundary slopes $\infty$ and $-1$ to the back face $T^2\times\{1\}$. If all the continued fraction blocks have the same sign, say positive then $L_1\sqcup S_+(L_2^0)$ is still non-loose but will have a half- Giroux torsion in its complement. The same is true if the path contains only negative signs and $L_2$ is negatively stabilized.  But if any of the continued fraction blocks has a mix of sign, adding $B_\pm$ will lead to an inconsistent shortening in $T^2\times I$ and thus the link will be loose.
 The link component $(L_2^0)^{m,n}$ with $m=|a_1+1|$ and $m=0$, is the base of $|a_2+1|\cdots |a_n+1|$ forward slashes and  $|a_2+1|\cdots |a_n+1|$ back slashes respectively. For $1\leq m,n\leq |a_1+2|$, $(L_2^0)_{m,n}$ are the base of $|a_1+2||a_2+1|\cdots|a_n+1|$ non-loose $V's$.   

Now we consider the stabilizations of $L_2^{1}$. Note that there are $A_2(k_1+1)$ non-loose representatives. The path $P$ from $a_0+1$ to $1$ has $|a_0|$ basic slices and suppose the decoration of this path has $i$ positive signs and $|a_0|-i$ negative signs. We call the $L_2^1$ whose complement correspond to these decorations as $(L_2^1)^{i,j}$ where $i$ and $j$ are the number of positive and negative basic slices in $P$. Clearly, the only two cases where we will see consistent shortening are when we positively stabilize $(L_2^1)^{|a_0|,0} $ or when we negatively stabilize $(L_2^1)^{0,|a_0|}$. Any stabilization of $(L_2^1)^{i,j}$ where $1\leq i,j\leq |a_0+1|$ is loose and these $(L_2^1)^{i,j}$ also give us the base of $|a_0+1||a_1+1|\cdots|a_n+1|$ non-loose $V$'s.

Finally for $L_2^{k_2}$ with $k_2>1$, we need to consider two cases. For $L_2^2$, the path from $a_0+1$ to $\frac{1}{2}$ can have the following decorations: $k$ positive basic slices and $l$ negative basic slices between $a_0+1$ to $0$ where $0\leq k,l\leq |a_0+1|$ and then a $\pm$ basic slice with boundary slopes $0$ and $\frac{1}{2}.$ Let us call the corresponding $L_2^2$ as $(L_2^2)^{i,j,\pm}$. In this case, stabilizing corresponds to adding a basic slice $B_{\pm}$ with boundary slopes $\frac{1}{2}$ and $1$. For the same reason as before we will have $S_+((L_2^2)^{k,l,+})$, $S_-((L_2^2)^{k,l,-})$ non-loose and correspond to $(L_2^1)^{k+1,l}$ and $(L_2^1)^{k,l+1}$ where $0\leq k,l\leq |a_0+1|$. On the other hand $S_-((L_2^2)^{k,l,+})$, $S_-((L_2^2)^{k,l,+})$ will be loose. For $k_2>2$, the path describing the contact structure of the complement contains a continued fraction block of length $|a_0+1|$ from $a_0+1$  to $0$ with $k$ positive and $l$ negative basic slices as before and a single jump from $0$ to $\frac{1}{k_2}.$ Note that, stabilizing in this case corresponds to adding a basic slice of slope $\frac{1}{k_2}$ and $\frac{1}{k_2-1}$ in the complement. Like before we denote the link component $L_2$ corresponding to these paths as $(L_2^{k_2})^{k,l,\pm}$. We will  have $S_+((L_2^{k_2})^{k,l,+})$, $S_-((L_2^{k_2})^{k,l,-})$ non-loose and these coincides with $(L_2^{k_2-1})^{k,l,+}$ and $(L_2^{k_2-1})^{k,l,-}$. %These elements will correspond to the right wing and the left wing of the $V$'s mentioned before.

%Notice now, after shortening the path we have a continued fraction block from $-1$ to $1$ describing the complements of $S_-(L_2^{+-})$ and  $S_+(L_2^{-+})$. As one can shuffle signs inside a continued fraction freely, these non-loose components are isotopic up to shuffling. Thus there are exactly $A_2(k_1+1)$ non-loose stabilizations which coincide with $L_2^{1}.$ Finally, for $k_2>2$, one could easily check that $S_\pm(L_2^{\pm\pm})$ will be non-loose, $S_\mp(L_2^{\pm\pm})$ will be loose, $S_\pm(L_2^{\mp\pm})$ are non-loose and distinct and $S_\mp(L_2^{\mp\pm})$ are all loose. So, for $k_2>2$, $A_3(k_2+1)$ of the stabilizations are non-loose and coincides with $L_2^{k_2-1}$.

Using the method mentioned in Section~\ref{sec:background} one could easily compute the rational Thurston-Bennequin invariants as $\tb_\mathbb{Q}(L_1^{s_{k_1}})=-k_1-\frac{p''}{p}$ and $\tb_\mathbb{Q}(L_1^{s_{k_2}})=k_2+\frac{q}{p}$.

Now putting these together , we have the following mountain range for $L_2^{s_{k_2}}$ fixing $L_1$ and with zero Giroux torsion in the complement of $L_1\sqcup L_2$:
We have $|a_2+1||a_3+1|\cdots|a_n+1|$ many forward and $|a_2+1||a_3+1|\cdots|a_n+1|$ many backward slashes based at $(r_i,\frac{q}{p})$, $|a_1+2||a_2+1|\cdots|a_n+1|$ non-loose $V$'s based at $(r_j,\frac{q}{p})$. Finally there are $|a_0+1||a_1+1|\cdots|a_n+1|$ many $V$'s based at $(r_k,1+\frac{q}{p})$.  The $d_3$ invariants can be computed from the surgery diagrams given in Section~\ref{sec:contactsurgery} using the techniques of \cite{Ding_Geiges_Stipcisz}. The rotation number and Euler classes can be computed using the algorithm mentioned in Section~\ref{sec:background}. Note that, as $L_1$ is loose, we will have a loose  cone peaked at $(r,-1-\frac{p''}{p})$ where value of $r$ can be computed using the technique mentioned before.  A similar analysis can be done for $n=1$. For $n=1, k_2=0$, we note that we have a continued fraction block of length $|a_1+1|$ from $-\frac{p'}{q'}$ to $\infty$ and thus we will have $|a_1|(k_2+1)$ non-loose representatives in this case. For $k_2\geq 1,$ the analysis is exactly the same as $n\geq 2$. Putting everything together for $L_2$, we will have, $1$ forward and $1$ back slash, $|a_1+2|$ $V's$ based at $(r,\frac{q}{p})$ and $|a_0+1||a_1+1|$ $V$'s based at $(r,1+\frac{q}{p})$ when $L_1$ is fixed. 

\subsection{Case 2 Large slope $\cup$ small slope} \label{ssec:large_small_general} Note that, we start with $k_1\geq 0$ and $k_2\geq 1$ as $k_2=0\ \text{and}\ k_1\geq0$ overlaps with the next case. This particular case is very similar to the proof of Theorem 1.15 in \cite{Chatterjee-Etnyre-Min_2025_existence}. Here we have $s_{k_1}=-\frac{p'+k_1p}{q'+k_1q}$ and $s_{k_2}=-\frac{1}{k_2}$ for $k_i\geq 0$. For $n\geq 2$, denote the path from  $r=[a_0, a_1,\cdots a_{n-2}+1]$ to $-1$ in the Farey graph by $P_2$. Decorations on this path will correspond to $B=|a_0+1||a_1+1|\cdots|a_{n-2}+1|$ tight contact structures on $T^2\times I$ with boundary slopes $r$ and $-1$.  Note that, any path from $s_{k_1}$ clockwise to $-1$ will contain this path. Now any path from $s_{k_1}$ clockwise to $s_{k_2}$ can be broken into the following sub-paths: a path from $s_{k_1}$ to $r$, then $P_2$ and finally a path from $-1$ to $-\frac{1}{k_2}.$ The last portion of the path is a continued fraction block of length $k_2-1$ and thus decorations on this path will correspond to $k_2$ tight contact structures. We denote this path by $P_3$. Now we will look at the path for different values of $k_1$. In particular for $k_1=0$, we see that the path from $s_0$ to $s_{k_2}$ consists of a continued fraction block of length $|a_{n-1}+1|$ from $-\frac{p'}{q'}$ to $r$ and then $P_2$ followed by $P_3$. This gives $|a_{n-1}|Bk_2$ tight contact structures. For $k_1=1,$  we see that that the path from $s_1$ to $-\frac{1}{k_2}$ consists of a continued fraction block of length $|a_n|$ from $s_1$ to $s=[a_0, a_1,\cdots a_{n-1}+1]$, then a continued fraction block of length $|a_{n-1}+2|$ from $s$ to $r$, followed by $P_2$ and $P_3$. This gives $|a_n-1||a_{n-1}+1|Bk_2$ tight contact structures. For $k_1>1$, the path from $s_{k_1}$ to $s_{k_2}$ consists of one jump from $s_{k_1}$ to $-\frac{p}{q}$, followed by a continued fraction block of length $|a_n+1|$ from $-\frac{p}{q}$ to $s=[a_0, a_1,\cdots a_{n-1}+1]$, then a continued fraction block of length $|a_{n-1}+2|$ from $s$ to $r$, then $P_2$ and finally $P_3$. This gives $2|a_n||a_{n-1}+1|Bk_2$ tight contact structures.
It is easy to check that in this case, $L_1$ is non-loose and $L_2$ is loose. Like Case~\ref{ssec:small_large_general}, now we will check the different stabilizations of $L_1$ and see how these stabilizations are related with each other. By $L_j^{k_i}$ we denote the Legendrian unknot whose standard neighborhood has dividing slope $s_{k_i}$. Note that, any stabilization of $L_1^0$ will be loose by Lemma~\ref{lem:stabilization_loose} but the link still can be non-loose. The complement of $L_1^0\sqcup L_2$  contains a union of continued fraction blocks from $-\frac{p'}{q'}$ to $s_{k_2}$. Lets call it $P'$. While stabilizing $L_1^0$ we are adding a basic slice with boundary slopes $-\frac{p''}{q''}$ and $-\frac{p'}{q'}$. Thus whenever $P'$ has a mix of sign, we can inconsistently shorten the path and $T^2\times I$ becomes overtwistsed loosening the link. So, when $P'$ only contains positive signs, $S_+(L_1^0)\sqcup L_2$ is still non-loose but contains a half Giroux torsion in its complement. Same is true if $P'$ just contains negative basic slices and we negatively stabilize $L_1^0$. Thus there will be two non-loose $L_1^0$ which stays non-loose when included in the link but the link now will have a half-Giroux torsion in the complement. % \textcolor{red}{Make sure this is correct.} %To see this, notice that we add a basic slice of slopes $(-\frac{p}{q})^a$ and $(-\frac{p}{q})^c$ to the complement of $L_1^0$ and clearly then the complement of $L_1^0\cup L_2^{s_{k_2}}$ becomes overtwisted and thus $S_\pm(L_1^0)$ is loose. 
The stabilizations of $L_1^{s_{k_1}}$ with $k_1\geq 1$ are exactly the same as shown in the proof of Theroem 1.15 in \cite{Chatterjee-Etnyre-Min_2025_existence}. So, we omit the details. An interested reader is suggested to check that proof.

 For $n=1$, we note that $-\frac{p'}{q'}=[a_0]$ and $s=[a_0+1]$ are integers. So, any path from $s_{k_1}$ to $s_{k_2}$ will be broken as follows. For $k_1=0$, a path from $s_0$ to $-\frac{1}{k_2}$ consists of a continued fraction block of length $|a_0+1|$ from $s_0$ to $-1$ followed by another continued fraction block of length $|k_2-1|$, denoted by $P_3$ from $-1$ to $-\frac{1}{k_2}$. Decorations on this path correspond to $|a_0|k_2$ tight contact structures. For $k_1=1$, the path from $s_{k_1}$ to $-\frac{1}{k_2}$ consists of a continued fraction block of length $|a_1|$ from $s_1$ to $s$, followed by a continued fraction block of length $|a_0+2|$ from $s$ to $-1$ and finally $P_3$. This corresponds to $|a_1-1||a_0+1|k_2$ contact structures. For $k_1>1$, the path will consist of one jump from $s_{k_1}$ to $-\frac{p}{q}$, followed by a continued fraction block of length $|a_1+1|$ from $-\frac{p}{q}$ to $s$, followed by a continued fraction block of length $|a_0+2|$ from $s$ to $-1$ and finally $P_3$. This corresponds to $2|a_1||a_0+1|k_2$ contact structures.

 The rational Thurston--Bennequin invariant of the components are given by $\tb_\mathbb{Q}(L_1^{s_{k_1}})= k_1+\frac{p'}{p}$ and $\tb_\mathbb{Q}(L_2^{s_{k_2}})= -k_2+\frac{q}{p}$. Putting everything together we see the following non-loose mountain range for $L_1$: Fixing the loose component $L_2$, there are $|a_0+1|\cdots |a_{n-2}+1|$ forward slashes and the same number of backward slashes based at $(r_i,\frac{p'}{p})$, $|a_0+1|\cdots|a_{n-2}+1||a_{n-1}+2|$ $V$'s based at $(r_j,\frac{p'}{p})$ and finally another $|a_0+1|\cdots|a_{ n-1}+1||a_n+1|$ $V$'s based at $(r_k,1+\frac{p'}{p}).$ For $n=1$, we will have the following mountain range : $1$ forward and $1$ back slash, $|a_0+2|$ $V$'s based at $(r,\frac{p'}{p})$ and additionally $|a_0+1||a_1+1|$ $V$'s based ar $(r,1+\frac{p'}{p}).$
\begin{center}
 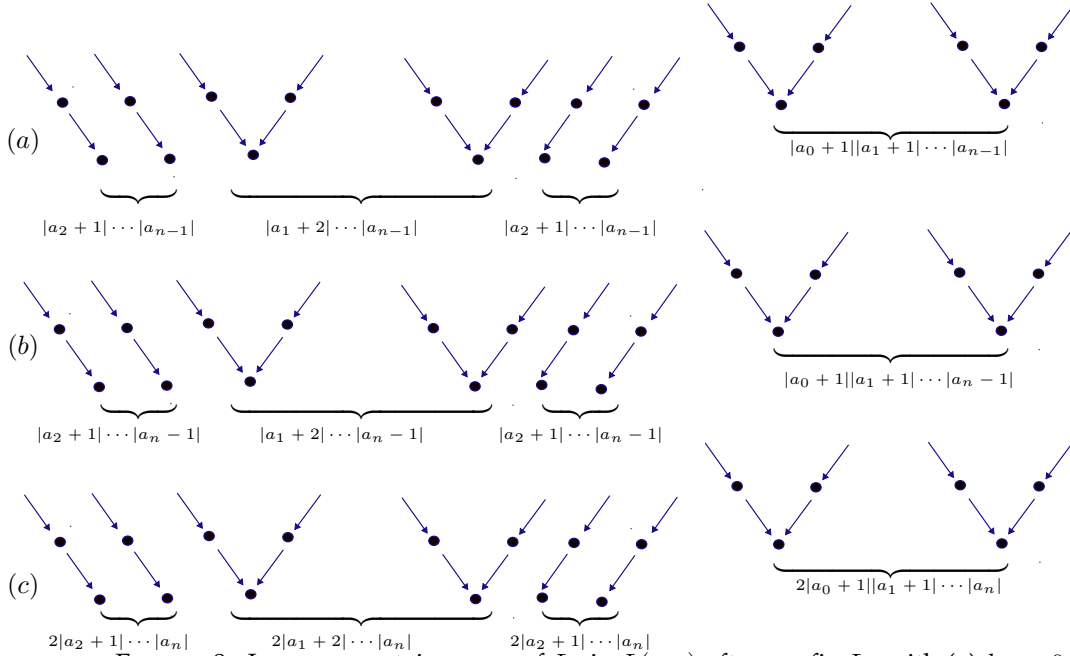
\begin{figure}
    
    \labellist
    \pinlabel $(a)$ at 0 600
    \pinlabel $(b)$ at 0 340
    \pinlabel $(c)$ at 0 40
    \pinlabel $\underbrace{\hspace{3 em}}$ at 145 530
     \pinlabel $\underbrace{\hspace{3 em}}$ at 700 530
       \pinlabel $\underbrace{\hspace{9 em}}$ at 1090 610
        \pinlabel $\underbrace{\hspace{10 em}}$ at 425 530
      \pinlabel $\underbrace{\hspace{3 em}}$ at 145 260
     \pinlabel $\underbrace{\hspace{3 em}}$ at 700 260
       \pinlabel $\underbrace{\hspace{9 em}}$ at 1090 330
        \pinlabel $\underbrace{\hspace{10 em}}$ at 425 260

         \pinlabel $\underbrace{\hspace{3 em}}$ at 145 0
     \pinlabel $\underbrace{\hspace{3 em}}$ at 700 0
       \pinlabel $\underbrace{\hspace{9 em}}$ at 1090 60
        \pinlabel $\underbrace{\hspace{10 em}}$ at 425 0
        
    \tiny
    \pinlabel $|a_2+1|\cdots|a_{n-1}|$ at 120 490
   
    \pinlabel $|a_1+2|\cdots|a_{n-1}|$ at 400 490
    \pinlabel $|a_2+1|\cdots|a_{n-1}|$ at 700 490
    \pinlabel $|a_0+1||a_1+1|\cdots|a_{n-1}|$ at 1100 590

    %%%%%%%%%%%%%%%%%%%%%%%%%%%%%%%%%%%%%%%%%%%%%%%%%%%%%%%%
     \pinlabel $|a_2+1|\cdots|a_{n}-1|$ at 120 230
    
    \pinlabel $|a_1+2|\cdots|a_{n}-1|$ at 400 230
    \pinlabel $|a_2+1|\cdots|a_{n}-1|$ at 700 230
    \pinlabel $|a_0+1||a_1+1|\cdots|a_{n}-1|$ at 1100 300
  %%%%%%%%%%%%%%%%%%%%%%%%%%%%%%%%%%%%%%%%%%%%%%%%%%%%%%%%%%%%
    \pinlabel $2|a_2+1|\cdots|a_{n}|$ at 120 -30
    
    \pinlabel $2|a_1+2|\cdots|a_{n}|$ at 400 -30
    \pinlabel $2|a_2+1|\cdots|a_{n}|$ at 700 -30
    \pinlabel $2|a_0+1||a_1+1|\cdots|a_{n}|$ at 1100 40
    \endlabellist
    \includegraphics[scale=0.3]{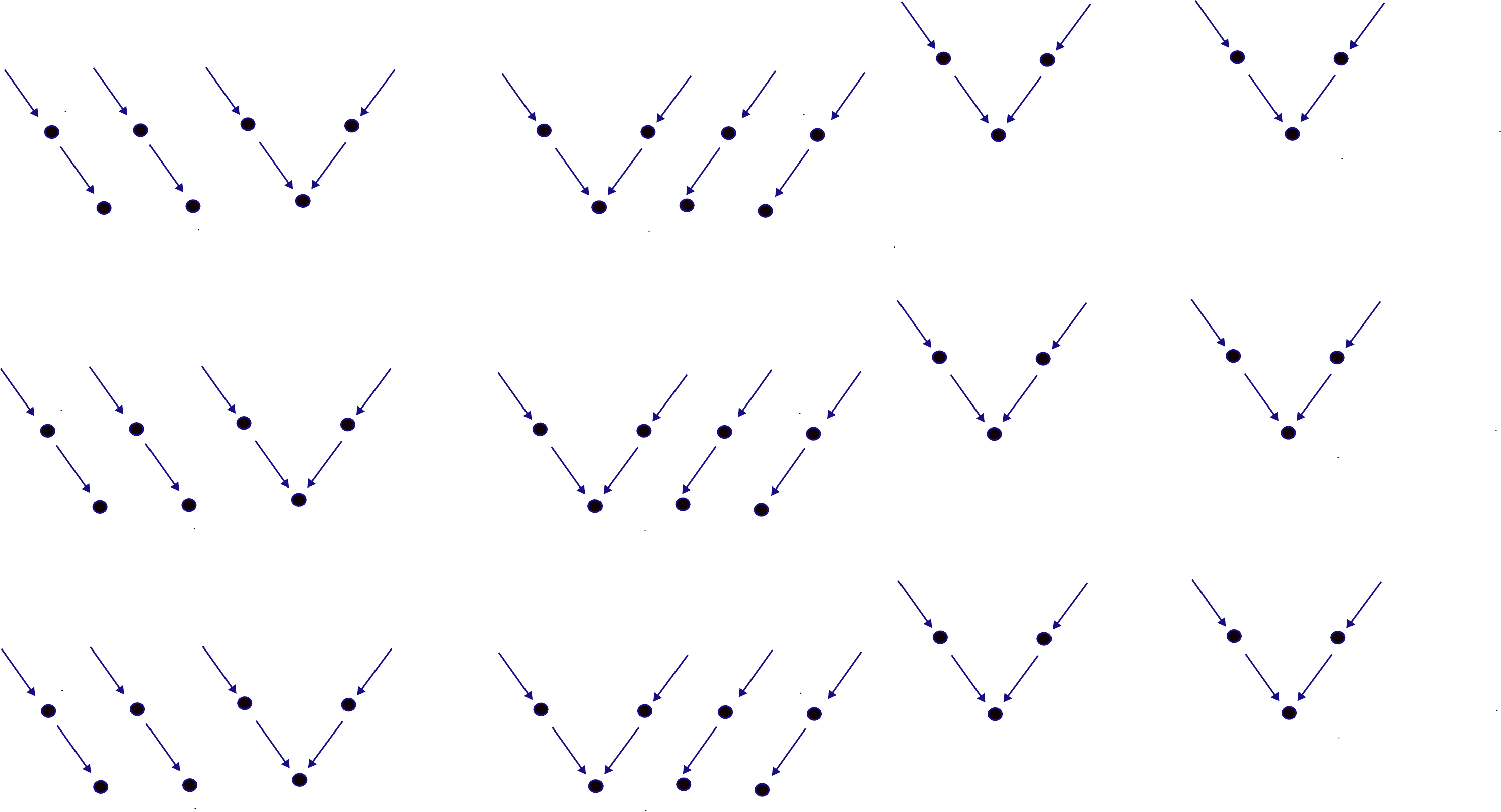}
    \caption{Loose mountain range of $L_2$ in $L(p,q)$ after we fix $L_1$ with (a) $k_1=0$, (b) $k_1=1$ and (c) $k_1>1.$}
    \label{fig:Fix_L_1_General}
\end{figure}
\end{center}

\subsection{Case 3: Large slope $\cup$ Large slope}\label{ssec:large_large_general} Here we will consider both large slopes. The dividing slope of the standard neighborhod of $L_1$ and $L_2$ in this case are given by $s_{k_1}=-\frac{p'+k_1p}{q'+k_1q}$ and $s_{k_2}=\frac{1}{k_2}$ for $k_i\geq 0$. Note that, the path from $s_{k_1}$ to $s_{k_2}$ will be very similar to our last case except that now the path will go to $\frac{1}{k_2}.$ First we consider $k_2=0.$ For $n\geq 3, $ denote the path from $r=[a_0,a_1,a_{n-2}+1]$ to $\infty$ in the Farey graph by $P_2$ (Note that, this path will go from $r$ to $a_0+1$ and then take a jump to infinity) . Decorations on this path correspond to $A_1=|a_1||a_2+1|\cdots|a_{n-2}+1|$ tight contact structures. Any path from $s_{k_1}$ to $s_{k_2}$ will contain this path for $k_2=0$. So after fixing $k_2=0$, we will vary $k_1$. Now as before, for $k_1=0,$ a path from $-\frac{p'}{q'}$ to $r$ consists of a continued fraction block of length $|a_{n-1}+1|$. Thus from $-\frac{p'}{q'}$ to $\infty$ we have a total $A_1|a_{n-1}|$ tight contact structures. Following same method as Case~\ref{ssec:large_small_general} for $k_1=1 \ \text{and} \ k_1>1$, the path gives us $A_1|a_{n-1}+1||a_n-1|$ and $2A_1|a_{n-1}+1||a_n|$ tight contact structures respectively. Now for $k_2=1$ we see that a path from $r$ to $1$ will give us $A_2=|a_0-1||a_1+1|\cdots|a_{n-2}+1|$ tight contact structures. Finally for $k_2\geq 2$, a path from $r$ to $\frac{1}{k_2}$ can be subdivided into a path from $r$ to $0$ and then a jump from $0$ to $\frac{1}{k_2}$.The first part gives us $A_3=|a_0-1||a_1+1|\cdots|a_{n-2}+1|$ tight contact structures. So in total this path corresponds to $2A_3$ tight contact structures. Thus for $k_2=1$ and $k_2\geq 2$, we can vary $k_1$ as before and see the numbers of tight contact structures are as shown in Table \ref{table:general}. Note that, for $n=1$ and $n=2$, we have few cases where the numbers are different. Those are given in Table~\ref{table:general}.

 It is easy to see that both components in each of the cases are loose. Next, we talk about the stabilizations of the components. Note that when we talk about stabilizations we are basically talking about the stabilizations of the loose component but these stabilizations still keep the link non-loose. As before we will consider the stabilizations for $n\geq 3$, the other cases can be analyzed similarly.
 First, we fix $L_1$ with $k_1=0$ and analyze the loose mountain range for $L_2$ while varying $k_2$. The complement of $L_1^0\sqcup L_2^0$ contains a union of continued fraction block. When all the continued fraction blocks are of the same sign,say positive, then $L_1^0\sqcup S_+(L_2^0)$ is still non-loose but has a half Giroux torsion in the complement. Same is true if all the continued fraction blocks are negative and we negatively stabilize $L_2^0$. Every other combination will loosen the link. Now note that, the complement of $L_1^0\sqcup L_2^0$ contains a continued fraction block of length $|a_1+1|$ from $|a_0+1|$ to $\infty$. We call this path $P$ and the corresponding $L_2^0$, as $(L_2^0)_{i,j}$ where $0\leq i,j\leq |a_1+1|$. For $i=|a_1+1|, j=0$ and $i=0, j=|a_1+1|$, $(L_2^0)_{i,j}$ are the bases of $|a_2+1|\cdots|a_{n-1}|$ loose forward and the same number of loose back slashes. For $1\leq i,j\leq |a_1+2|$, $(L_2)^0_{i,j}$ are bases of $|a_1+2||a_2+1|\cdots |a_{n-1}|$ loose $V$'s.

Now we consider $L_2^1$. The complement of $L_2^1$ contains a path $P_1$ of length $|a_0|$ from $a_0+1$ to $1$ and another continued fraction block $P_2$ of length $|a_1+2|$ from $[a_0,a_1+1]$ to $a_0+1$. Let us suppose now that $P_1$ contains $i$ positive basic slices, $j$ negative basic slices and $P_2$ contains $k$ positive basic slices and $l$ negative basic slices. We call the corresponding $L_2^1$ as $(L_2)^1_{i,j,k,l}$ where $0\leq i,j\leq |a_0|$, $0\leq k,l\leq |a_1+2|.$ As we can consistently shorten the path in the complement,  $L_1\sqcup S_+((L_2^1)_{|a_0|,0,|a_1+2|,0})$ and $L_1\sqcup S_-((L_2^1)_{0,|a_0|,0,|a_1+2|})$ are still non-loose. In fact, one could check $S_+((L_2^1)_{|a_0|,0,|a_1+2|,0})$ coincides with $(L_2^0)_{|a_1+1|,0}$ and  $S_-((L_2^1)_{0,|a_0|,0,|a_1+2|})$ coincides with $(L_2^0)_{0,|a_1+1|}$. For $i=|a_0|, j=0$ and $1\leq k,l\leq |a_1+3|$ and $i=0, j=|a_0|$ and $1\leq k,l\leq |a_1+3|$, the $S_\pm((L_2^0)_{i,j,k,l}$ corresponds to $(L_2^0)_{i,j}$ where $1\leq i,j\leq |a_1+2|$.  When $1\leq i,j\leq |a_0+1|$, any stabilization will give us inconsistent shortening in the complement and $L_1^0\sqcup S_\pm( (L_2^1)_{i,j,k.l})$ will be loose for any $k,l.$ There are the bases of $|a_0+1||a_1+1|\cdots|a_{n-1}|$ many loose $V$'s. 

Now if we fix $L_1$ with $k_1=1 \ \text{or} \ k_1\geq 2$, the mountain range will change. The analysis is exactly the same as before, the only difference is in the number of forward, back slashes and $V$'s as shown in Figure~\ref{fig:Fix_L_1_General}. %As we have seen before $L_1^1\sqcup S_\pm((-L_2^0)_{i,j})$ will be non-loose if and only if $i=|a_1+1|, j=0$ or $i=0, j=|a_1+1|$ where $i,j$ are the number of positive and negative basic slices respectively in the continued fraction block of length $|a_1+1|$ that goes from $a_0+1$ to $\infty$.
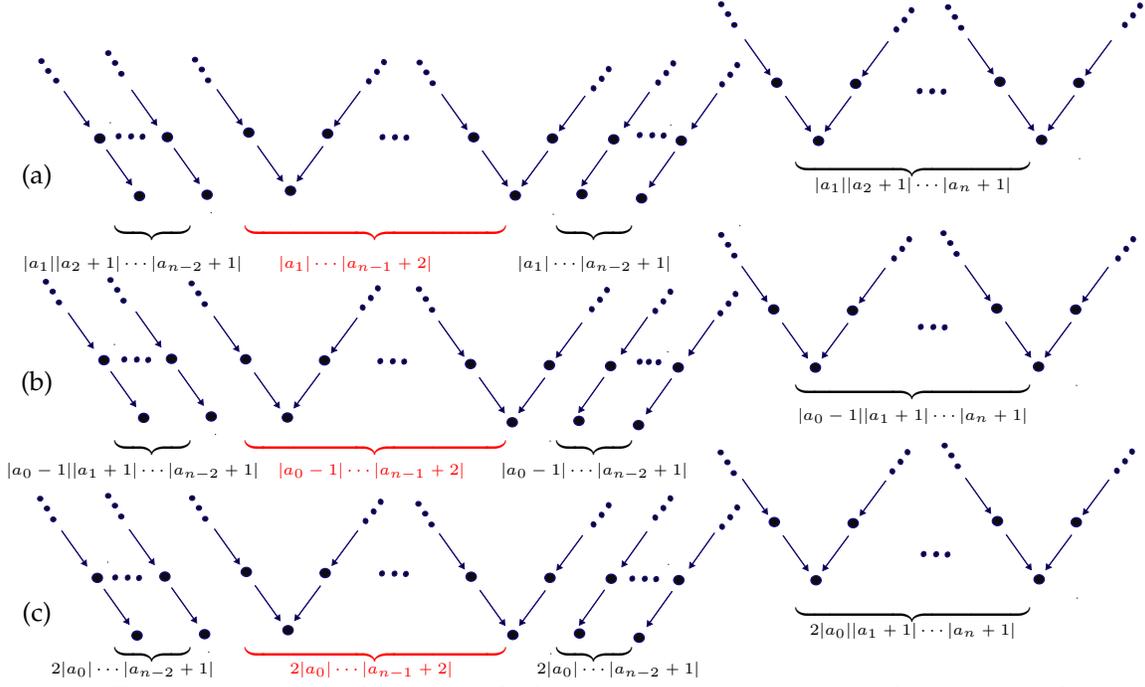
\begin{figure}
 \centering
    \labellist
    \pinlabel (a) at 0 600
    \pinlabel (b) at 0 340
    \pinlabel (c) at 0 50
    \pinlabel $\underbrace{\hspace{3 em}}$ at 145 530
     \pinlabel $\underbrace{\hspace{3 em}}$ at 700 530
       \pinlabel $\underbrace{\hspace{9 em}}$ at 1100 610
        \pinlabel $\textcolor{red}{\underbrace{\hspace{10 em}}}$ at 425 530
      \pinlabel $\underbrace{\hspace{3 em}}$ at 145 260
     \pinlabel $\underbrace{\hspace{3 em}}$ at 700 260
       \pinlabel $\underbrace{\hspace{9 em}}$ at 1100 330
        \pinlabel $\textcolor{red}{\underbrace{\hspace{10 em}}}$ at 425 260
        \pinlabel $\underbrace{\hspace{3 em}}$ at 145 0
     \pinlabel $\underbrace{\hspace{3 em}}$ at 700 0
       \pinlabel $\underbrace{\hspace{9 em}}$ at 1100 50
        \pinlabel $\textcolor{red}{\underbrace{\hspace{10 em}}}$ at 425 0
        
    \tiny
    
    \pinlabel $|a_1||a_2+1|\cdots|a_{n-2}+1|$ at 120 490
   
    \pinlabel $\textcolor{red}{|a_1|\cdots|a_{n-1}+2|}$ at 400 490
    \pinlabel $|a_1|\cdots|a_{n-2}+1|$ at 700 490
    \pinlabel $|a_1||a_2+1|\cdots|a_{n}+1|$ at 1100 590

    %%%%%%%%%%%%%%%%%%%%%%%%%%%%%%%%%%%%%%%%%%%%%%%%%%%%%%%%
     \pinlabel $|a_0-1||a_1+1|\cdots|a_{n-2}+1|$ at 120 230
    
    \pinlabel $\textcolor{red}{|a_0-1|\cdots|a_{n-1}+2|}$ at 420 230
    \pinlabel $|a_0-1|\cdots|a_{n-2}+1|$ at 700 230
    \pinlabel $|a_0-1||a_1+1|\cdots|a_{n}+1|$ at 1100 300

    %%%%%%%%%%%%%%%%%%%%%%%%%%%%%%%%%%%%%%%%%%%%%%%%%%%%%%%%%%%%
  \pinlabel $2|a_0|\cdots|a_{n-2}+1|$ at 120 -20
    
    \pinlabel $\textcolor{red}{2|a_0|\cdots|a_{n-1}+2|}$ at 420 -20
    \pinlabel $2|a_0|\cdots|a_{n-2}+1|$ at 730 -20
    \pinlabel $2|a_0||a_1+1|\cdots|a_{n}+1|$ at 1100 30

    \endlabellist
    \centering
    \includegraphics[scale=0.3]{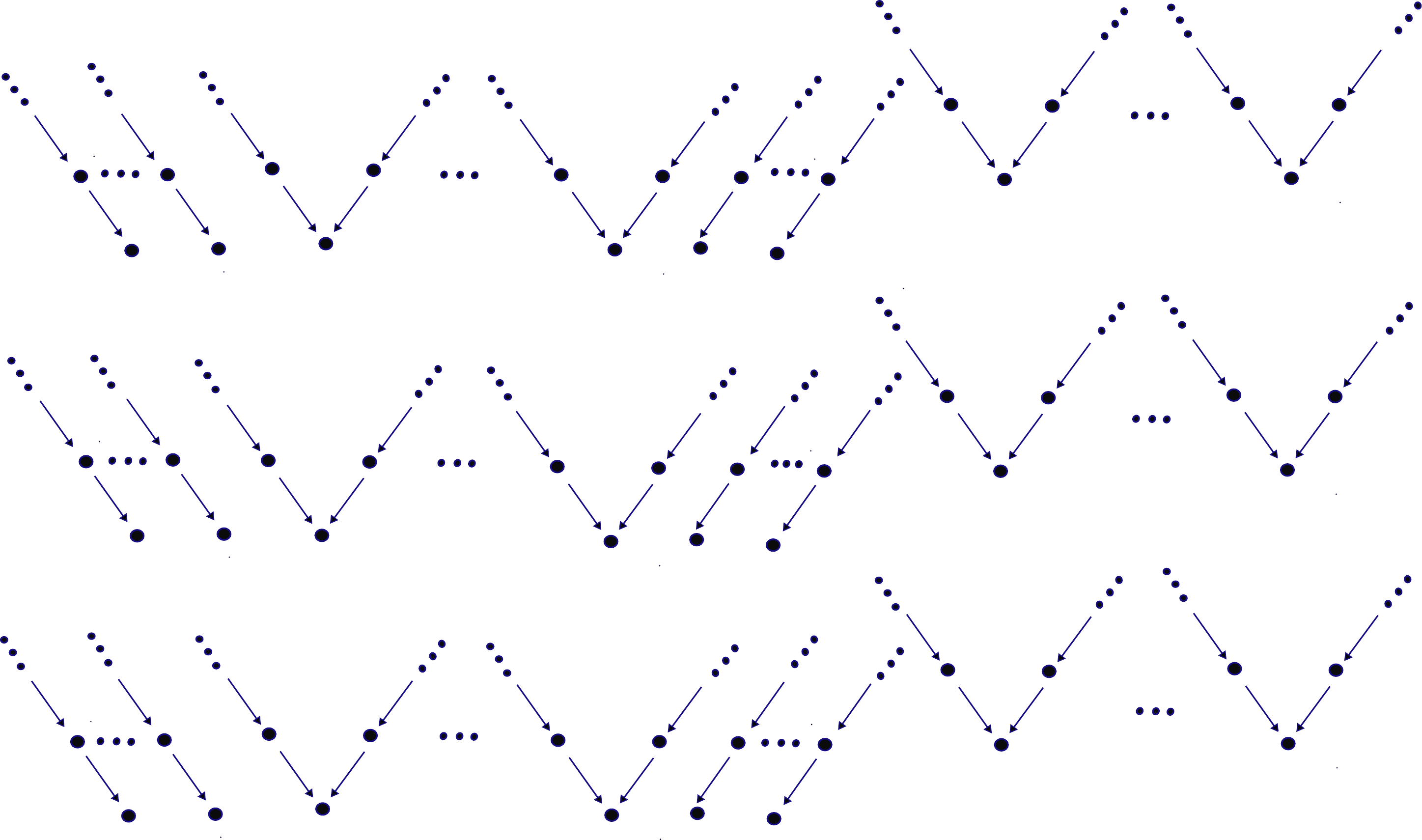}
    \caption{(a), (b) and (c) shows the loose mountain range for $K_1$ when $K_2$ is fixed with $k_2=0$,$k_2=1$,and $k_2\geq 2.$ Here $n\geq 3$.}
    \label{fig:Fix_K_2_general}
\end{figure}

Now one can fix $L_2$ with $k_2=0$ and analyze the loose mountain range of $L_1.$ We will do the detailed analysis for this case and the rest will follow similarly. Now if all the continued fraction blocks in the complement of $L_1^0\sqcup L_2^0$ have the same sign, say positive then  $S_+(L_1^0)\sqcup L_2^0$ will still be non-loose, but will have a half Giroux torsion in the complement. Same is true if the continued fraction blocks are all negative and we negatively stabilize $L_1$. Note that, when we stabilize $L_1^0$ we add a basic slice of boundary slopes $\frac{p'}{q'}$ and $\frac{p''}{q''}$ in the complement. If there is any mix of sign in any of the continued fraction blocks, then any stabilization of $L_1$ will loosen the link. Note that, in the complement of $L_1^0\sqcup L_2^0$, we have a continued fraction block of length $|a_{n-1}+1|$ from $-\frac{p'}{q'}$ to $r$. We denote this path by $P$ and suppose there are $i$ positive basic slices and $j$ negative basic slices. For $i=|a_{n-1}+1|, j=0$,  $(L_1^0)_{i,j}$ are the bases of $|a_1||a_2+1|\cdots|a_{n-2}+1|$ many forward slashes and for $i=0, j=|a_{n-1}+1|$, $(L_1^0)_{i,j}$ are the bases of the same number of backward slashes. For $1\leq i,j\geq |a_{n-1}+2|$, $(L_1^0)_{i,j}$  they will be the bases of $|a_1|\cdots |a_{n-1}+2|$ many \ $V$s. Now when we consider $L_1^1$ the complement contains a continued fraction block of length $|a_n|$ from $s_1$ to $s$ and then another one of length $|a_{n-1}+2|$ from $s$ to $r$. We call the first path $P_1$ and then $P_2$. Suppose $P_1$ contains $i'$ positive signs and $j'$ negative signs and $P_2$ contains  $k$ positive signs and $l$ negative signs. The corresponding $L_1^1$ will be denoted as $(L_1^1)_{i',j',k,l}$ For $i'=|a_n|, j'=0$, $S_+((L_1^1)_{i,j,k,l}$ are non-loose for any $k,l$ and for $i'=0,j'=|a_n|$, $S_+((L_1^1)_{i',j',k,l}$ are non-loose for any $k,l$. These will be part of the forward, back slash and the $V$'s from above. For $i'=|a_n|,j'=0,k=|a_{n-1}+2|$ and $l=0$ $S_+((L_1^1)_{i',j',k,l}$ is same as $(L_1^0)_{|a_{n-1}+1|,0}$. On the other hand, for $i'=0,j'=|a_n|, k=0, l=|a_{n-1}+2|$, $S_-((L_1^1)_{i',j',k,l}$ is same as $(L_1^0)_{0,|a_{n-1}+1|}$. When $P_1$ has all same (positive or negative) sign and $1\leq k,l\leq |a_{n-1}+3|$, $S_\pm((L_1^1)_{i',j'k,l})$ will coincide with $(L_1^0)_{i,j}$ where $1\leq i,j\leq |a_{n-1}+2|.$ Finally for $1\leq i',j'\leq |a_n+1|$, $(L_1^1)_{i',j',k,l}$ are the bases of $|a_1||a_2+1|\cdots |a_n+1|$ loose $V$'s. 

For fixed $K_2$ with $k_2=1, k_2\geq 2$, the mountain range is given in Figure~\ref{fig:Fix_K_2_general} (b), (c). The analysis is very similar to the $k_2=0$ case, so we omit the details here.

%\textcolor{red}{This proof is wrong. Needs to be rewritten.}\textcolor{olive}{This will have several cases and we need to write the proofs individually.} For this case, the dividing slopes on the standard neighborhoods of $L_1$ and $L_2$ are given by $s_{k_1}=-\frac{p'+k_1p}{q'+k_1q}$ and $s_{k_2}=\frac{1}{k_2}$ for $k_i\geq 0$. The path from $s_{k_1}$ to $s_{k_2}$ is identical to the previous case except the last part. The path from $s_{k_1}$ to $s_{k_2}$ will be divided into a path from $s_{k_1}$ to $-1$ exactly as in Case 2, followed by a path from $-1$ to $\frac{1}{k_2}$.  The path from $-1$ to $\frac{1}{k_2}$ is a of length $1$ for $k_2=0$, length $2$ continued fraction block for $k_2=1$ and a path of length $2$ which is not a continued fraction block for $k_2>1$. So combining this path with the path in Case 2, Table \ref{table:general} gives us the number of tight contact structures in the complement of $L_1\cup L_2$ and thus the number of non-loose representatives of $L_1\cup L_2$. From the Farey graph, it is easy to see that both the components are loose here although each of them are non-loose as links. Note that, any stablization of $L_1^0$ and $L_2^0$ will make the link loose.\textcolor{red}{Can we write the stabilizations of the individual loose components here?}

The rational Thurston--Bennequin invariants of the components are given by $\tb_\mathbb{Q}(L_1^{s_{k_1}})= k_1+\frac{p'}{p}$ and $\tb_\mathbb{Q}(L_1^{s_{k_2}})=k_2+\frac{q}{p}$. The rotation number and the Euler classes can be computed as before. This completes the proof of Theorem~\ref{thm:general}.
\begin{center}
\begin{table}

  \begin{tabular}{||c||c||c||c||c||}
  \hline
\label{table:general}
$k_1$ &$k_2$ & $n=1$ & $n=2$ &$n> 2$ \\
     \hline
      0  &0 &$2$ &$|a_1-1|$ &$|a_1||a_2+1|\cdots|a_{n-1}|$\\
      \hline
      0  &1 &$|a_0-2|$ & &$|a_0-1||a_1+1|\cdots|a_{n-1}|$\\
      \hline
      0   &>1 &$2|a_0-1|$ & &$2|a_0||a_1+1|\cdots|a_{n-1}|$ \\
      \hline
      1 &0 &$|a_1-2|$ & &$|a_1||a_2+1|\cdots|a_{n-1}+1||a_n-1|$\\
      \hline
      1 &1 & & &$|a_0-1||a_1+1|\cdots|a_n-1|$\\
      \hline
      1 &>1 & & &$2|a_0||a_1+1|\cdots|a_{n-1}+1||a_n-1|$ \\
      \hline
      >1&0 &$2|a_1-1|$ & &$2|a_1||a_2+1|\cdots|a_n|$\\
      \hline
      >1 &1 & & &$2|a_0-1||a_1+1|\cdots|a_{n-1}+1||a_n|$\\
      \hline
      >1 &>1 & & &$4|a_0||a_1+1|\cdots|a_{n-1}+1||a_n|$\\
      \hline

\end{tabular}
 \caption{Number of non-loose Hopf links in $L(p,q)$ with both components loose where $-\frac{p}{q}=[a_0, a_1,\cdots, a_{n}]$}
\end{table}

\end{center}

\end{proof}

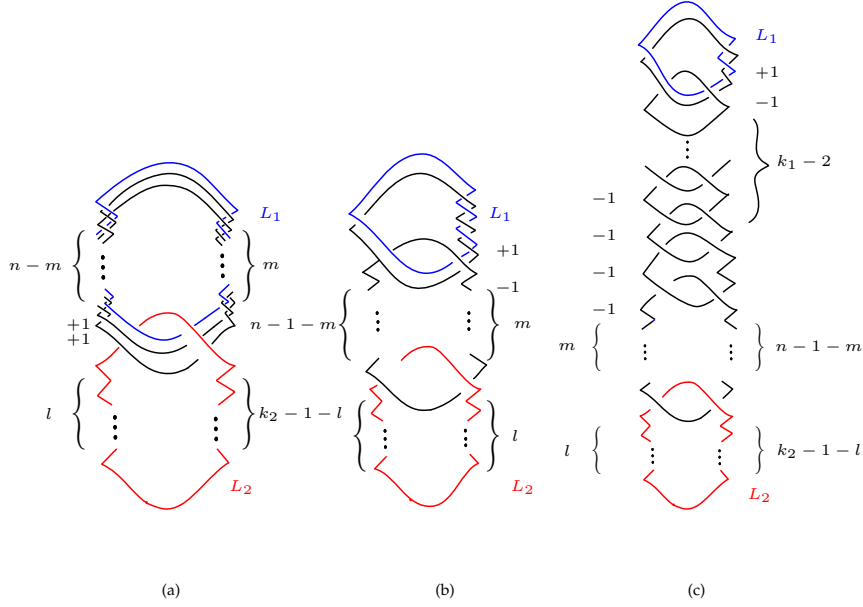
\begin{figure}
 \labellist
\small\hair 2pt
\pinlabel$\Bigg{\{}$  at 20 430
\pinlabel$\Bigg{\}}$  at 250 430

 \pinlabel$\Bigg{\{}$  at 20 230
\pinlabel$\Bigg{\}}$  at 250 230
%%%%%%%%%%%%%%%%%%%%%%%%%%%%%%%%%%%
 \pinlabel $\Bigg\}$ at 580 200

 \pinlabel $\Bigg\{$ at 400 200

 \pinlabel $\Biggl\}$ at 580 350
 \pinlabel $\Biggr\{$ at 380 350
%%%%%%%%%%%%%%%%%%%%%%%%%%%%%%%%%%%%% 
\tiny
\pinlabel{(a)} at 145 -10

\pinlabel {\textcolor{blue}{$L_1$}} at 280 500
\pinlabel $\textcolor{red}{L_2}$ at 240 130

\pinlabel$m$ at 280 430
 \pinlabel${n-m}$ at -38 430 

 %%%%%%%%%%%%%%%%%%%%%%%%
 \pinlabel {(b)} at 515 -10
 \pinlabel $-1$ at 600 400
 \pinlabel $+1$ at 600 450
 
  \pinlabel $n-1-m$ at 310 350
  \pinlabel $m$ at 620 350
   \pinlabel $l$ at 610 200

 \pinlabel $\textcolor{blue}{L_1}$ at 590 500
 \pinlabel $\textcolor{red}{L_2}$ at 620 130
 %%%%%%%%%%%%%%%%%%%%%%%%
 \pinlabel (c) at 855 -10
 \pinlabel $\textcolor{blue}{L_1}$ at 950 740
 \pinlabel $\textcolor{red}{L_2}$ at 940 120
 \pinlabel $+1$ at 950 690
 \pinlabel $-1$ at 950 650
 \pinlabel $k_1-2$ at 1000 570
 \pinlabel $-1$ at 730 470
 \pinlabel $-1$ at 730 420
 \pinlabel $-1$ at 730 370
 \pinlabel $-1$ at 730 520
  \pinlabel $\Biggl\{$ at 720 320
 \pinlabel $\Biggr\}$ at 940 320
 \pinlabel $\Bigg\{$ at 720 180

 \pinlabel $\Bigg\}$ at 940 180
 \tiny

 \pinlabel $+1$ at 20 350
\pinlabel $+1$ at 20 330

\pinlabel $l$ at -20 230
 \pinlabel $k_2-1-l$ at 320 230
 \pinlabel $l$ at 680 180
 \pinlabel $k_2-1-l$ at 1020 180
 \pinlabel $n-1-m$ at 1020 320
 \pinlabel $m$ at 680 320
\endlabellist
    \centering
    \includegraphics[width=0.7\linewidth]{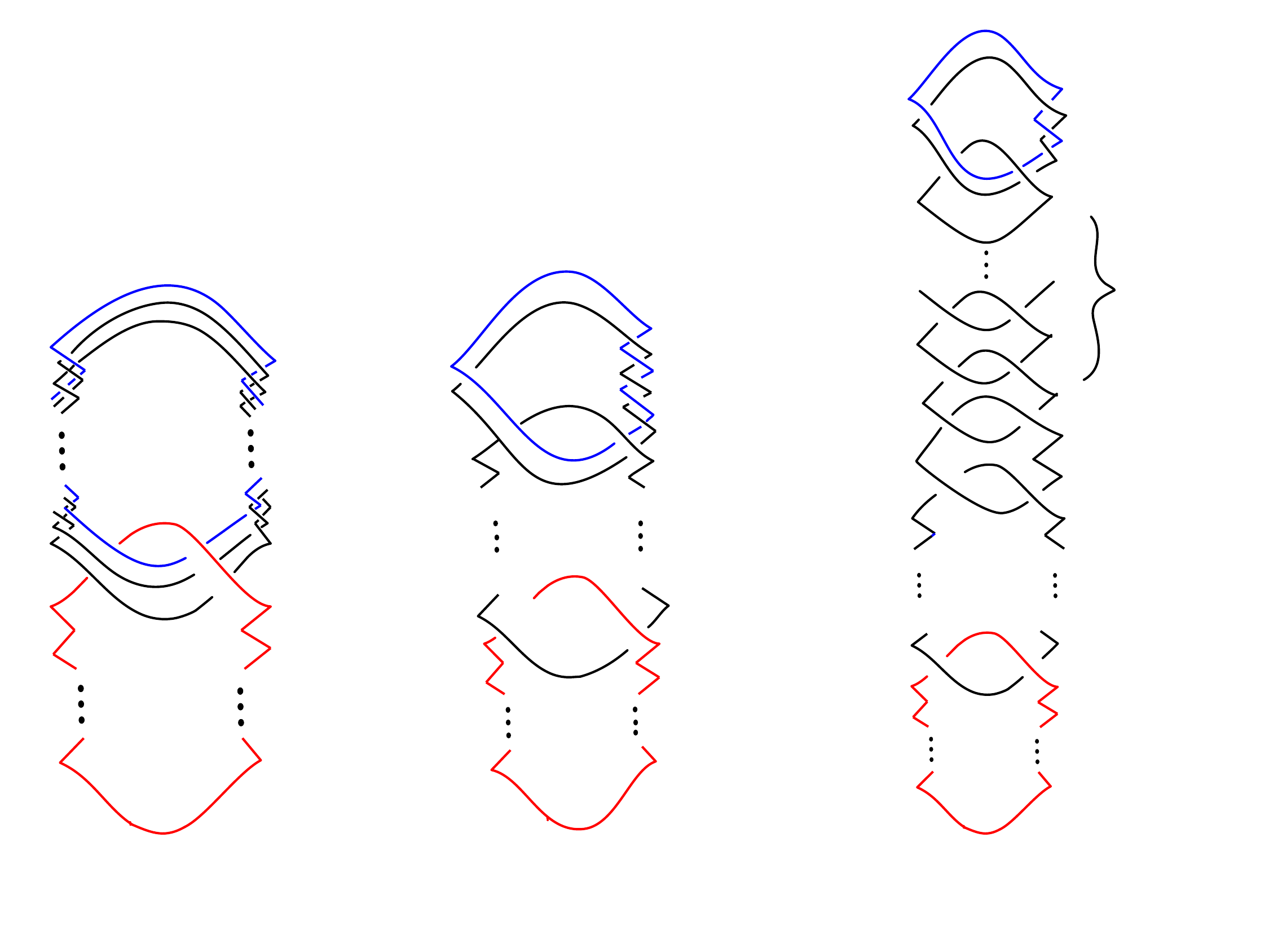}
    \caption{ Non-loose $L_1\sqcup L_2$ in $L(2n+1,2)$ with $\tbr(L_1)=k_1+\frac{n+1}{2n+1}$ and $\tbr(L_2)=-k_2+\frac{2}{2n+1}$. $L_1$ is non-loose and $L_2$ is loose.(a) $k_2(n+1)$ candidates with  $k_1=0$ and $k_2\geq 1$ (b) $3nk_2$ candiadates with $k_1=1, k_2\geq 1$ and (c) $4nk_2$ candidates with $k_1\geq 2$ and $k_2\geq 1.$}
    \label{fig:nonlooseandlooseinL(2n+ 1)}
\end{figure}

\begin{proof}{Proof of Theorem~\ref{thm:Giroux_torsion}} Note that, to add Giroux torsion in the link complement we remove a tubular neighborhood of $L_1\sqcup L_2$ and add a Giroux torsion layer with $m\pi$ twisting. The complement of $L_1\sqcup L_2$ is a $T^2\times I$ which is a union of continued fraction blocks. Clearly, if there is any mismatch of sign once we add a layer, that leads to an inconsistent shortening and the complement will become overtwisted. Thus except the one case when $q=1$ and $k_1=0=k_2$, there are exactly two non-loose $L_1\sqcup L_2$, one with the complement having only positive continued fraction blocks and one with only negative union of continued fraction blocks. In these two cases we can add any layer of torsion (consistent with the sign of the continued fraction blocks) and the complement is still tight. Notice that, when $q=1\ \text{and}\  k_1=k_2=0,$ the complement is $I$-invariant. So there is only one such non-loose Hopf link and we can add any twising in the complement keeping the complement tight. The computations of $\tbr$ are same as before. One could check the two $L_1\sqcup L_2$ can be distinguished by their rotation numbers. This finishes the proof. 
\end{proof}

\begin{remark}
    The above theorem recovers the result of \cite{Geiges_Onaran} for $p=1$.
\end{remark}

\section{Contact surgery diagrams for non-loose Hopf links in $L(p,q)$}
\label{sec:contactsurgery}\subsection{Contact surgery diagrams for non-loose Hopf links in $L(p,1)$} The explicit diagrams are already shown in \cite{chatterjee2025_Hopf}. The only explicit surgery diagrams missing in \cite{chatterjee2025_Hopf} were the $2$ reali\-zations with $\tbr(L_1)=\frac{1}{p}$ and $\tbr=1+\frac{1}{p}.$ (a) in Figure \ref{fig:Loose_General2} shows those $2$ explicit realizations. 
\begin{remark}
   We have an explicit algorithm for the Legendrian realizations of the non-loose Hopf links in $L(p,1)$ as well. The algorithm for $L(p,1)$ is quite similar to the general case.  But as the diagrams are already given in \cite{chatterjee2025_Hopf}, we are not including it here.
\end{remark}

%\begin{figure}
 %\labellist
%\small\hair 2pt

%\pinlabel{(a)} at 180 -10
%\pinlabel $[\frac{1}{p+3}]$ at -10 170
%\pinlabel $[\frac{1}{2}]$ at -10 250
%\pinlabel $\textcolor{blue}{L_1}$ at 0 350
%\pinlabel $\textcolor{red}{L_2}$ at 10 40
%\endlabellist
 %   \centering
  %  \includegraphics[width=0.5\linewidth]{Case2(ab)looseloose}
   % \caption{(a) The $2$ (two choices from the one stabilizations) loose-loose realization of the non-loose Hopf link in $L(p,1)$ with $\tbr(L_1)=1+\frac{1}{p}$ and %$\tbr(L_2)=\frac{1}{p}.$ }
   % \label{fig:Case2(ab)looseloose}
%\end{figure}

\subsection{Contact surgery diagrams for non-loose Hopf links in $L(2n+1,2)$} Figure \ref{fig:nonlooseandlooseinL(2n+ 1)} shows Legendrian representations of non loose Hopf links in $L(2n+1,2)$ where $L_1$ is non-loose and $L_2$ is loose. To see this, one could do a sequence of $-1$ surgeries on $L_1$ that cancels the $+1$ surgeries and we see a tight manifold in the complement. Figure \ref{fig:loosenonlooseL(2n+1)} shows the non-loose Hopf links with $L_1$ loose and $L_2$ non-loose. Finally, Figure \ref{fig:looseloose1_L(2n+1,2)} and \ref{fig:looseloose2_L(2n+1,2)} show all the non-loose representations with loose-loose component. There does not exist a non-loose Hopf link with non-loose/non-loose components in $L(2n+1,2)$. One can do a sequence of Kirby moves to see indeed these are Hopf links in $L(2n+1,2).$

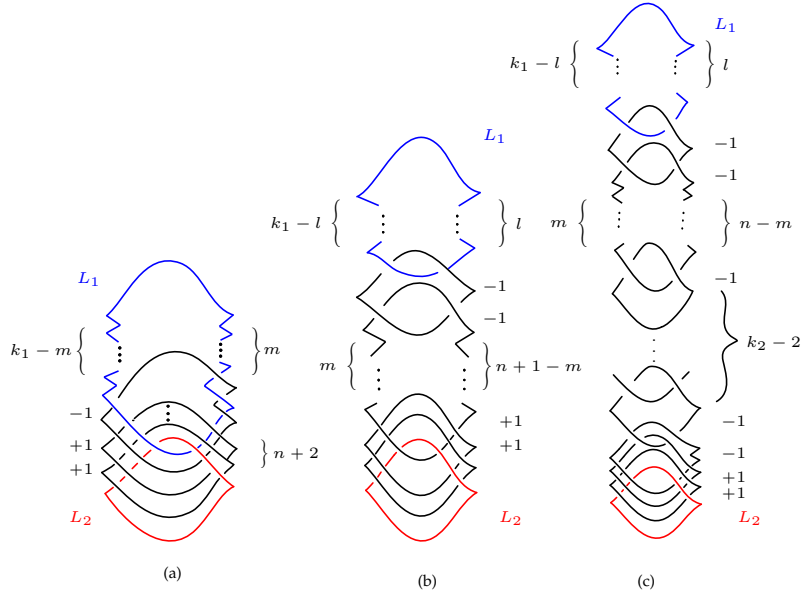
\begin{figure}
\labellist
\small\hair 2pt
\tiny
\pinlabel{(a)} at 125 0
\pinlabel $+1$ at 10 160
\pinlabel $+1$ at 10 130
\pinlabel $\Big{\}}$ at 240 150
\pinlabel $n+2$ at 280 150
\pinlabel $-1$ at 10 200
\pinlabel $\textcolor{blue}{L_1}$ at 20 370
\pinlabel $\textcolor{red}{L_2}$ at 10 70
\pinlabel$\Bigg{\{}$  at 10 280
\pinlabel$\Bigg{\}}$  at 230 280
\pinlabel$m$ at 250 280
 
 %%%%%%%%%%%%%%%%%%%%%%%%
 \pinlabel {(b)} at 445 -10
 \pinlabel $\Bigg{\{}$ at 350 260
 \pinlabel $m$ at 320 260
 \pinlabel $n+1-m$ at 585 260
 \pinlabel $\Biggr\}$ at 520 260
 \pinlabel $\Bigg{\{}$ at 330 440
 \pinlabel $k_1-l$ at 280 440
 \pinlabel $\Biggr\}$ at 540 440
 \pinlabel $l$ at 560 440
 \pinlabel $+1$ at 550 190
 \pinlabel $+1$ at 550 160
 \pinlabel $-1$ at 530 320
 \pinlabel $-1$ at 530 360
 \pinlabel $\textcolor{blue}{L_1}$ at 530 550
 \pinlabel $\textcolor{red}{L_2}$ at 550 70
 %%%%%%%%%%%%%%%%%%%%%%%%
 \pinlabel (c) at 720 -10
  \pinlabel $\Bigg{\{}$ at 630 640
 \pinlabel $k_1-l$ at 580 640
 \pinlabel $\Biggr\}$ at 800 640
 \pinlabel $l$ at 820 640
 \pinlabel $\textcolor{blue}{L_1}$ at 820 690
 \pinlabel $-1$ at 820 540
 \pinlabel $-1$ at 820 500
 \pinlabel $\Biggr\}$ at 820 440
 \pinlabel $\Bigg\{$ at 640 440
 \pinlabel $m$ at 610 440
 \pinlabel $n-m$ at 870 440
 \pinlabel $-1$ at 820 370
 \pinlabel $k_2-2$ at 880 290
 \pinlabel $\vdots$ at 730 290
 \pinlabel $-1$ at 830 190
 \pinlabel $-1$ at 830 150
 \pinlabel $+1$ at 830 120
 \pinlabel $+1$ at 830 100
 \pinlabel $\textcolor{red}{L_2}$ at 850 70
\tiny
 \pinlabel${k_1-m}$ at -40 280 
\endlabellist
    \centering
    \includegraphics[scale=0.3]{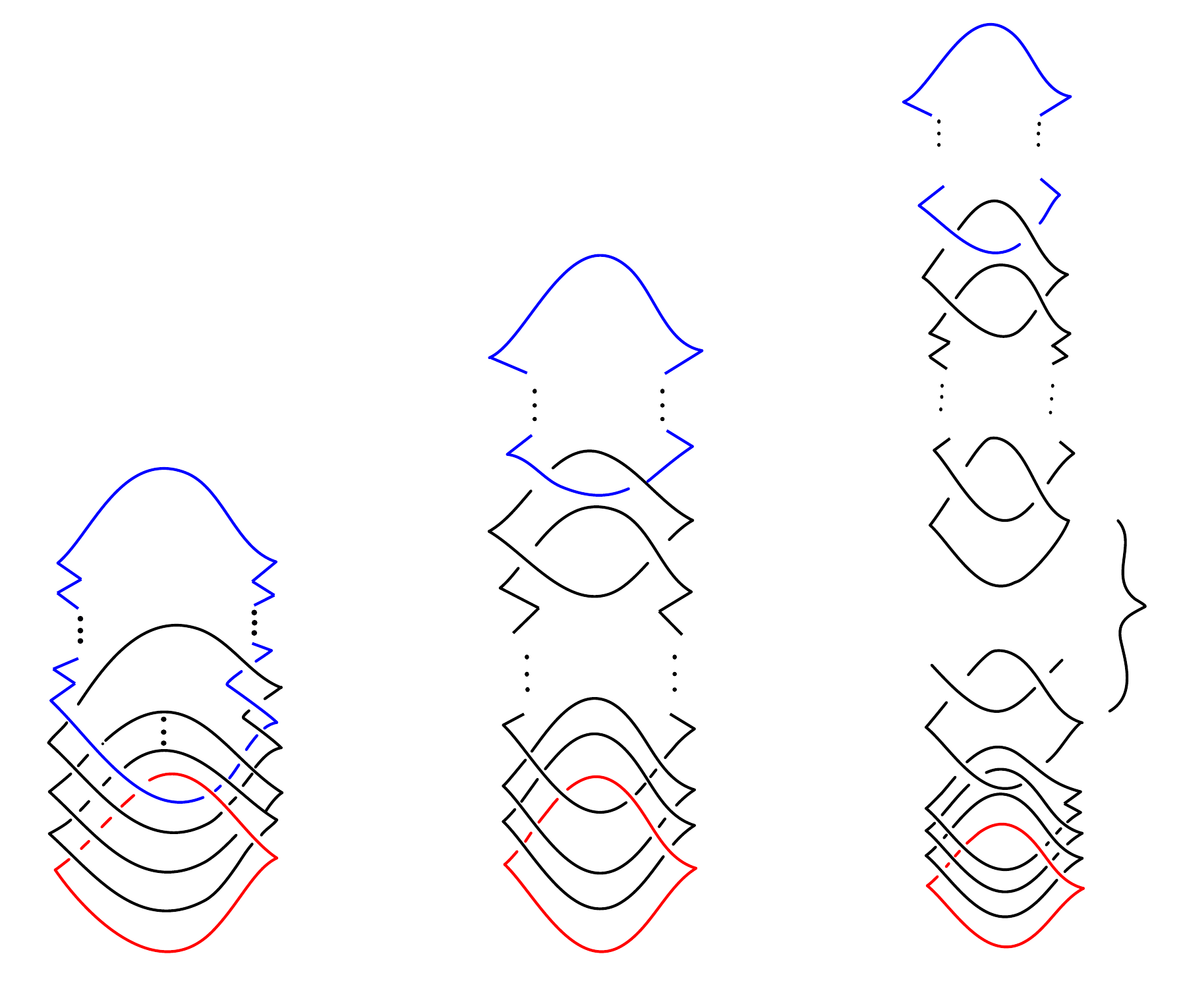}
    \caption{The non-loose Hopf link with $L_1$ loose and $L_2$ non-loose with $\tbr=-(k_1+1)+\frac{n+1}{2n+1}$ and $\tbr=k_2+\frac{2}{2n+1}$. (a) $2(k_1+1)$ candidates with $k_1=0, k_2\geq 0$. (b) $(k_1+1)(n+2)$ candidates with $k_1\geq 0, k_2=1$ and (c) $2(k_1+1)(n+1)$ candidates with $k_1\geq 0, k_2\geq 2.$}
    \label{fig:loosenonlooseL(2n+1)}
\end{figure}
%%%%%%%%%%%%%%%%%%%%%%%%%%%%%%%%%%%%%

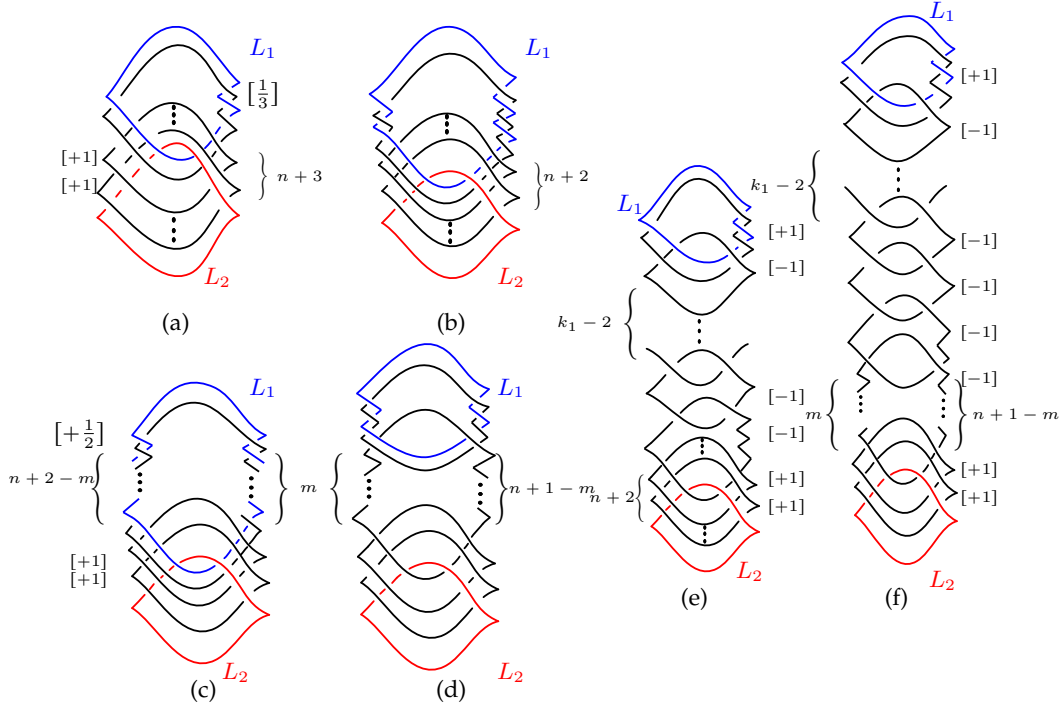
\begin{figure}[!htbp]
\labellist
\small\hair 2pt

\pinlabel{(a)} at 105 400
\pinlabel (b) at 405 400
\pinlabel (c) at 135 0
\pinlabel (d) at 405 0
\pinlabel (e) at 670 100
\pinlabel (f) at 890 100

%%%%%%%%%%%%%%%%%%%%%%%%%%%%%%%%%%%%%%%%%%%%%%% (a)
\pinlabel $[\frac{1}{3}]$ at 200 650
\pinlabel $\textcolor{blue}{L_1}$ at 200 700

\pinlabel $\textcolor{red}{L_2}$ at 150 450

%%%%%%%%%%%%%%%%%%%%%%%%%%%%%%%%%%%%%%%%%%%%%%%%%(b)
\pinlabel $\textcolor{blue}{L_1}$ at 500 700
\pinlabel $\textcolor{red}{L_2}$ at 470 450

%%%%%%%%%%%%%%%%%%%%%%%%%%%%%%%%%%%%%%%%%%%%%%(e)
\pinlabel $\textcolor{blue}{L_1}$ at 600 530
\pinlabel $\Bigg\{$ at 600 400
\pinlabel $\textcolor{red}{L_2}$ at 730 130
%%%%%%%%%%%%%%%%%%%%%%%%%%%%%%%%%%%%%%%%%%%%
\pinlabel $\textcolor{blue}{L_1}$ at 200 330
\pinlabel $[+\frac{1}{2}]$ at 0 280
\pinlabel $\textcolor{red}{L_2}$ at 170 20
\pinlabel $\Bigg\{$ at 20 220
\pinlabel $\Bigg\}$ at 220 220

%%%%%%%%%%%%%%%%%%%%%%%%%%%%%%%%%%%%%%%%%
\pinlabel $\textcolor{red}{L_2}$ at 470 20
\pinlabel $\textcolor{blue}{L_1}$ at 470 330
\pinlabel $\Bigg\{$ at 280 220
\pinlabel $\Bigg\}$ at 460 220

%%%%%%%%%%%%%%%%%%%%%%%%%%%%%%%%%%%%%%%
\pinlabel $\textcolor{red}{L_2}$ at 940 120
\pinlabel $\textcolor{blue}{L_1}$ at 940 740
\pinlabel $\Bigg\{$ at 820 300
\pinlabel $\Bigg\}$ at 960 300
\pinlabel $\Bigg\{$ at 800 550

\tiny
\pinlabel$\Bigg\}$ at 200 560
\pinlabel $\Bigg\}$ at 500 550
\pinlabel $[+1]$ at 0 550
\pinlabel $[+1]$ at 0 580
\pinlabel $n+3$ at 240 560

%%%%%%%%%%%%%%%%%%%%%%%%%%%%%%%%%%
\pinlabel $n+2$ at 530 560
%%%%%%%%%%%%%%%%%%%%%%%%%%%%%%%%%%%
\pinlabel $[+1]$ at 770 500
\pinlabel $[-1]$ at 770 460
\pinlabel $[-1]$ at 770 320
\pinlabel $[-1]$ at 770 280
\pinlabel $[+1]$ at 770 230
\pinlabel $[+1]$ at 770 200
\pinlabel $\Bigg\{$ at 610 210
\pinlabel $n+2$ at 580 210
\pinlabel $k_1-2$ at 550 400
%%%%%%%%%%%%%%%%%%%%%%%%%%%%%%%%%
\pinlabel $[+1]$ at 10 140
\pinlabel $[+1]$ at 10 120
\pinlabel $n+2-m$ at -30 230
\pinlabel $m$ at 250 220
\pinlabel $n+1-m$ at 515 220
%%%%%%%%%%%%%%%%%%%%%%%%%%%%%%%
\pinlabel $[+1]$ at 980 210
\pinlabel $[+1]$ at 980 240
\pinlabel $n+1-m$ at 1020 300
\pinlabel $m$ at 800 300
\pinlabel $[-1]$ at 980 340
\pinlabel $[-1]$ at 980 390
\pinlabel $[-1]$ at 980 440
\pinlabel $[-1]$ at 980 490
\pinlabel $[-1]$ at 980 610
\pinlabel $[+1]$ at 980 670
\pinlabel $k_1-2$ at 760 550
\endlabellist
    \centering
    \includegraphics[width=0.8\linewidth]{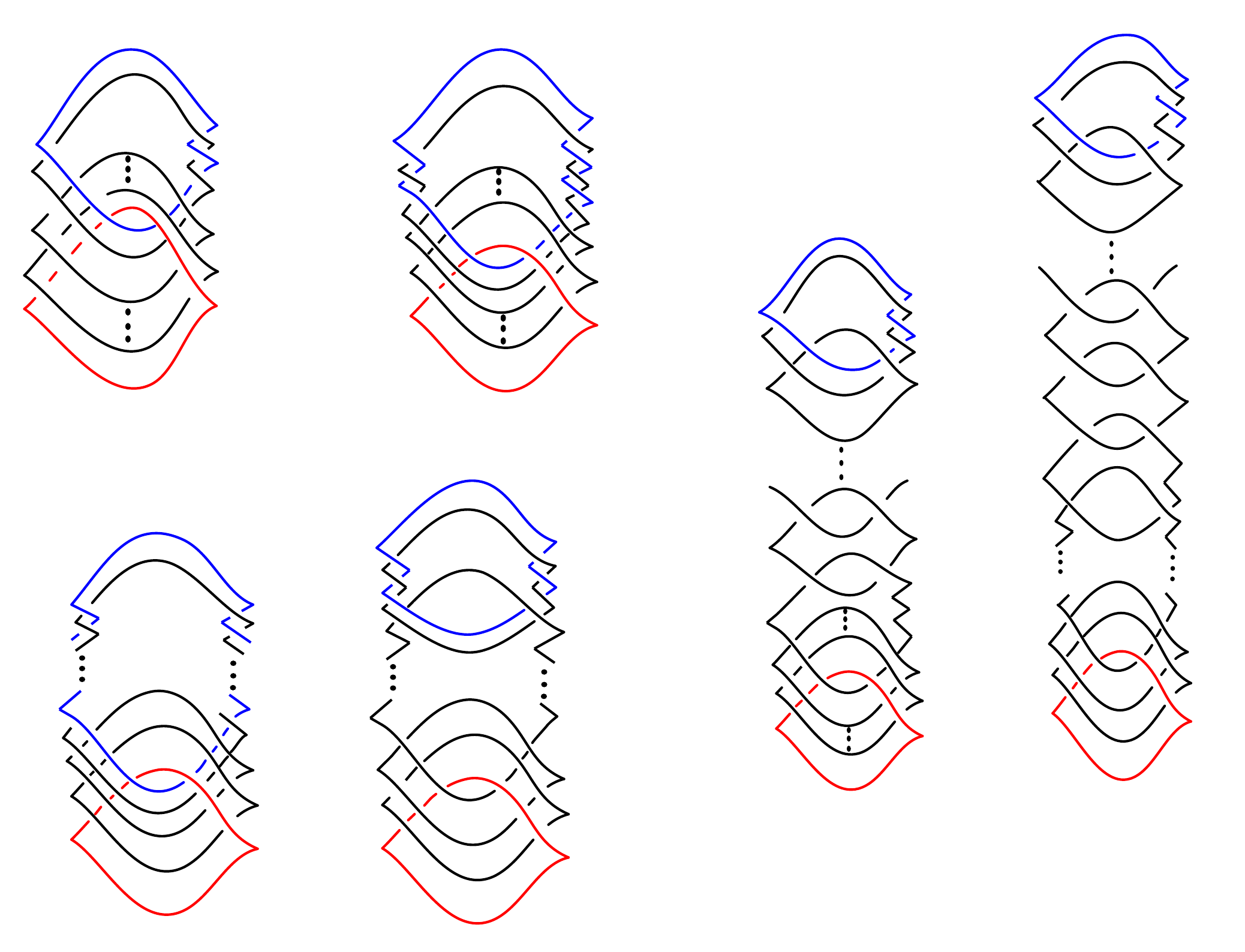}
    \caption{The non-loose Hopf link in $L(2n+1,2)$ where both $L_1$ and $L_2$ are loose with $\tbr(L_1)=k_1+\frac{n+1}{2n+1}$ and $\tbr(L_2)=k_2+\frac{2}{2n+1}$.(a) $2$ candidates with $k_1=k_2=0$ (b) $4$ candidates with $k_1=1, k_2=0$. (c) $n+3$ candidates with $k_1=0, k_2=1.$ (d) $3(n+2)$ candidates with $k_1=k_2=0.$  (e) $6$ candidates with $k_1\geq 2, k_2=0.$ (f) $2(n+2)$ candidates with $k_1\geq 2, k_2=1.$}
    \label{fig:looseloose1_L(2n+1,2)}
\end{figure}

%%%%%%%%%%%%%%%%%%%%%%%%%%%%%%%%%%%%

\begin{figure}[!htbp]
\labellist
\small\hair 2pt
\tiny
\pinlabel (a) at 120 -10
\pinlabel $\textcolor{red}{L_2}$ at 20 150

\pinlabel $k_1-2$ at 240 440
\pinlabel $[+1]$ at 25 195
\pinlabel $[+1]$ at 25 215
\pinlabel $[-1]$ at 20 270
\pinlabel $[-1]$ at 20 320 
\pinlabel $[-1]$ at 20 380
\pinlabel $[-1]$ at 20 530
\pinlabel $[+\frac{1}{2}]$ at 20 580
\pinlabel $\textcolor{blue}{L_1}$ at 30 720
\pinlabel $\Biggr\{$ at 25 630
\pinlabel $\Biggl\}$ at 170 630
\pinlabel $m$ at 5 630
\pinlabel $n+1-m$ at 220 630
%%%%%%%%%%%%%%%%%%%%%%%%%%%%%%%%%%%%%%%%%%%%%%%%%%%%%%%%%%%%%
\pinlabel (b) at 380 -10
\pinlabel $k_1-2$ at 500 390
\pinlabel $\textcolor{red}{L_2}$ at 300 100
\pinlabel $[+1]$ at 300 150
\pinlabel $[+1]$ at 300 170
\pinlabel $[-1]$ at 300 230
\pinlabel $[-1]$ at 300 290
\pinlabel $[-1]$ at 300 340
\pinlabel $[-1]$ at 300 500
\pinlabel $[-1]$ at 300 630
\pinlabel $\Bigg\{$ at 300 570
\pinlabel $m$ at 280 570
\pinlabel $\Biggl\}$ at 450 570
\pinlabel $n-m$ at 490 570
\pinlabel $[+1]$ at 280 670
\pinlabel $\textcolor{blue}{L_1}$ at 300 720
%%%%%%%%%%%%%%%%%%%%%%%%%%%%%%%%%%%%%%%%%%%%%%%%%%%%%%%%%%%%%

\pinlabel (c) at 620 -10

\pinlabel $k_2-2$ at 775 270
\pinlabel $\textcolor{red}{L_2}$ at 550 30
\pinlabel $[+1]$ at 550 80
\pinlabel $[+1]$ at 550 110
\pinlabel $[-1]$ at 550 150
\pinlabel $[-1]$ at 550 190
\pinlabel $[-1]$ at 540 330
\pinlabel $[-1]$ at 540 370
\pinlabel $[-1]$ at 540 500
\pinlabel $[-1]$ at 540 530
\pinlabel $\Biggr\{$ at 550 410
\pinlabel $m$ at 530 410
\pinlabel $\Biggl\}$ at 670 410
\pinlabel $n-m$ at 720 410
\pinlabel $[-1]$ at 530 660
\pinlabel $k_1-2$ at 760 590
\pinlabel $[+1]$ at 530 710
\pinlabel $\textcolor{blue}{L_1}$ at 580 770

\endlabellist
    \centering
    \includegraphics[width=0.7\linewidth]{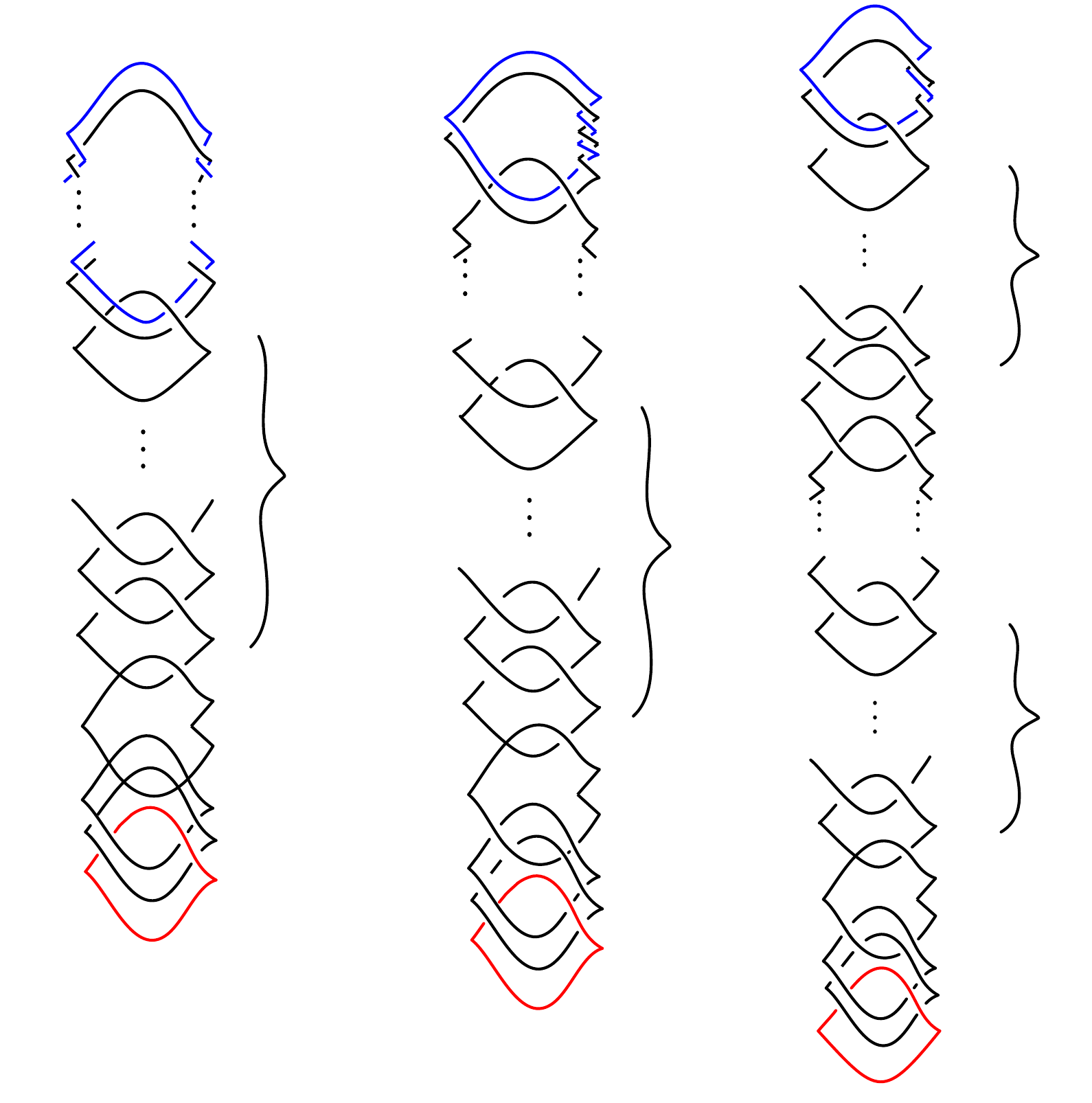}
    \caption{ The non-loose Hopf link in $L(2n+1,2)$ with both components loose. (a) $2(n+2)$  candidates with $k_1=0, k_2\geq 2$  (b) $6(n+1)$  candidates with $k_1=1, k_2\geq 2.$ (c) $8(n+1)$ candidates with $k_1\geq 2 \ \text{and}\ k_2\geq 2.$ }
    \label{fig:looseloose2_L(2n+1,2)}
\end{figure}

\subsection{General surgery diagrams}
Here we give an algorithm on how to find the contact surgery representations of the rational Hopf links. Note that, this is just one way to find the representatives. We could have infinitely many possibilities for the Legendrian Hopf links but all of them are related by contact handle slide and Legendrian isotopy. We will give a detailed algorithm for the first and the second case and the rest follow similarly.

\subsection*{ Case 1: small $\cup$ large slope}
\label{subsec:small_large_algo}
 We give a detailed analysis for $\frac{p}{q}=[a_0,a_1,\cdots , a_n]$. When $k_2=0$ and $n=1$ the analysis is slightly different and we leave it to the reader. Note that, when considering small slope for $K_1$ and large slope for $K_2$, the path that describes the contact structure of the complement of $L_1\sqcup L_2$ has endpoints $s_{k_1}$ and $s_{k_2}$ where $s_{k_1}<0$ but $s_{k_2}>0$, thus we have a non-uniform twisting in the complement. To solve this issue, we apply $k_2+1$ fold Rolfsen twist to the slopes. The $k_2+1$ fold Rolfsen twist matrix is given by $$\begin{pmatrix}
    1 & -(k_2+1)\\
    0 &1
\end{pmatrix}$$

\begin{figure}
\labellist
\small\hair 2pt
\pinlabel $\vdots$ at 100 550
\pinlabel $\vdots$ at 0 550
\pinlabel $\vdots$ at 410 550
\pinlabel $\vdots$ at 300 550
\pinlabel $\vdots$ at 730 650
\pinlabel $\vdots$ at 740 250
\pinlabel $\vdots$ at 850 600
\tiny
\pinlabel (a) at 130 -20
\pinlabel $\textcolor{red}{L_2}$ at 20 70
\pinlabel $[+1]$ at 20 95
\pinlabel $[+1]$ at 20 120
\pinlabel $\Bigg\}$ at 210 120
\pinlabel $\textcolor{red}{|a_0-1|}$ at 260 120
\pinlabel $|a_1+1|$ at 0 220
\pinlabel $|a_2+1|$ at 0 370
\pinlabel $[-1]$ at 220 600
\pinlabel $[-1]$ at 220 370
\pinlabel $[-1]$ at 220 220
\pinlabel $|a_n+2|$ at -10 600
\pinlabel $\textcolor{blue}{L_1}$ at 20 700

%%%%%%%%%%%%%%%%%%%%%%%%%%%%%%%%%%%%%%%%%%

\pinlabel (b) at 430 -20
\pinlabel $\textcolor{red}{L_2}$ at 320 70
\pinlabel $[+1]$ at 540 95
\pinlabel $[+1]$ at 540 120

\pinlabel $|a_0|$ at 310 220
\pinlabel $|a_2+1|$ at 310 370
\pinlabel $[-1]$ at 520 600
\pinlabel $[-1]$ at 520 370
\pinlabel $[-1]$ at 530 200
\pinlabel $|a_n+2|$ at 310 600
\pinlabel $\textcolor{blue}{L_1}$ at 320 700

%%%%%%%%%%%%%%%%%%%%%%%%%%%%%%%%%%%%%%%%%%%
\pinlabel (c) at 730 -20
\pinlabel $\textcolor{red}{L_2}$ at 650 50
\pinlabel $[+1]$ at 650 70
\pinlabel $[+1]$ at 650 90
\pinlabel $[-1]$ at 650 110
\pinlabel $[-1]$ at 650 160

\pinlabel $[-1]$ at 650 310
\pinlabel $[-1]$ at 650 410
\pinlabel $[-1]$ at 650 510
\pinlabel $[-1]$ at 650 650
\pinlabel $\textcolor{blue}{L_1}$ at 650 750
\pinlabel $k_2-2$ at 900 240
\pinlabel $|a_0+1|$ at 850 400
\pinlabel $|a_1+2|$ at 850 520

\pinlabel $|a_n+2|$ at 850 670

\endlabellist
    \centering
    \includegraphics[width=0.6\linewidth]{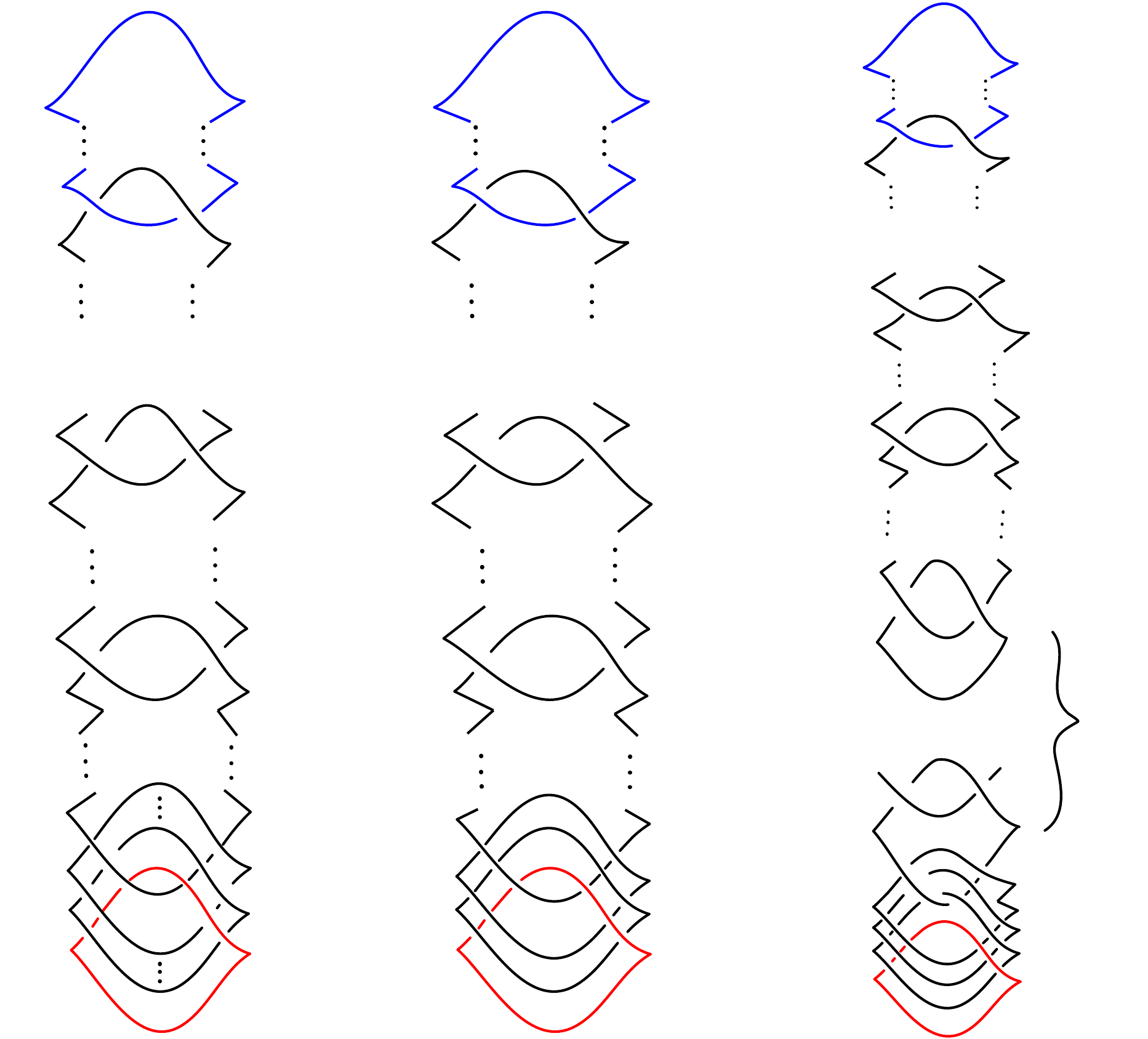}
    \caption{Legendrian repressentaives of non-loose Hopf links in $L(p,q)$ where $L_1$ is loose and $L_2$ non-loose. (a) $k_2=0$ and $n\geq 2$ (b)$k_2=1$ (c) $k_2\geq 2$. In all the cases $k_1\geq 0$. The value $|a_i+1|$ in the diagram denotes the number of stabilization of that component.}
    \label{fig:general_loose_non-loose}
\end{figure}

\begin{figure}[!htbp]
    \centering
    \labellist
    \small\hair 2pt
    %%%%%%%%%%%%%%%%%%%%%%%%%%%%%%%%%%%%(a)
     \pinlabel $\}$ at 200 640
      \pinlabel $\vdots$ at 230 380
      \pinlabel $\vdots$ at 110 480
      %%%%%%%%%%%%%%%%%%%%%%%%%%%%%%%%%%%%%
        \pinlabel $\vdots$ at 580 380
      \pinlabel $\vdots$ at 460 480
      %%%%%%%%%%%%%%%%%%%%%%%%%%%%%%%%%%%%%%%%%
        \pinlabel $\vdots$ at 780 460
      \pinlabel $\vdots$ at 920 380
    \tiny
    %%%%%%%%%%%%%%%%%%%%%%%%%%%%%%%%%%%%%%%%%%%%%%%%(a)
   \pinlabel (a) at 120 -20
    \pinlabel $\textcolor{blue}{L_1}$ at 0 700
   
    \pinlabel $\textcolor{red}{|a_n|}$ at 230 640
    \pinlabel $[+1]$ at 0 640
    \pinlabel $[+1]$ at 0 620
    \pinlabel $|a_{n-1}+1|$ at 230 580
    \pinlabel $|a_{n-2}+2|$ at 230 480
    \pinlabel $[-1]$ at 0 480
    \pinlabel $|a_1+2|$ at 230 300
    \pinlabel $[-1]$ at 0 300
    \pinlabel $|a_0+2|$ at 230 190
    \pinlabel $[-1]$ at 0 190
    \pinlabel $\textcolor{red}{L_2}$ at 230 20
    \pinlabel $\Bigg\{$ at 30 80
    \pinlabel $k_2-1-m$ at -40 80
    \pinlabel $\Bigg\}$ at 210 80
    \pinlabel $m$ at 230 80

    %%%%%%%%%%%%%%%%%%%%%%%%%%%%%%%%%%%%%%%%%%%%%%(b)
    \pinlabel (b) at 480 -20
    \pinlabel $\textcolor{blue}{L_1}$ at 360 700
    \pinlabel $[+1]$ at 360 640
    \pinlabel $|a_n|$ at 560 620
    \pinlabel $|a_{n-1}+2|$ at 580 500
    \pinlabel $[-1]$ at 360 500
    \pinlabel $|a_1+2|$ at 580 300
    \pinlabel $[-1]$ at 360 300
    \pinlabel $|a_0+2|$ at 580 190
    \pinlabel $[-1]$ at 360 190
    \pinlabel $\Bigg\{$ at 390 80
    \pinlabel $k_2-1-m$ at 320 80
    \pinlabel $\Bigg\}$ at 560 80
    \pinlabel $m$ at 580 80
    \pinlabel $\textcolor{red}{L_2}$ at 560 20
    %%%%%%%%%%%%%%%%%%%%%%%%%%%%%%%%%%%%%%%%%%%%%%%%%%%(c)
    \pinlabel (c) at 790 -20
    \pinlabel $k_1-2$ at 920 610
    \pinlabel $\textcolor{blue}{L_1}$ at 700 770
    \pinlabel $[+1]$ at 680 730
    \pinlabel $[-1]$ at 680 680
    \pinlabel $[-1]$ at 680 550
    \pinlabel $[-1]$ at 680 450
    \pinlabel $|a_n+1|$ at 920 450
    \pinlabel $[-1]$ at 680 300
    \pinlabel $|a_1+2|$ at 920 300
    \pinlabel $|a_0+2|$ at 920 200
    \pinlabel $[-1]$ at 680 200
    \pinlabel $\Bigg\{$ at 680 80
    \pinlabel $m$ at 650 80
    \pinlabel $\Bigg\}$ at 880 80
    \pinlabel $k_2-1-m$ at 950 80
    \pinlabel $\textcolor{red}{L_2}$ at 900 20
    
    \endlabellist
    \includegraphics[scale=0.3]{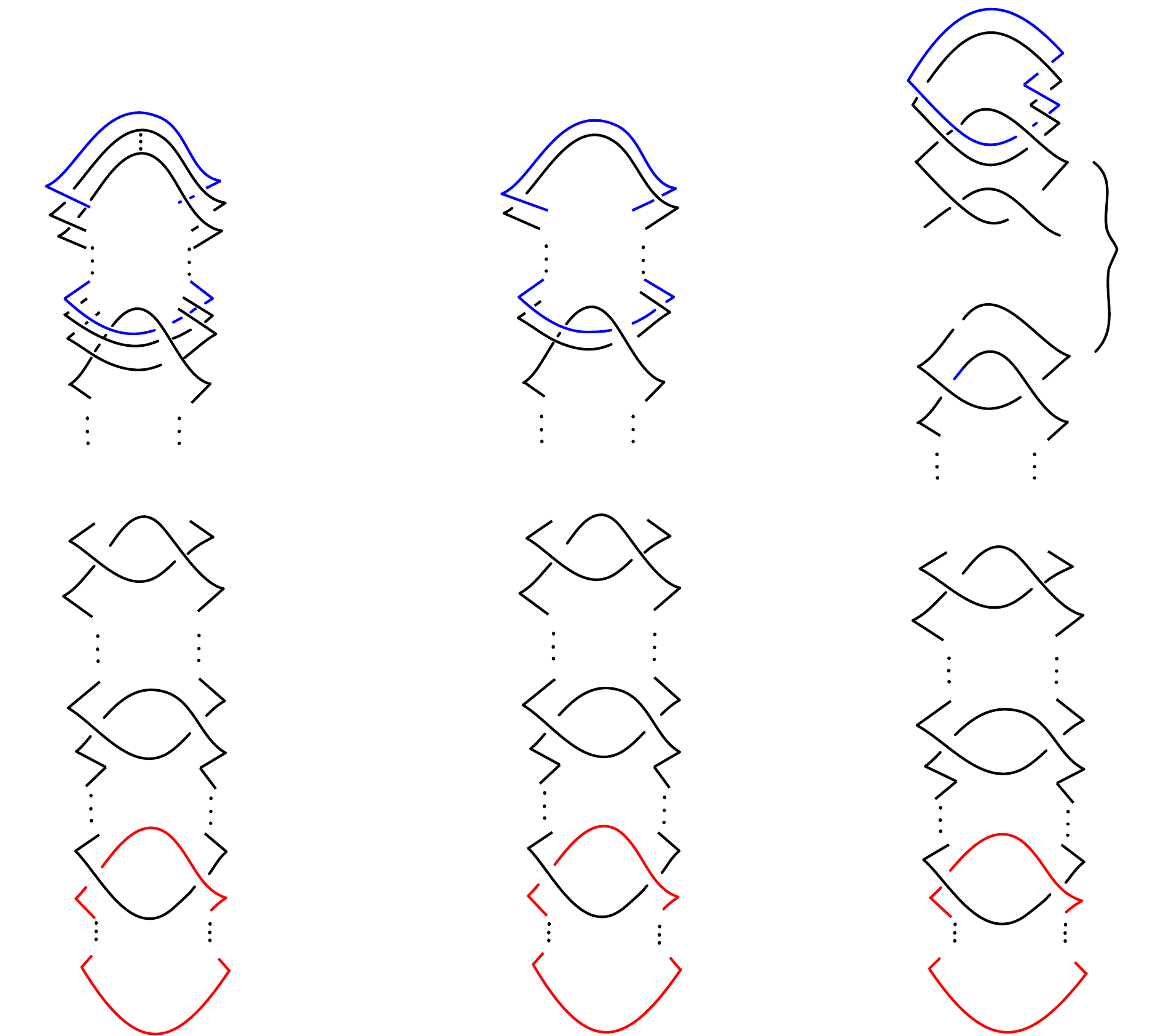}
    \caption{Non-loose Hopf links in $L(p,q)$ where $L_1$ is non-loose and $L_2$ is loose. (a) $k_1=0$ (b) $k_1=1$, $n\geq 0$ (c) $k_1\geq 2$ and $-p/q=[a_0,a_1,\cdots a_n].$ In all the cases $k_2\geq 1$.  }
    \label{fig:nonloose_loose_general}
\end{figure}
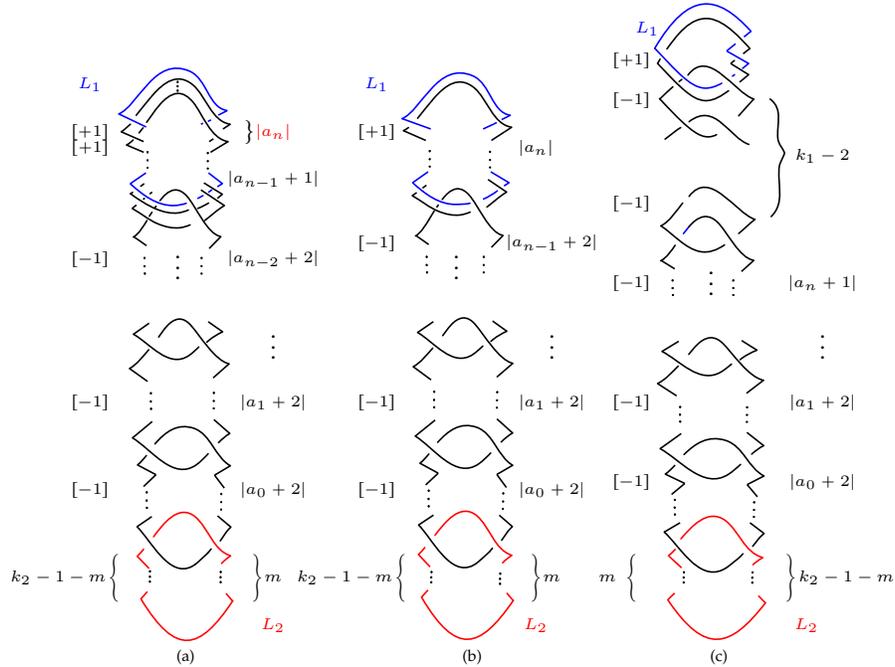

\begin{figure}[!htbp]
    \centering
    \labellist
     \pinlabel (i) at -20 1230
     \pinlabel (ii) at -20 1050
     \pinlabel (iii) at -20 860
     \pinlabel (iv) at -20 700
     \pinlabel (v) at -20 470
\pinlabel (vi) at -20 250
     \pinlabel (vii) at -20 100
    \tiny
    %%%%%%%%%%%%%%%%%%%%%%%%%%%%%%(vii)
    \pinlabel $\textcolor{blue}{L_1}$ at 120 90
    \pinlabel $a_n$ at 160 110
    \pinlabel $a_{n-1}$ at 240 110
    \pinlabel $\dots$ at 320 50
    \pinlabel $a_1$ at 440 110
    \pinlabel $a_0$ at 500 110
    \pinlabel $\textcolor{red}{L_2}$ at 550 90
    %%%%%%%%%%%%%%%%%%%%%%%%%%%%%%%%%%%%%%%%%%% (i)
 
    \pinlabel $\textcolor{blue}{L_1}$ at 160 1190
    \pinlabel $-1$ at 210 1180
    \pinlabel $\ddots$ at 130 1240
    \pinlabel $a_{n-1}+1$ at 95 1170
    \pinlabel $\underbrace{a_{n-1}+1/\dots /a_{n-1}+1}$ at 100 1100
    \pinlabel $|a_{n}|$ at 100 1080
    \pinlabel $a_{n-2}$ at 260 1240
    \pinlabel $\dots$ at 375 1180
    \pinlabel $a_1$ at 460 1240
    \pinlabel $a_0$ at 540 1240
    \pinlabel $\textcolor{red}{L_2}$ at 640 1220

    %%%%%%%%%%%%%%%%%%%%%%%%%%%%%%%%%%%%%%%%%% (ii)
    \pinlabel $\textcolor{blue}{L_1}$ at 175 1020
    \pinlabel $1$ at 120 1030
    \pinlabel $1$ at 120 970
    \pinlabel $\ddots$ at 150 1060
    \pinlabel $\Bigg\{$ at 80 1000
    \pinlabel $|a_{n-1}+1|$ at 30 1000
    \pinlabel $+1$ at 240 960
    \pinlabel $a_{n-2}+1$ at 290 1070
    \pinlabel $\dots$ at 375 1000
    \pinlabel $a_1$ at 460 1070
     \pinlabel $\underbrace{1/1/\dots/\textcolor{red}{1}}$ at 80 930
     \pinlabel $|a_n|$ at 80 910
    \pinlabel $a_0$ at 540 1070
    \pinlabel $\textcolor{red}{L_2}$ at 640 1030
    %%%%%%%%%%%%%%%%%%%%%%%%%%%%%%%%%%%%%%%%%%%% (iii)
     \pinlabel $\textcolor{blue}{L_1}$ at 115 870
    \pinlabel $+1$ at 230 780
    \pinlabel $\Bigg\{$ at 50 810
    \pinlabel $|a_{n-1}+1|$ at 0 810
    \pinlabel $1$ at 105 820
      \pinlabel $1$ at 105 800
       \pinlabel $\underbrace{1/1/\dots/\textcolor{red}{1}}$ at 80 740
       \pinlabel $|a_n|$ at 80 720
    \pinlabel $a_{n-2}+1$ at 290 900
    \pinlabel$\ddots$ at 140 890
    \pinlabel $\dots$ at 365 830
    \pinlabel $a_1$ at 460 900
    \pinlabel $a_0$ at 540 900
    \pinlabel $\textcolor{red}{L_2}$ at 640 850
    %%%%%%%%%%%%%%%%%%%%%%%%%%%%%%%%%%%%%%%%%%%% (iv)
     \pinlabel $\textcolor{blue}{L_1}$ at 70 620
    \pinlabel $+1$ at 290 600
    \pinlabel $a_{n-2}+1$ at 320 700
    \pinlabel $\underbrace{2/2/\dots/2}$ at 80 550
    \pinlabel $|a_n+1|$ at 80 530
    \pinlabel $1$ at 170 710
    \pinlabel $1$ at 220 710
    \pinlabel $\dots$ at  15 620
    \pinlabel $+1$ at 120 620
    \pinlabel $\textcolor{red}{1}$ at 250 700
    \pinlabel $\overbrace{}$ at 195 720
    \pinlabel $|a_{n-1}+1|$ at 195 740
    \pinlabel $\dots$ at 420 640
    \pinlabel $a_1$ at 490 700
    \pinlabel $a_0$ at 560 700
    \pinlabel $\textcolor{red}{L_2}$ at 640 590                          
    %%%%%%%%%%%%%%%%%%%%%%%%%%%%%%%%%%%%%%%% (v)
     \pinlabel $\textcolor{blue}{L_1}$ at 60 380
     \pinlabel $\textcolor{red}{a_{n-1}+1}$ at 230 510
    \pinlabel $-1$ at 160 380
    \pinlabel $+1$ at 300 380
    \pinlabel $\underbrace{1/1/\dots/1}$ at 70 320
    \pinlabel $|a_n+1|$ at 70 300
    \pinlabel $a_{n-2}+1$ at 330 490
    \pinlabel $\dots$ at 430 420
    \pinlabel $a_1$ at 500 490
    \pinlabel $a_0$ at 560 490
    \pinlabel $\textcolor{red}{L_2}$ at 640 470
    %%%%%%%%%%%%%%%%%%%%%%%%%%%%%%%%%%%%%%%%% (vi)
    
    %\pinlabel $-1$ at 240 930
    \pinlabel $a_{n-2}+1$ at 330 300
    \pinlabel $a_{n-1}+1$ at 230 300
    \pinlabel $+1$ at 280 170
    \pinlabel $\dots$ at 415 220
    \pinlabel $a_1$ at 500 300
    \pinlabel $a_0$ at 560 300
    \pinlabel $\textcolor{red}{L_2}$ at 640 260
     \pinlabel $\textcolor{blue}{L_1}$ at 80 240
       \pinlabel $a_n$ at 125 160
    \endlabellist
    \includegraphics[width=0.6\linewidth]{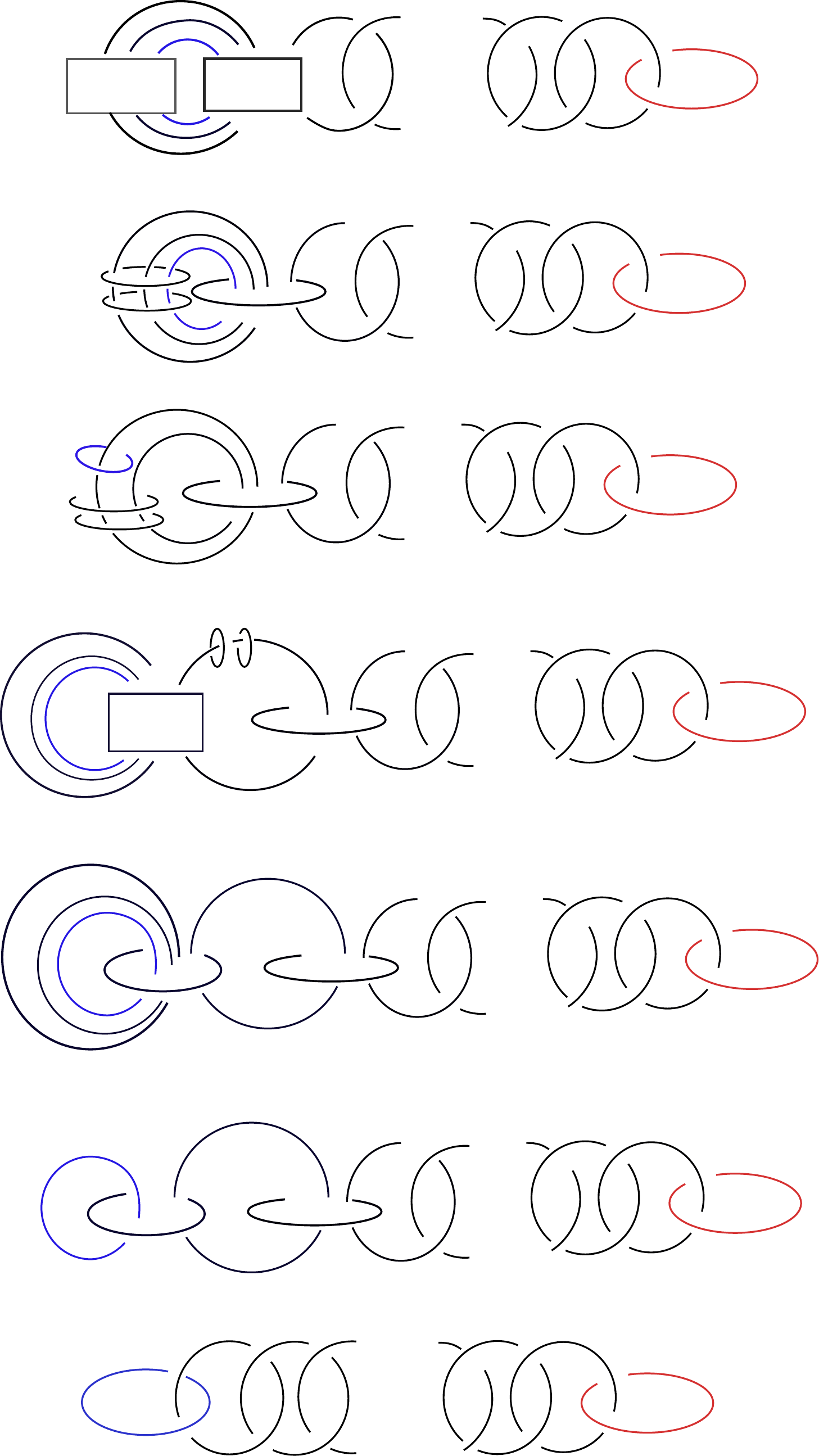}
    \caption{A sequence of Kirby moves for candidate Figure~\ref{fig:nonloose_loose_general} (a). }
    \label{fig:kirbymoves}
\end{figure}

This matrix turns $s_{k_2}=\frac{1}{k_2}$ to $-1$ and fixes $0$. First we do it for $k_1\geq 0, k_2=0$ thus we apply $1$ Rolfsen twist.  Appling $1$ Rolfsen twist turns $-\frac{p}{q}=[a_0,a_1,\cdots, a_n]$ to $[-1, a_0-1,a_1,\cdots, a_n]$. Now note that, a standard neighborhood of $K_1$ will have meridional slope $[-1, a_0-1,
a_2\cdots a_n]$ and dividing slope $s_{k_1}.$ We start with a max $\tb$ Legendrian unknot in $S^3$. It has standard neighborhood with meridional slope $\infty$ and dividing slope $-1$. First we do $|a_0-1|$ $+1$ surgeries on it. This will change the meridional slope to $[-1, a_0-1]$. After this we take a Legendrian push-off of the surgery curve and do a $|a_1+1|$ stabilizations and finally a $-1$ surgery. This makes the new meridian $[-1,a_0-1, a_1]$ and the dividing curve $[-1,a_0-1, a_1+1]$. We keep doing a sequence of stablizations and $-1$ surgery as shown in Figure~\ref{fig:general_loose_non-loose} (a). A Legendrian push-off of the last surgery curve will have a standard neighborhood with meridional slope $[-1, a_0-1, a_1\cdots, a_n]$ and dividing slope $[-1,a_0-1, a_1,\cdots, a_n+1]$. Note that, this is the dividing slope corresponding to $k_1=0$. Thus we can stabilize this push-off $k_1$ times to get all the representaives of $K_1$. For $K_2$, note that, the meridian of the standard unknot has dividing slope $-1$ and meridional slope $0$ which are actually dividing slope $\frac{1}{k_2}$ and meridinonal slope $0$ before the Rolfsen twist. Thus we get our desired Hopf link. It is easy to see that $ L_2$ is non-loose as a sequence of $-1$ surgeries on $L_2$ will cancel the $+1$ surgeries and we will get a tight manifold in the complement. We have a total $|a_1||a_2+1|\cdots|a_n+1|(k_1+1)$ candidates in this case. 

For $n=1$, as mentioned the analysis will be slightly different. We just replace $|a_1+2|$ by $|a_1+1|$ in (a) of Figure~\ref{fig:general_loose_non-loose}.

\begin{figure}[!htbp]
    \centering
    \labellist
    \pinlabel $\textcolor{blue}{L_1}$ at 50 300
     \pinlabel $\textcolor{red}{L_2}$ at 300 50
  \pinlabel (a) at 170 -10
  %%%%%%%%%%%%%%%%%%%%%%%%%%%%
  \pinlabel (b) at 520 -10
   \pinlabel $\textcolor{blue}{L_1}$ at  400 400
    \pinlabel $\textcolor{red}{L_2}$ at 650 50
   %%%%%%%%%%%%%%%%%%%%%%%%%%%
   \pinlabel (c) at 900 -10
    \pinlabel $\textcolor{blue}{L_1}$ at 800 800
     \pinlabel $\textcolor{red}{L_2}$ at 980 50
        \pinlabel $\vdots$ at 890 570
        \pinlabel $\vdots$ at 1020 500
    \tiny
    %%%%%%%%%%%%%%%%%%%%%%%%%%%%%%%%%%%%%%%(a)
    \pinlabel $[\frac{1}{|a_0-2|}]$ at 0 200
    \pinlabel $\Bigg\{$ at 30 130
    \pinlabel $|a_0-2|$ at -10 130
    \pinlabel $[+1]$ at 300 160
    \pinlabel $\vdots$ at 300 130
    \pinlabel $[+1]$ at 300 100
    %%%%%%%%%%%%%%%%%%%%%%%%%%%%%%%%%%%%% (b)
    \pinlabel $[\frac{1}{|a_2|}]$ at 400 300
    \pinlabel $[+1]$ at 390 100
    \pinlabel $[+1]$ at 390 140
    \pinlabel $\Bigg\}$ at 645 120
    \pinlabel $|a_0-1|$ at 690 120
    \pinlabel $|a_1|$ at 650 280
    %%%%%%%%%%%%%%%%%%%%%%%%%%%%%%%%%%%%%(c)
    \pinlabel $[\frac{1}{|a_n|}]$ at 770 670
    \pinlabel $|a_{n-1}+1|$ at 1020 670
    \pinlabel $[-1]$ at 770 570
    \pinlabel $|a_{n-2}+2|$ at 1020 570
    \pinlabel $[-1]$ at 770 400
   \pinlabel $|a_2+2|$ at 1020 400
   \pinlabel $[-1]$ at 770 250
   \pinlabel $|a_1+1|$ at 1020 230
   \pinlabel $\Bigg\}$ at 1000 120
   \pinlabel $[+1]$ at 770 100
   \pinlabel $[+1]$ at 770 140
   \pinlabel $\textcolor{red}{|a_0-1|}$ at 1050 120
   
     \endlabellist
    \includegraphics[width=0.75\linewidth]{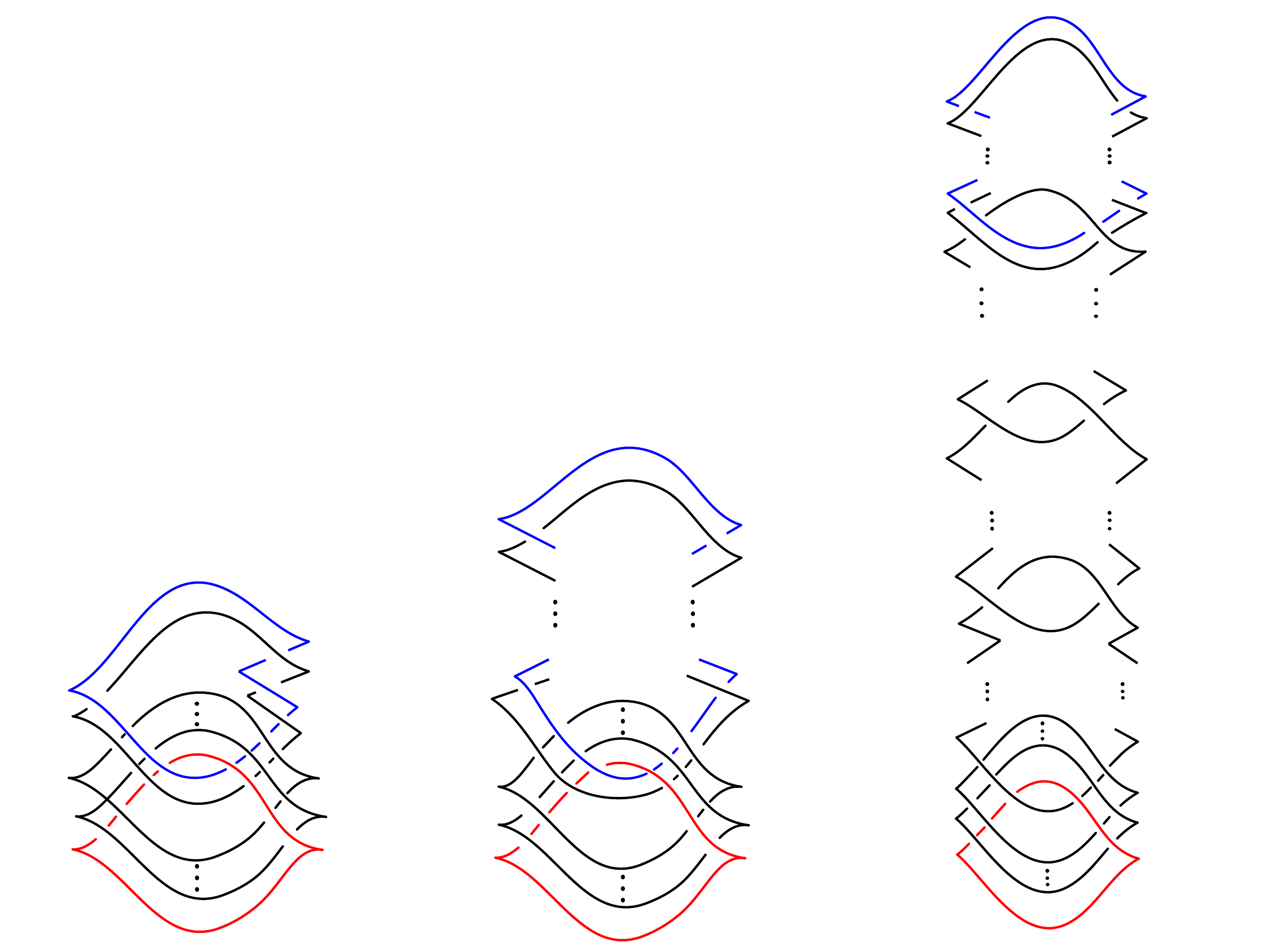}
    \caption{These are all non-loose Legendrian representatives in $L(p,q)$ with $L_1$ and $L_2$ both loose and $\tbr(L_1)=\frac{p'}{p}$ and $\tbr(L_2)=
    \frac{q}{p}$ . (a) For $n=1$ we have $2$ representatives. (b) For $n=2$ there are $|a_1-1|$ representatives and (c) For $n\geq 3$ we will have $|a_1||a_2+1|\cdots|a_{n-2}+1||a_{n-1}|$ non-loose representatives. }
    \label{fig:Loose_General1}
\end{figure}

\subsubsection*{For $k_2=1$} Here we need to apply a $2$ fold Rolfsen twist and the meridional slope in this case changes to $[-1,-2,a_0-1,a_1,\cdots, a_n]$. Like the previous case, we start with a standard unknot in $S^3$ with meridional slope $\infty$ and dividing slope $-1$. After 2 $+1$ surgeries the meridional slope becomes $[-1, -2]$. Now stabilizing a Legendrian push-off of this surgery curve $|a_0|$ times the dividing slope changes to $[-1,-2,a_0]$ and doing $-1$ surgery on this curve makes our new meridional slope $[-1,-2, a_0-1].$ We keep repeating this sequence by  stabilizing $|a_i+2|$ times and then a  $-1$ surgery until we reach the meridional slope $[-1,-2,a_0-1,a_1\cdots, a_n]$ and  dividing slope $[-1,-2,a_0-1,a_1, \cdots ,a_n+1]$. Then following the same procedure as before we get exactly $|a_0-1||a_1+1|\cdots |a_n+1| (k_1+1)$ non-loose candidates. Check (b) of Figure~\ref{fig:general_loose_non-loose}.

\subsubsection*{For $k_2\geq 2$} In this case, we need to apply a $k_2+1$ fold Rolfsen twist which brings our lower meridian to $[-1,\underbrace{-2,-2,\cdots -2}_{k_2}, a_0-1, a_1, \cdots, a_n]$. Doing $2$ consecutive $+1$ surgeries on the standard unknot changes our meridional slope to $[-1,-2].$ Next we take a Legendrian push-off of this surgery curve and stabilize it once and then do a $-1$ surgery. Our meridian now becomes $[-1,-2,-2]$ and dividing slope $0$.  Now we do a sequence of $-1$ surgeries of length $k_2-2$ on the chain of unknots that as shown (Note that, each of these unknots are actually a Legendrian push-off of the once stabilized surgery curve. One could see it via a contact handle slide.) After performing the $-1$ surgery on the chain our new meridional slope becomes $[-1,\underbrace{-2, -2,\cdots, -2}_{k_2}]$. Once we take a Legendrian push-off of the last surgery curve, a standard neighborhood of this has meridional slope  $[-1,\underbrace{-2, -2,\cdots, -2}_{k_2}]$ and dividing slope $0$. If now we do a $-1$ surgery on this curve after stabilizing it $|a_0+1|$ times we see the new meridian and dividing slopes are  $[-1,\underbrace{-2, -2,\cdots, -2}_{k_2}, a_0-1]$ and $[-1,\underbrace{-2, -2,\cdots, -2}_{k_2}, a_0]$ respectively. Then we proceed as before. Note that, in this case we will have $2(k_1+1)|a_0||a_1+1||a_2+1|\cdots|a_n+1|$ non-loose representatives.

%%%%%%%%%%%%%%%%%%%%%%%%%%%%%%%%%%%%%%%%%%%%%%%%%%%%%%

\begin{figure}[!htbp]

\labellist
\small\hair 2pt
%%%%%%%%%%%%%%%%%%%%%%%%%%%%%%%%(a)
\pinlabel (a) at 200 440

\pinlabel $\textcolor{blue}{L_1}$ at 60 750
\pinlabel $\textcolor{red}{L_2}$ at 60 500
%%%%%%%%%%%%%%%%%%%%%%%%%%%%%%%%%%%%%%(b)
\pinlabel (b) at 200 -10
\pinlabel $\textcolor{red}{L_2}$ at 60 40
\pinlabel $\textcolor{blue}{L_1}$ at 60 350
%\pinlabel $\Bigg\}$ at 320 270
%%%%%%%%%%%%%%%%%%%%%%%%%%%%%%%%%%%%%%(c)
\pinlabel (c) at 570 -10
\pinlabel $\vdots$ at 560 570
\pinlabel $\textcolor{red}{L_2}$ at 500 20
\pinlabel $\textcolor{blue}{L_1}$ at 450 720
\pinlabel $\vdots$ at 700 530
%%%%%%%%%%%%%%%%%%%%%%%%%%%%%%%%%%%%%(d)
\pinlabel (d) at 970 -20
%\pinlabel $\vdots$ at 860 570
\pinlabel $\textcolor{red}{L_2}$ at 700 20
\pinlabel $\textcolor{blue}{L_1}$ at 650 720
\pinlabel $\vdots$ at 970 600
\pinlabel $\}$ at 1050 90
\pinlabel $\vdots$ at 960 340
%\pinlabel $\vdots$ at 10
\tiny
%%%%%%%%%%%%%%%%%%%%%%%%%%%%%%%%%(a)
\pinlabel $[+1]$ at 40 570
\pinlabel $[+1]$ at 40 530
\pinlabel $\Bigg\}$ at 300 550
\pinlabel $|a_0-3|$ at 350 550
\pinlabel $[+1]$ at 40 650
\pinlabel $[+1]$ at 40 680
%%%%%%%%%%%%%%%%%%%%%%%%%%%%%%%%%%%(b)
\pinlabel $[+1]$ at 40 300
\pinlabel $[+1]$ at 40 100
\pinlabel $[+1]$ at 40 140
\pinlabel $\Bigg\}$ at 310 120
\pinlabel $|a_0-1|$ at 360 120
\pinlabel $|a_1-1|$ at 360 270
%%%%%%%%%%%%%%%%%%%%%%%%%%%%%%%%%%%%(c)
\pinlabel $[+1]$ at 450 100
\pinlabel $[+1]$ at 450 130
\pinlabel $\Bigg\}$ at 650 110
\pinlabel $|a_0-1|$ at 700 110
\pinlabel $|a_1+1|$ at 700 220
\pinlabel $[-1]$ at 450 220
\pinlabel $|a_2+2|$ at 700 350
\pinlabel $[-1]$ at 450 350
\pinlabel $|a_{n-1}+2|$ at 700 570

\pinlabel $[-1]$ at 450 570
\pinlabel $|a_{n}|$ at 700 670
\pinlabel $[+1]$ at 450 670
%%%%%%%%%%%%%%%%%%%%%%%%%%%%%%%%%%%%%(d)
\pinlabel $[+1]$ at 850 70
\pinlabel $[+1]$ at 850 100

\pinlabel $|a_0-1|$ at 1100 90

\pinlabel $|a_1+1|$ at 1100 180
\pinlabel $[-1]$ at 850 180
\pinlabel $|a_2+2|$ at 1100 270
\pinlabel $[-1]$ at 850 270

\pinlabel $|a_{n}+1|$ at 1100 450
\pinlabel $|a_{n-1}+2|$ at 1100 370
\pinlabel $\vdots$ at 1100 320
\pinlabel $[-1]$ at 850 450
\pinlabel $[-1]$ at 850 520
\pinlabel $[-1]$ at 850 670
\pinlabel $k_1-2$ at 1120 590

\pinlabel $[+1]$ at 850 750

\endlabellist
    \centering
    \includegraphics[width=0.85\linewidth]{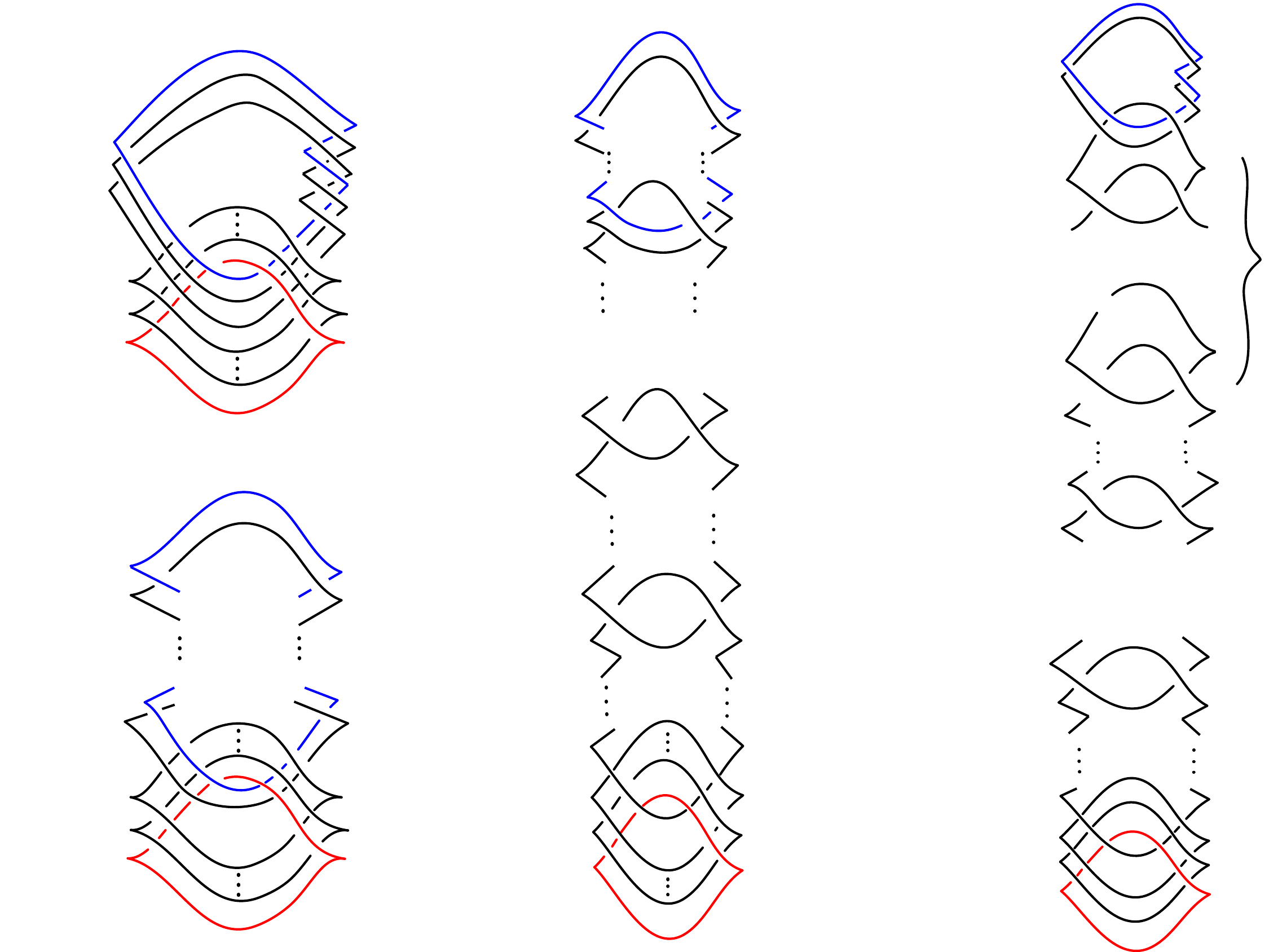}
    \caption{ All the non-loose Legendrian representatives of the Hopf link in $L(p,q)=[a_0,a_1,\cdots, a_n]$ with $\tbr(L_1)=k_1+\frac{p'}{q} \text{and}\tbr(L_2)=\frac{q}{p}$.\\ (a) $n=0$, $k_1=1$, (b) $n=1$, $k_1=1$ (c) $n\geq 2$, $k_1=1$ (d) $n\geq 0$,$k_1\geq 2$. }
    \label{fig:Loose_General2}
\end{figure}

\subsection*{large $\cup$ small slope} \label{subsec: large_small_algo} In this case we construct the diagrams for $k_1\geq 0$ and $k_2\geq 1$. (Note that, $k_1\geq 0\ \text{and}\ k_2=0$ overleaps with Case \ref{subsec:large_large_algo}. ) % We give a detailed description for $-\frac{p}{q}=[a_0, a_1,\cdots a_{n}]$ for $n\geq 2$. The diagrams for $n=1$  are slightly different but one could see it clearly by observing the complement structure from the Farey graph. Thus we omit the details here.

\subsubsection*{For $k_1=0$, $k_2\geq 1$} As usual we start with a standard unknot of slope $-1$ and stabilize it $|a_0+2|$ times. Note that, this is a standard neighborhood of an unknot with dividing slope $a_0+1$. A $-1$ surgery on a Legendrian divide on this neighborhood will change the meridian to $a_0$. Next we will do a series of stabilization and Legendrian surgery (as shown in Figure \ref{fig:nonloose_loose_general}) to bring us to slope $r$ and meridional slope $[a_0, \cdots a_{n-2}]$ . Once we are at dividing slope $r$ and meridional slope $[a_0, \cdots a_{n-2}]$, we do a $|a_{n-1}+1|$ stabilization and finally $|a_n|$ number of $+1$ surgeries to get our desired neighborhood of a Legendrian with dividing slope $-\frac{p''}{q}$ and meridinal slope $-\frac{p}{q}$. A Legendrian divide of this neighborhood is a component of the desired Hopf link.  Note that, the meridian of the standard unknot we started with is the other component with slope $-1$. Stabilizing this unknot $(k_2-1)$ number of times will give us all the non-loose Hopf links with $k_1=0$ and $k_2\geq 1$. One could see there are $|a_0+1||a_1+1|\cdots|a_{n-1}|k_2$ such candidates as desired. To see $K_1$ is non-loose, do a $-\frac{1}{a_n}$ surgery. This cancels all the $+1$ surgeries and thus the complement of $K_1$ is clearly tight( in fact Stein fillable, thus there cannot exist any Giroux torsion). To see this is indeed a Hopf link, one could do a series of Kirby moves and handle slides as shown in Figure ~\ref{fig:kirbymoves}.

\subsubsection*{For $k_1=1, k_2\geq 1$} Like the previous case, we will reach the dividing slope  $r$ and meridional slope $[a_0, a_1,\cdots, a_{n-2}]$. Next we will stabilize a Legendrian divide of this neighborhood $|a_{n-1}+2|$ number of times and do a $-1$ surgery on that. This takes our meridional slope to $[a_0, a_1, \cdots a_{n-1}]$. Next we stabilize $|a_n|$ number of times and finally a $+1$ surgery to get to the desired meridional and dividing slope. A Legendrian divide (Legendrian pushoff) of this neighborhood is the component of the non-loose Hopf link with $\tb=1+\frac{p''}{q''}$. Clearly this component is non-loose as $-1$ surgery on $K_1$ produces a tight manifold. Like the previous case, we can find the other component which is loose. There are $|a_0+1|a_1+1|\cdots a_{n-1}+1||a_n-1|k_2$ candidates as desired.

\subsubsection*{$k_1>1, k_2\geq 1$} We will approach exactly like before to reach dividing slope $s$ and meridional slope $[a_0, a_1,\cdots a_{n-1}]$. Now we do $|a_n+1|$ stabilizations to get dividing slope $-\frac {p}{q}.$ Next doing $k_1-1$ Legendrian surgery will bring us to meridional slope $[a_0, a_1, \cdots a_{n}-k_1+1]$. Finally another stabilization and a $+1$ surgery give us the desired neighborhood. Note that, like before a Legendrian divide on this neighborhood is the non-loose $K_1$. There are $2k_2|a_0+1||a_1+1|\cdots |a_n|$ non-loose representatives as desired.

\subsection*{ large $\cup$ large slope} \label{subsec:large_large_algo} Note that, like Case~\ref{subsec:small_large_algo} here too we have an inconsistent twisting and thus we need to apply the $k_2+1$ fold Rolfsen twist trick. We also see that depending on $n=1$ and $n>1$ there will be cases significantly different as we saw in the proof of Theorem \ref{thm:general}. The surgery pictures for each of the cases are given in Figure ~\ref{fig:Loose_General1}, ~\ref{fig:Loose_General2}, ~\ref{fig:Loose_General3} and ~\ref{fig:Loose_General4}. One could do a similar analysis as the previous cases about the slopes. Thus we omit the details here.

\subsection{$k_2=0$} We subdivide this case into various subcases. Note that, we are applying $1$ Rolfsen twist to the slopes here and thus $-\frac{p}{q}=[a_0,a_1,\cdots, a_n]$ changes into $[-1, a_0-1, a_1,\cdots, a_n]$.

\subsubsection{$k_1=0=k_2$} 
\begin{enumerate}
    \item $n=0$ In this case, we will have a unique non-loose representative in $L(a_0,1)$ with both components non-loose. The surgery diagram is given in 
\cite{chatterjee2025_Hopf}. Just replace $p$ by $a_0$.

\item $n=1$. We start with a standard unknot as before with $-1$ dividing slope and $\infty$ meridional slope. We do $|a_0-2|$ $+1$ surgeries first. That take us to meridional slope $[-1, a_0-2]$. Now we do $1$ stabilization on a Legendrian push-off of the last surgery curve that changes the dividing slope to $[-1, a_0-1]$. And finally we do a sequence of $+1$ surgeries of length $|a_1-1|$ which changes the meridional slope to $[-1,a_0-1, a_1]$ which is our required slope (Note that, we applied a Rofsen twist here). We have $2$ such candidates as shown in Figure~\ref{fig:Loose_General1} (a).

\item $n=2$. In this case, first we do $|a_0-1|$ $+1$ surgeries to the standard unknot that changes the meridional slope to $[-1, a_0-1]$, then $|a_1|$ stabilization of it's push-off and finally end with another $|a_2|,$ $+1$ surgeries.  This gives us $|a_1-1|$ candidate as shown in Figure~\ref{fig:Loose_General1} (b).

\item $n\geq 3$. For the general case, we do a sequence of stabilizations and surgeries as shown in (c) of Figure~\ref{fig:Loose_General1}. These are all $|a_1||a_2+1|\cdots|a_{n-2}+1||a_{n-1}|$ representatives.

\end{enumerate}

%%%%%%%%%%%%%%%%%%%%%%%%%%%%%%%%%%%%%%%%%%%%%%%%%%%%%%%%%%%

\subsubsection{$k_1=1, k_2=0$}
\begin{enumerate}
    \item $n=0$. This is the $L(a_0,1)$ case. Note that, this particular contact surgery diagram was missing in \cite{chatterjee2025_Hopf}. So we add it here. See (a) of Figure~\ref{fig:Loose_General2}. We have $2$ such representatives.
   \item $n=1$. We have $|a_1-2|$ cases which are coming from the $|a_1-1|$ choice of stabilizations. See (b) of Figure~\ref{fig:Loose_General2}.
    \item $n\geq 2$ This is the general case and we have $|a_1||a_2+1|\cdots |a_{n-1}+1||a_n-1|$ representatives coming from the choices of stabilizations in (c) of Figure~\ref{fig:Loose_General2}.
\end{enumerate}

\subsubsection{$k_1\geq 2, k_2=0$} (d) of Figure~\ref{fig:Loose_General2} gives all $2|a_1||a_2+1|\cdots |a_n|$ representatives.
%%%%%%%%%%%%%%%%%%%%%%%%%%%%%%%%%%%%%%%%%%%%%%%%%%%%%%%%%%%%%%%%%%%%%%

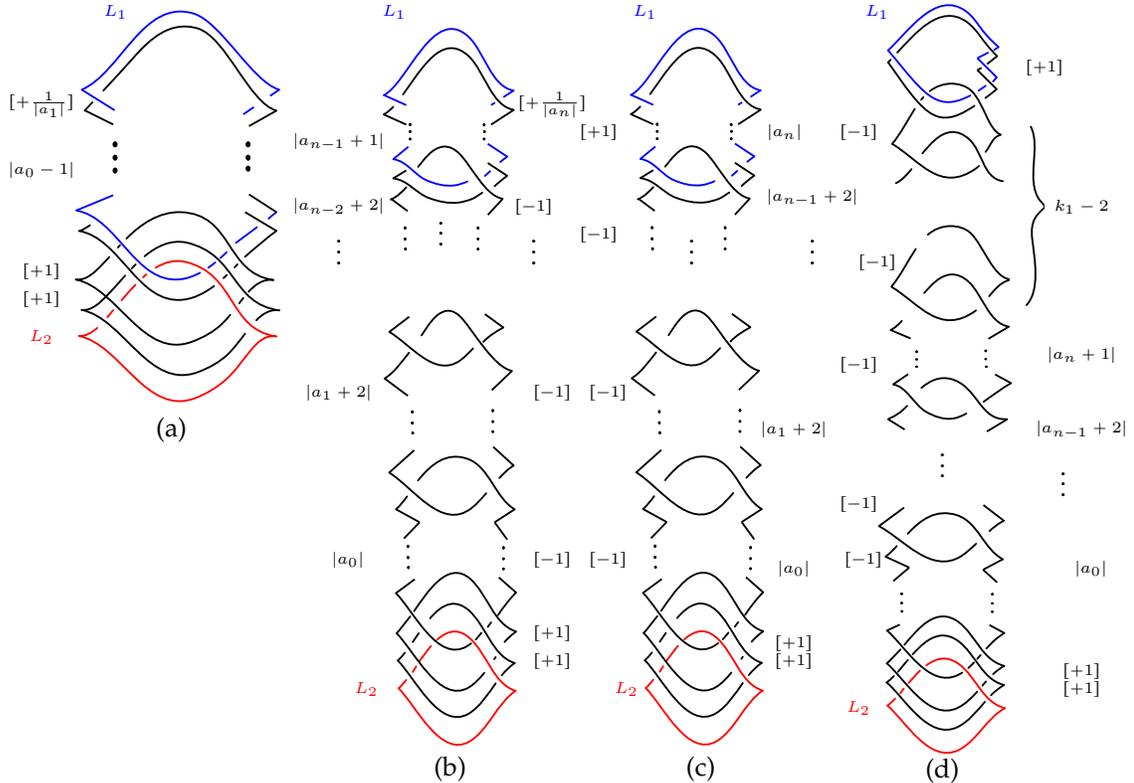
\begin{figure}[!htbp]
    \centering
    \labellist
    \pinlabel (a) at 140 350
    %%%%%%%%%%%%%%%%%%%%%%%%%%%%%
    \pinlabel (b) at 440 -15
     \pinlabel $\vdots$ at 430 570
     \pinlabel $\vdots$ at 320 550
     \pinlabel $\vdots$ at 530 550
     %%%%%%%%%%%%%%%%%%%%%%%%%%%%%%%%%(c)
     \pinlabel (c) at 710 -15
     \pinlabel $\vdots$ at 700 550
     \pinlabel $\vdots$ at 830 550
     %%%%%%%%%%%%%%%%%%%%%%%%%%%%%%%%%(d)
     \pinlabel (d) at 970 -18
     \pinlabel $\vdots$ at 1100 300
     \pinlabel $\vdots$ at 970 320
    \tiny
    %%%%%%%%%%%%%%%%%%%%%%%%%%%%%%%(a)
    \pinlabel $\textcolor{blue}{L_1}$ at 80 800
    \pinlabel $[+\frac{1}{|a_1|}]$ at 0 700
    \pinlabel $|a_0-1|$ at 0 630
    \pinlabel $[+1]$ at 0 520
    \pinlabel $[+1]$ at 0 490
    \pinlabel $\textcolor{red}{L_2}$ at 0 450
    %%%%%%%%%%%%%%%%%%%%%%%%%%%%%%%%%%%%%%%%%%%%%%%%%%%%%%%%(b)
    \pinlabel $\textcolor{blue}{L_1}$ at 380 800
    \pinlabel $[+\frac{1}{|a_n|}]$ at 550 700
    \pinlabel $|a_{n-1}+1|$ at 320 660
    \pinlabel $[-1]$ at 530 590
    
    \pinlabel $|a_{n-2}+2|$ at 320 590
    \pinlabel $|a_1+2|$ at 320 390
    \pinlabel $[-1]$ at 550 390
    \pinlabel $|a_0|$ at 330 210
    \pinlabel $[-1]$ at 550 210
    \pinlabel $[+1]$ at 550 130
    \pinlabel $[+1]$ at 550 100
    \pinlabel $\textcolor{red}{L_2}$ at 350 70
    %%%%%%%%%%%%%%%%%%%%%%%%%%%%%%%%%%%%%% (c)
    \pinlabel $\textcolor{blue}{L_1}$ at 650 800
    \pinlabel $[+1]$ at 600 670
    \pinlabel $|a_n|$ at 800 670
    \pinlabel $|a_{n-1}+2|$ at 830 600
    \pinlabel $[-1]$ at 600 560
    \pinlabel $|a_1+2|$ at 810 350
    \pinlabel $|a_0|$ at 810 200
    \pinlabel $[+1]$ at 810 120
    \pinlabel $[+1]$ at 810 100
    \pinlabel $[-1]$ at 610 210
     \pinlabel $[-1]$ at 610 390
    \pinlabel $\textcolor{red}{L_2}$ at 630 70
    %%%%%%%%%%%%%%%%%%%%%%%%%%%%%%%%%%%%%%%%%%%%(d)
      \pinlabel $\textcolor{blue}{L_1}$ at 900 800
    \pinlabel $[+1]$ at 1080 740
    \pinlabel $[-1]$ at 880 670
    \pinlabel $[-1]$ at 880 270
    \pinlabel $[-1]$ at 880 420
     \pinlabel $[-1]$ at 885 210
    \pinlabel $k_1-2$ at 1120 590
    \pinlabel $|a_{n}+1|$ at 1120 430
    \pinlabel $[-1]$ at 900 530
    \pinlabel $|a_{n-1}+2|$ at 1120 350
    \pinlabel $|a_0|$ at 1130 200
    \pinlabel $[+1]$ at 1120 90
    \pinlabel $[+1]$ at 1120 70
    \pinlabel $\textcolor{red}{L_2}$ at 880 50

    \endlabellist
    \includegraphics[scale=0.35]{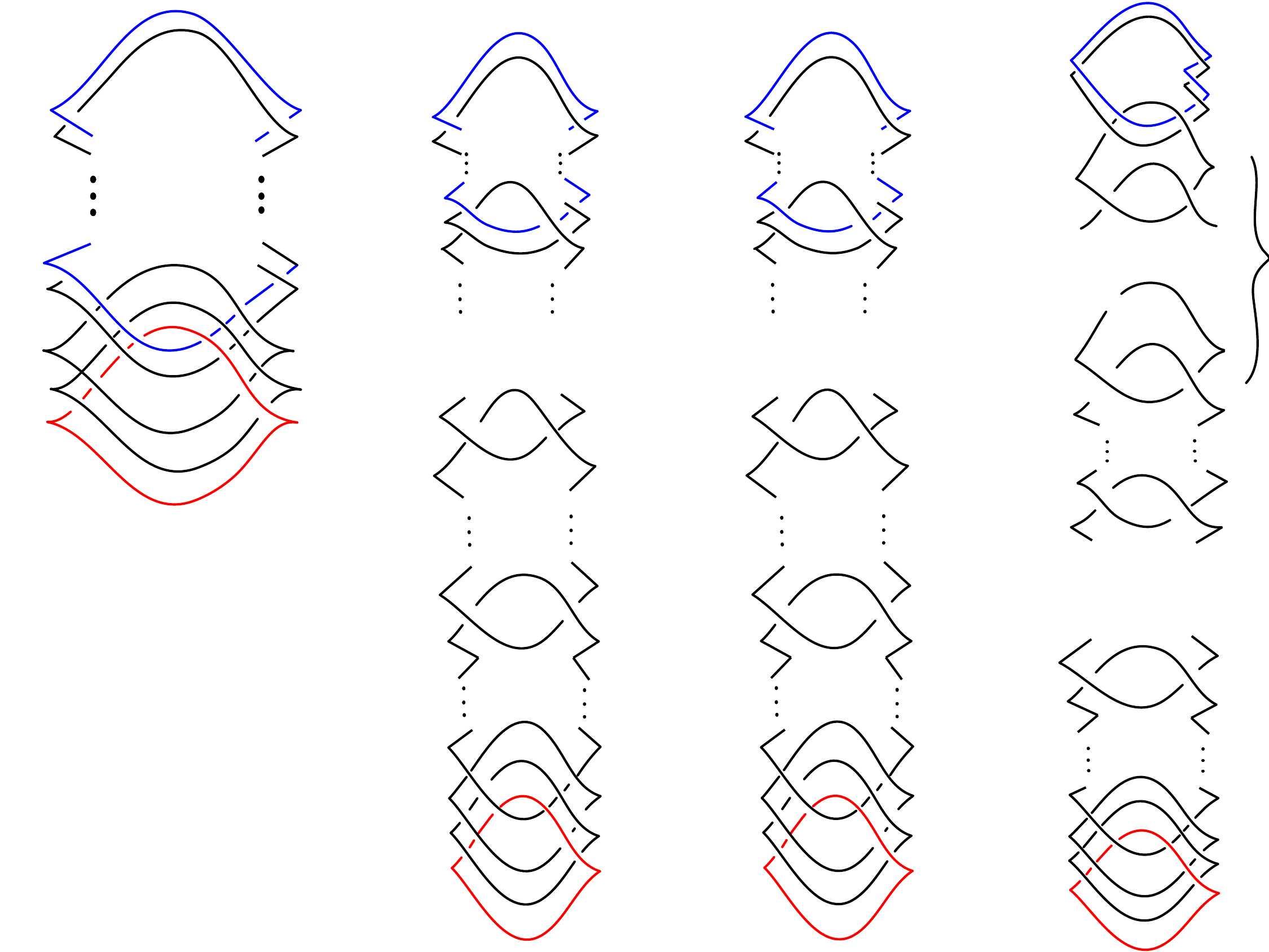}
\caption{The non-loose Hopf links in $L(p,q)=[a_0,a_1,\cdots, a_n]$ where $L_0,L_1$ are loose with $\tbr(L_1)=k_1+\frac{p'}{p}$ and $\tbr(L_2)=1+\frac{q}{p}$.\\ (a) $ k_1=0, \ n=1.$ (b) $ k_1=0, \ n\geq 2$. (c) $k_1=1$,\ $n\geq 0.$ (d) $k_1\geq 2,\ n\geq 0$.}
    \label{fig:Loose_General3}
\end{figure}

%%%%%%%%%%%%%%%%%%%%%%%%%%%%%%%%%%%%%%%%%%%%%%%%%%%%%%%%%%%%%%%%%%%%%%%%%%%%%%%%%%%%%%%
\begin{figure}[!htbp]
    \centering
    \labellist
    \pinlabel (a) at 125 -15
    \pinlabel $\vdots$ at 200 500
    \pinlabel $\vdots$ at 125 500
    %%%%%%%%%%%%%%%%%%%%%%%%%%%%%%%%%%%
       \pinlabel (b) at 515 -15
     \pinlabel $\vdots$ at 600 500
    \pinlabel $\vdots$ at 515 500
    %%%%%%%%%%%%%%%%%%%%%%%%%%%%%%%%%%
      \pinlabel (c) at 895 -15
    \pinlabel $\vdots$ at 990 450
    \pinlabel $\vdots$ at 895 440
    \tiny
    %%%%%%%%%%%%%%%%%%%%%%%%%%%%%%%%%%%%%%%(a)
        \pinlabel $\textcolor{blue}{L_2}$ at 70 780
    \pinlabel $[\frac{1}{|a_n|}]$ at 50 750 
    \pinlabel $|a_{n-1}+1|$ at 200 750
    \pinlabel $[-1]$ at 50 660
     \pinlabel $|a_{n-2}+2|$ at 200 660
    \pinlabel $|a_{n-3}+2|$ at 200 580
    \pinlabel $[-1]$ at 50 580
     \pinlabel $|a_1+2|$ at 200 360
     \pinlabel $|a_0+1|$ at 200 280
     \pinlabel $k_2-2$ at 215 180
     \pinlabel $[-1]$ at 50 360
     \pinlabel $[-1]$ at 50 280
     \pinlabel $[-1]$ at 50 230
      \pinlabel $[-1]$ at 50 120
       \pinlabel $[-1]$ at 50 150
       \pinlabel $[-1]$ at 50 80
        \pinlabel $[+1]$ at 50 40
         \pinlabel $[+1]$ at 50 60
         \pinlabel $\textcolor{red}{L_2}$ at 50 20
         %%%%%%%%%%%%%%%%%%%%%%%%%%%%%%%%%%%%%%%%%%%%%%%%(b)
          \pinlabel $\textcolor{blue}{L_2}$ at 470 780
    \pinlabel $[+1]$ at 450 750 
    \pinlabel $|a_{n}|$ at 600 750
    \pinlabel $[-1]$ at 450 660
     \pinlabel $|a_{n-1}+2|$ at 600 680
    \pinlabel $|a_{n-2}+2|$ at 600 590
    \pinlabel $[-1]$ at 450 580
     \pinlabel $|a_1+2|$ at 600 360
     \pinlabel $|a_0+1|$ at 600 280
     \pinlabel $k_2-2$ at 610 185
     \pinlabel $[-1]$ at 450 360
     \pinlabel $[-1]$ at 450 280
     \pinlabel $[-1]$ at 450 230
      \pinlabel $[-1]$ at 450 120
       \pinlabel $[-1]$ at 450 150
       \pinlabel $[-1]$ at 450 85
        \pinlabel $[+1]$ at 450 45
         \pinlabel $[+1]$ at 450 65
         \pinlabel $\textcolor{red}{L_2}$ at 450 20
         %%%%%%%%%%%%%%%%%%%%%%%%%%%%%%%%%%%%%%%%%%%%%%%%%%%(c)
              \pinlabel $\textcolor{blue}{L_2}$ at 850 800
               \pinlabel $[+1]$ at 820 770 
               \pinlabel $k_2-2$ at 990 680
               \pinlabel $|a_n+1|$ at 990 580
               \pinlabel $|a_{n-1}+2|$ at 990 500
                \pinlabel $|a_1+2|$ at 990 360
               \pinlabel $|a_0+1|$ at 990 280
                 \pinlabel $[-1]$ at 820 745 
                  \pinlabel $[-1]$ at 820 720
                   \pinlabel $[-1]$ at 820 630
                   \pinlabel $[-1]$ at 820 600
                   \pinlabel $[-1]$ at 820 550
                    \pinlabel $[-1]$ at 820 500
                     \pinlabel $[-1]$ at 820 450
                    \pinlabel $[-1]$ at 820 370
                    \pinlabel $[-1]$ at 820 300
                    \pinlabel $[-1]$ at 820 235
                    \pinlabel $[-1]$ at 820 150
                    \pinlabel $[-1]$ at 820 130
                  \pinlabel $k_2-2$ at 990 185
                    \pinlabel $[-1]$ at 820 90
                    \pinlabel $[+1]$ at 820 60
                    \pinlabel $[+1]$ at 820 45
                    \pinlabel $\textcolor{red}{L_2}$ at 820 30
    \endlabellist
    \includegraphics[scale= 0.4]{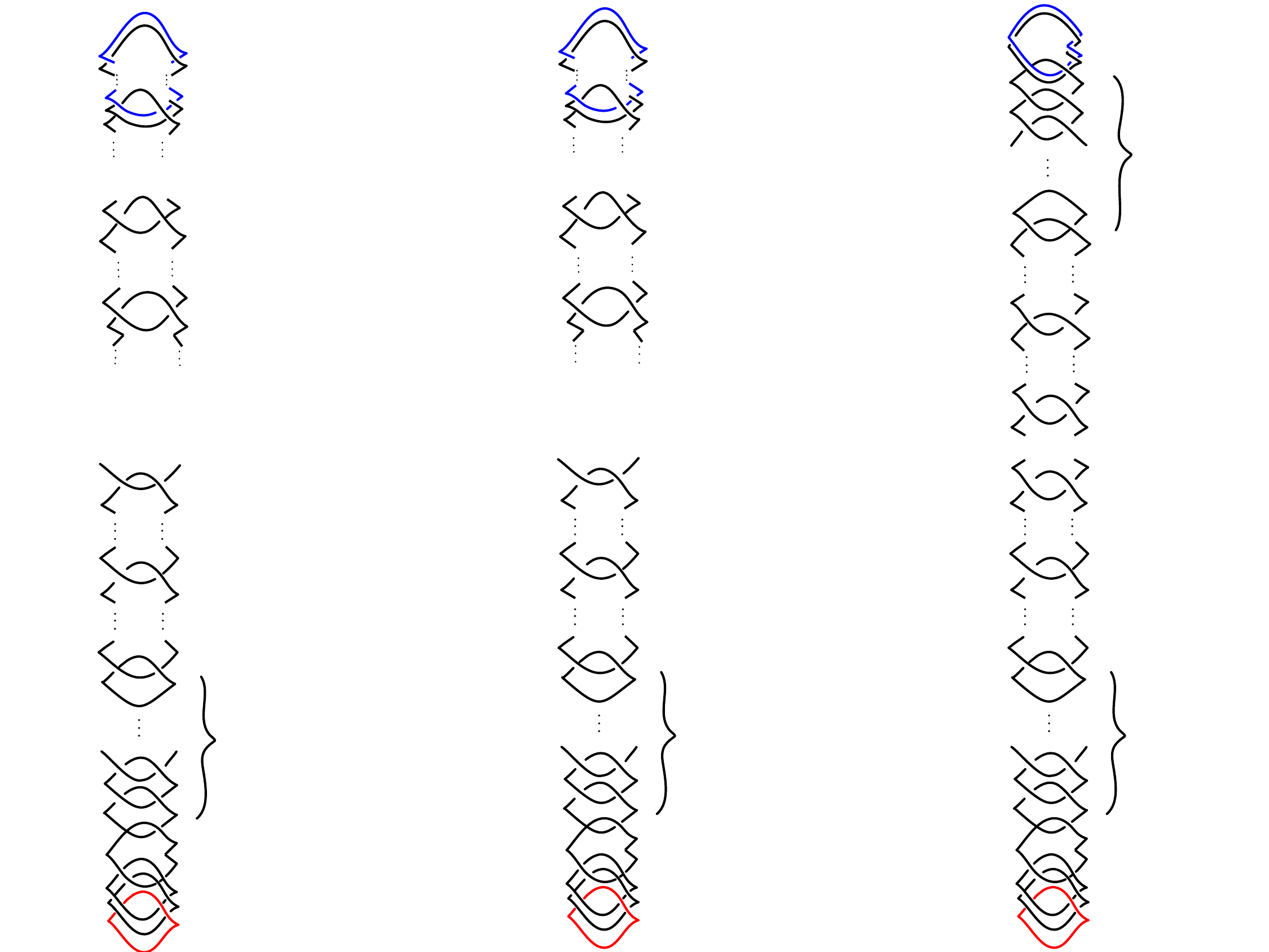}
    \caption{Non-loose Hopf links in $L(p,q)$ with both $L_1, L_2$ loose. \\$\tbr(L_1)=k_1+\frac{p'}{p}$ and $\tbr(L_2)=k_2+\frac{q}{p}$ with $k_2\geq 2$. (a) $k_1=0$, (b) $k_1=1$ and (c) $k_1\geq 2.$ In all the cases $n\geq 0.$}
    \label{fig:Loose_General4}
\end{figure}

%%%%%%%%%%%%%%%%%%%%%%%%%%%%%%%%%%%%%%%%%%%%%%%%%%%%

%%%%%%%%%%%%%%%%%%%%%%%%%%%%%%%%%%%%%%%%%%%%%%%%%%%%%%
\subsection{$k_2=1$} For this case, we have the following subcases. 

\subsubsection{$k_1= 0, k_2=1$}
\begin{enumerate}
    \item $n=1$ This is (a) of Figure \ref{fig:Loose_General3}. There are $|a_0-2|$ representatives.
    \item $n\geq 2$ For general $n$ we have $|a_0-1||a_1+1|\cdots|a_{n-1}|$ non-loose representatives as shown in (b) of Figure~\ref{fig:Loose_General3}.
\end{enumerate}

\subsubsection{$k_1=k_2=1$} Check (c) of Figure~\ref{fig:Loose_General3} for the $|a_0-1||a_1+1|\cdots|a_n-1|$ representatives.
\subsubsection {$k_1\geq 2, k_2=1$} (d) of Figure~\ref{fig:Loose_General3} gives $2|a_0-1||a_1+1|\cdots|a_{n-1}+1||a_n|$ representatives.

\subsection{$k_2\geq 2$}For the final case, we have the three subcases. The Legendrian representaives are given in Figure~\ref{fig:Loose_General4}. With these we get all the explicit Legendrian realizations of non-loose Hopf links in $L(p,q)$.
To check that all these diagrams are indeed Hopf links one needs to do a sequence of handle slides and Kirby moves. We do not include it here are they are repetitive sequences. 

%\subsubsection{$k_1=0, k_2\geq 2$}
%\subsubsection{$k_1=1, k_2\geq 2$}
%\subsubsection{$k_1\geq 2, k_2\geq 2$}

\newpage

\bibliographystyle{alpha}
\bibliography{references}

\end{document}